\documentclass[twoside,11pt]{article}
\usepackage[preprint]{jmlr2e}
\usepackage{amsmath}
\usepackage{hyperref}

\hypersetup{
    colorlinks=true,
    citecolor=blue,
    linkcolor=red,
    filecolor=magenta,      
    urlcolor=cyan
}

\usepackage{cleveref}
\usepackage{todonotes}

\newtheorem{assumption}{Assumption}
\newtheorem{theorem}{Theorem}
\newtheorem{lemma}{Lemma}
\newtheorem{proposition}{Proposition} 
\newtheorem{remark}{Remark}
\newtheorem{corollary}{Corollary}

\usepackage{xcolor}
\usepackage{enumitem}
\usepackage{constants}
\usepackage{mathrsfs} 
\usepackage{bm}
\usepackage{bbm}
\usepackage{subcaption}
\usepackage{lastpage}
\usepackage[skins]{tcolorbox}

\DeclareMathOperator{\E}{\mathbb{E}}
\DeclareMathOperator{\PP}{\mathbb{P}}
\DeclareMathOperator{\iid}{\overset{\mathrm{iid}}{\sim}}

\DeclareMathOperator*{\argmin}{arg\,min}
\DeclareMathOperator*{\lip}{lip}
\DeclareMathOperator*{\diag}{diag}

\DeclareMathOperator{\Cauchy}{Cauchy}
\DeclareMathOperator{\tr}{tr}
\DeclareMathOperator{\opt}{opt}
\DeclareMathOperator{\prox}{prox}
\DeclareMathOperator{\sign}{sign}
\DeclareMathOperator{\dist}{dist}
\DeclareMathOperator{\N}{\mathcal{N}}

\newcommand{\diff}{\mathrm{d}}

\newcommand{\indep}{\perp \!\!\! \perp}
\newcommand{\by}{\bm y}
\newcommand{\bX}{\bm X}
\newcommand{\bbeta}{\bm \beta}
\newcommand{\bSigma}{\bm \Sigma}
\newcommand{\beps}{\bm \epsilon}
\newcommand{\R}{\mathbb{R}}
\newcommand{\bx}{\bm x}
\newcommand{\bh}{\bm h}
\newcommand{\bpsi}{\bm\psi}
\newcommand{\bV}{\bm V}
\newcommand{\bG}{\bm G}
\newcommand{\br}{\bm r}
\newcommand{\PPox}{\text{prox}}
\newcommand{\tdist}{\text{t-dist}}
\newcommand{\df}{\text{df}}
\newcommand{\bbh}{\bm h}
\newcommand{\hbbeta}{\hat{\bbeta}}
\newcommand{\bv}{\bm{v}}
\newcommand{\bu}{\bm{u}}
\newcommand{\bg}{\bm{g}}

\newcommand{\bepsilon}{\bm{\epsilon}}

\newenvironment{edited}{}{}
\newcommand\tk[1]{#1}
\newcommand\editline[1]{#1}
\newcommand\none[1]{{}}

\jmlrheading{26}{2025}{1-\pageref{LastPage}}{1/24; Revised
12/24}{1/25}{24-0060}{Pierre C. Bellec and Takuya Koriyama}
\ShortHeadings{Error estimation and adaptive tuning for unregularized robust M-estimator}{Bellec and Koriyama}

\begin{document}
\title{Error estimation and adaptive tuning for unregularized robust M-estimator}

\author{\name Pierre C. Bellec \email pierre.bellec@rutgers.edu \\
       \addr Department of Statistics\\
       Rutgers University\\
       Piscataway, NJ 08854, USA
       \AND
       \name Takuya Koriyama \email tkoriyam@uchicago.edu \\
       \addr \editline{Booth School of Business\\
        The University of Chicago\\
       Chicago, IL 60637, USA} }

       \editor{Benjamin Guedj}
\maketitle 

\begin{abstract}
    We consider unregularized robust M-estimators for linear models under Gaussian design and heavy-tailed noise, in the proportional asymptotics regime where
  the sample size $n$ and the number of features $p$ are both increasing such that $p/n \to \gamma\in (0,1)$. 
  An estimator of the out-of-sample error of a robust M-estimator
      is analyzed and proved to be consistent for a large family
      of loss functions that includes the Huber loss.
  As an application of this result, 
  we propose an adaptive tuning procedure of the scale parameter $\lambda>0$
  of a given loss function \( \rho \):
  choosing \( \hat \lambda \) in a given interval \( I \) that minimizes
  the out-of-sample error estimate of the M-estimator constructed with loss
  $\rho_\lambda(\cdot) = \lambda^2 \rho(\cdot/\lambda)$
  leads to the optimal out-of-sample error over \( I \). 
  The proof relies on a smoothing argument: the unregularized M-estimation
  objective function
  is perturbed, or smoothed, with a Ridge penalty that vanishes as \( n\to+\infty \),
  and shows that the unregularized M-estimator of interest
  inherits properties of its smoothed version.
  \end{abstract}%

\begin{keywords}
  Robust regression,
  proportional regime,
  Huber loss, adaptive tuning, 
  high-dimensional statistics.
\end{keywords}

\section{Introduction}\label{sec:intro}
\editline{
    Robust statistics originated with the foundational work of \cite{huber1964robust}, which introduced methods for estimating a location parameter under heavy-tailed noise assumptions. Over time, it has evolved into a crucial area of study, providing tools to address the challenges posed by outliers and non-standard error distributions. The practical applications of robust statistics span diverse fields, including financial modeling (\cite{lambert2011robust}) and genomic data analysis (\cite{sun2020adaptive}), underscoring its broad utility across scientific and applied domains. For a comprehensive overview of the theoretical foundations and their practical implementations in data analysis, see \cite{loh2024theoretical}, \cite{maronna2019robust}, and references therein.
}

\editline{In this paper, we} consider the linear model
$\by = \bX \bbeta^\star + \beps$
where the design matrix $\bX \in \R^{n\times p}$ has i.i.d rows $\bx_i \sim \mathcal{N}(\bm 0_p, \bSigma)$ and the noise $\bm \epsilon\in \R^n$ has an i.i.d. marginal $F_\epsilon$.
Our assumption (\Cref{as:noise} below) allows
$F_\epsilon$ to be heavy-tailed, including distributions
with no finite moments.
We focus on the high-dimensional regime where the sample size $n$ and the dimension $p$ are both increasing such that \editline{$p/n \to \gamma\in (0,1)$}.
In this setting, we consider the unregularized robust M-estimator
\begin{align}\label{eq:intro_est}
   \hat{\bbeta}(\bm{y}, \bm{X})  \in \argmin_{\bbeta\in\R^p} \frac{1}{n}\sum_{i=1}^n \rho(y_i-\bx_i^\top \bbeta),
\end{align}
where $\rho: \R \to \R$ is a convex and differentiable loss such that its derivative $\psi=\rho':\R\to\R$ is bounded and Lipschitz. \editline{A} well-studied example is the Huber loss, which is defined by
\begin{equation}
    \label{huber-loss}
\forall x\in \R, \quad \rho(x) = \int_0^{|x|} \min(1, u)\diff u = 
\left\{\begin{array}{ll}
    x^2/2 & |x| \le 1\\
    |x|-1/2 & |x| \ge 1
\end{array}\right..
\end{equation}
Our goal is to select a robust loss achieving a small out-of-sample error in a data-driven manner, i.e., by only looking at the observation $(\bm y, \bm X)$. 
 Here, the out-of-sample error refers to the random quantity
\begin{align}\label{eq:intro_ofs}
  \editline{R:=} \bigl\|\bSigma^{1/2} (\hat{\bbeta}(\bm{y}, \bm{X}) -\bbeta^\star)\bigr\|_2^2 = \E\Bigl[\bigl\{\bx_{0}^\top(\hat{\bbeta}-\bbeta^\star)\bigr\}^2| \bX, \by \Bigr],
\end{align}
where $\bx_0$ is independent of $(\by, \bX)$ and has the same law as any row $\bx_i$ of $\bX$. 
\editline{In this paper, we occasionally refer to the random quantity \eqref{eq:intro_ofs} simply as the \textit{Risk}.}
Since the calculation of the out-of-sample error requires the unobservable quantities $(\bbeta^\star, \bSigma)$,
we need an observable proxy to the out-of-sample error for such data-driven selection of loss. 

\editline{
Classical statistics for M-estimation, for instance
confidence intervals for the coefficients of $\bm\beta^\star$,
require that
$p/n\to 0$, that is, the dimension is negligible compared to
the sample size.
However, the last decade has identified ample evidence that
such classical asymptotic results fail in moderate dimension
where dimension and sample size are of the same order
\citep{donoho2016high,el2013robust}. Beyond linear models, this 
phenomenon is also
clearly exposed in logistic regression \cite{sur2018modern}
with clear departure from classical results
as soon as $n\ge 10 p$ is violated.
From the point of view of classical statistics,
a peculiar property of the proportional asymptotic regime
where $\lim p/n$ is a positive constant is that
\emph{consistency fails}, and that the error
$\|\hbbeta - \bbeta^\star\|_2^2$ does not converge to 0 in probability.
}

\subsection{Results at a glance}
Our contribution is to propose a consistent estimator of the out-of-sample error for unregularized robust M-estimators. For $\psi=\rho': \R\to\R$, let  
$\psi(\bm{r})=(\psi(r_i))_{i=1}^n \in \R^n$ for any vector $\br\in \R^n$, in other words $\psi:\R\to\R$ acts componentwise on vectors.
Then, under some regularity condition on the loss and noise distribution (see \Cref{as:loss}-\ref{as:noise}), we show
that the random quantity
\begin{equation}
    \label{eq:intro_R}
\hat R = \frac{p\|\psi(\bm{y}-\bX\hat{\bbeta})\|_2^2}{\{\tr[\bV]\}^2}
\quad
 \text{with } 
        \bV := \frac{\partial \psi(\bm{y}- \bX\hat{\bbeta})}{\partial \bm{y}} = \Bigl(\frac{\partial  \psi(y_i-\bx_i^\top\hat{\bbeta})}{\partial y_j}\Bigr)_{i,j} \in \R^{n\times n}
\end{equation}
is consistent as an estimate of the out-of-sample error, in the sense that
\begin{align}\label{eq:intro_consistency}
        \|\bSigma^{1/2}(\hat{\bbeta}-\bbeta^\star)\|_2^2 
        &= \hat R + o_P(1).
\end{align}
The matrix $\bV$ in \eqref{eq:intro_R} is the Jacobian of the map $\by\mapsto \psi(\by-\bX\hbbeta)$, and its existence will be proved with probability one
in \Cref{prop:psi_lip}. Note that the random quantity $\hat{R}$ in \eqref{eq:intro_R}
is observable since 
both of $\psi(\bm{y}-\bX\hat{\bbeta})$ and $\tr[\bV]$ 
can be computed by observation of $(\by, \bX)$ only. \editline{See \Cref{subfig:risk_estimate_intro} for simulation.}
 The formal statement of the consistency result \eqref{eq:intro_consistency}
is in \Cref{th:ofs} below.

\begin{figure}[htpb]
    \centering
    \begin{subfigure}[b]{0.49\textwidth}
        \centering
        \includegraphics*[width=\textwidth]{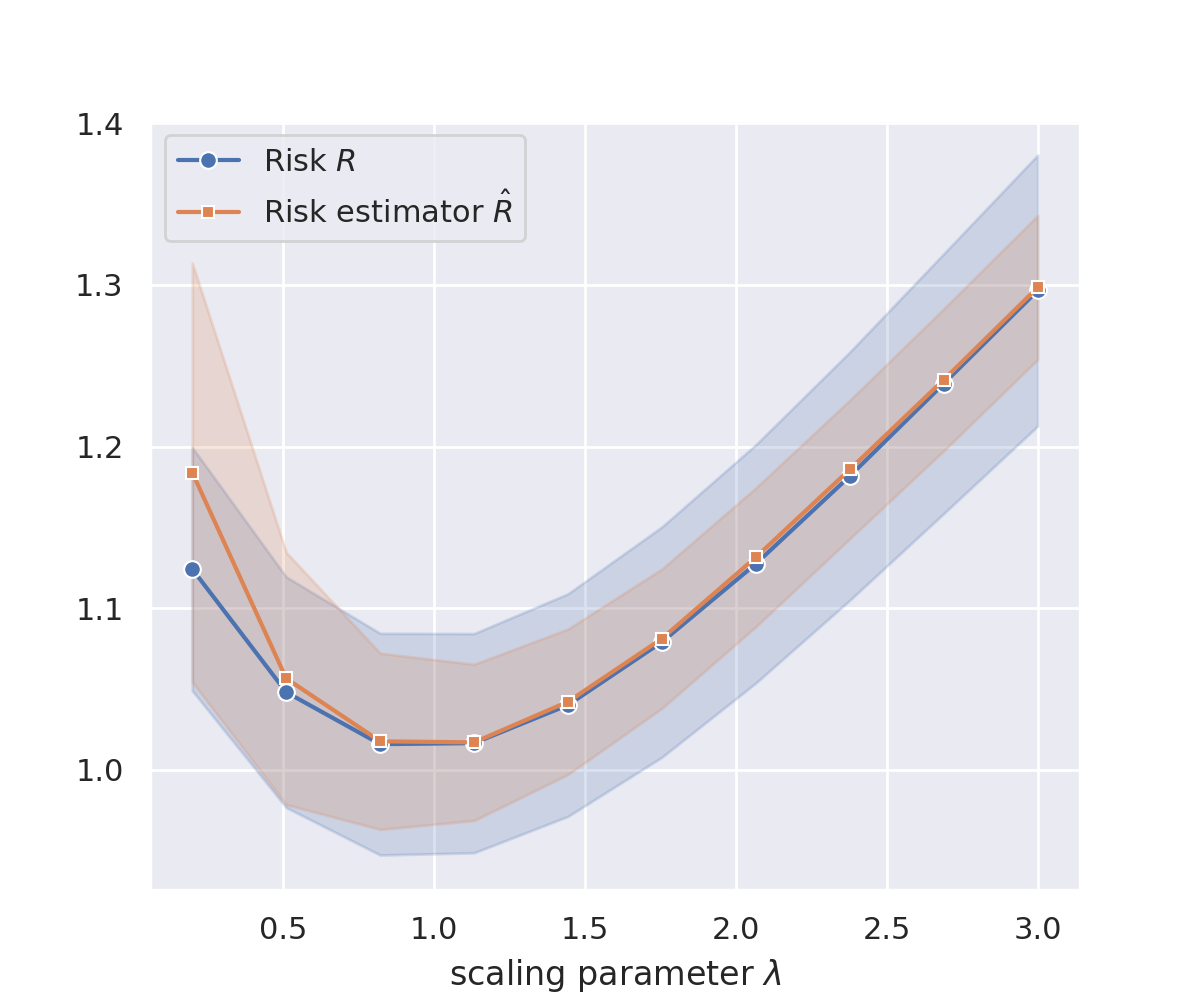}
        \caption{Estimation of out-of-sample error}
        \label{subfig:risk_estimate_intro}
    \end{subfigure}
    \begin{subfigure}[b]{0.49\textwidth}
        \centering
        \includegraphics[width=\textwidth]{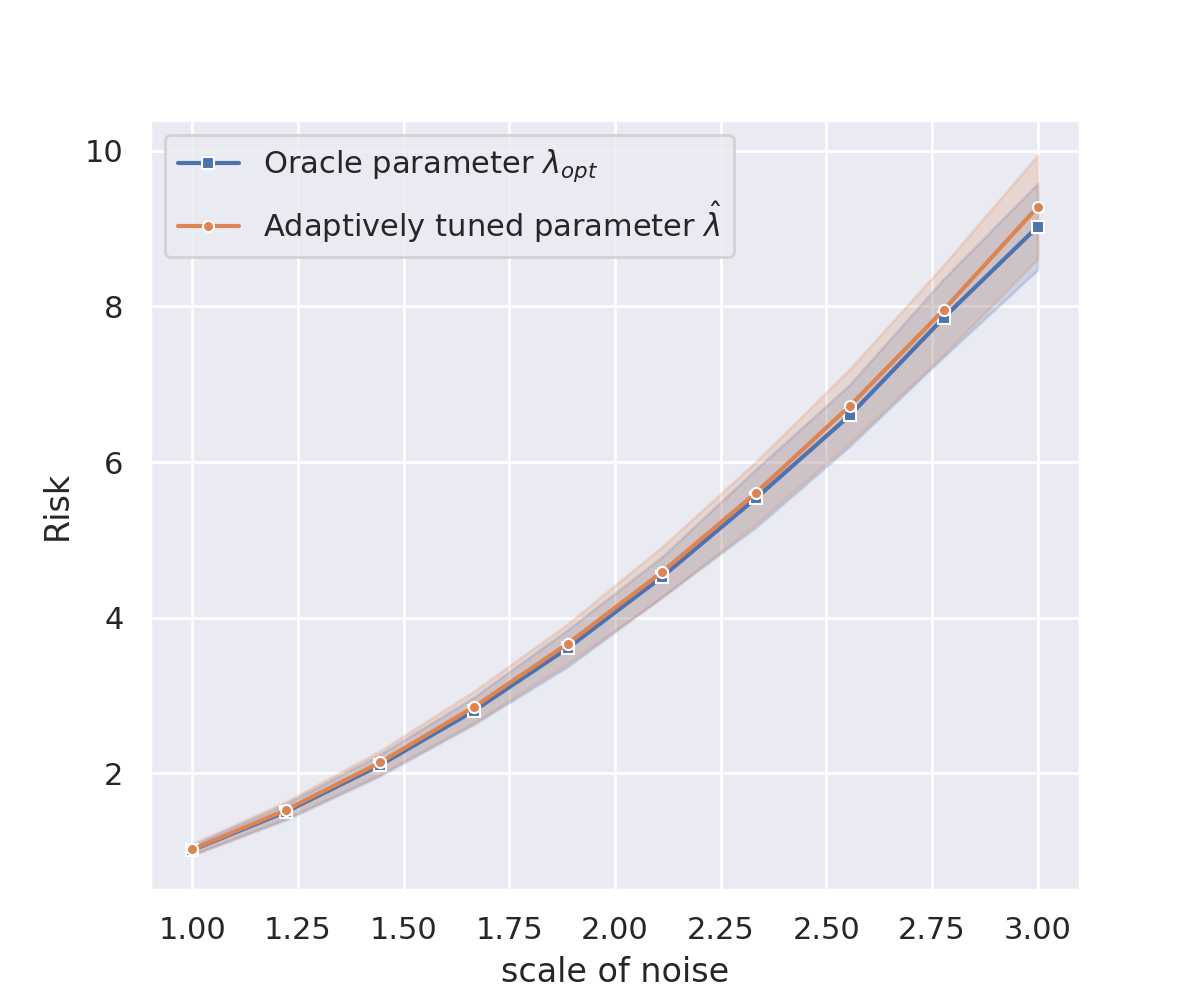}
        \caption{Adaptive tuning}
        \label{subfig:adaptive_tuning_intro}
    \end{subfigure}
    \caption{
        \editline{
        \Cref{subfig:risk_estimate_intro} is the plot of the out-of-sample error $R$ and its estimator $\hat{R}$ with the Huber loss for different scaling parameter $\lambda>0$. 
        \Cref{subfig:adaptive_tuning_intro}  is the plot of the oracle out-of-sample error $R(\lambda_{\text{opt}})$ and the out-of-sample error $R(\hat{\lambda})$ with $\hat{\lambda}$ being the minimizer of the estimator $\hat{R}$ among a finite grid $I$, as the scale of noise changes. 
        See \Cref{sec:numeric} for the details.} }
        \label{fig:intro_simulation}
\end{figure}        

\editline{
A first application of the above result in statistical inference
is the ability to choose, between two competing loss functions
$\rho$ and $\tilde \rho$, the loss function yielding the M-estimator
$\hbbeta$ enjoying the narrowest confidence interval
for a component $\beta_j^\star$ of $\bbeta^\star$.
Lemma 1 in \cite{el2013robust} establishes that
$\sqrt p (\beta_j^\star - \hat\beta_j)/\|\bSigma^{1/2}(\bbeta^\star - \hbbeta)\|$ is asymptotically normal with variance only depending on the $j$-th diagonal
element of $\bSigma^{-1}$.
Thus, the size of the confidence interval is proportional
$\|\bSigma^{1/2}(\bbeta^\star - \hbbeta)\|$, and the ability
to choose a loss function among $\{\rho,\tilde \rho\}$
leading to the smallest error
$\|\bSigma^{1/2}(\bbeta^\star - \hbbeta)\|$
through \eqref{eq:intro_consistency} will lead to the smallest
confidence interval.
}

As a second application of the above result, we propose an adaptive tuning procedure of the scale parameter $\lambda$ in a collection of loss functions $\{\rho_\lambda (\cdot) = \lambda^2\rho(\cdot/\lambda), \lambda\in I\}$, for some fixed interval $I\subset(0,\infty)$ and some robust loss $\rho$. 
Note that this hyperparameter
$\lambda$ controls the sensitivity of the loss $\rho_\lambda$ to outliers. For each $\lambda>0$, 
let $\hbbeta(\lambda)$ be the M-estimator 
\eqref{eq:intro_est} with $\rho=\rho_\lambda$
 and let $\hat{R}(\lambda)$ be the corresponding estimate of the out-of-sample error \eqref{eq:intro_R}.
Then, we show that there exists a finite subset $J\subset I$ (a finite grid made of regularly spaced parameters) such that
\begin{align}\label{eq:intro_tuning}
     \hat{\lambda} \in \argmin_{\lambda\in J} \hat{R}(\lambda) 
     \quad
     \text{ satisfies }
     \quad
     \|\bSigma^{1/2}(\hbbeta(\hat \lambda)-\bbeta^\star)\|^2
     \le
     \|\bSigma^{1/2}(\hbbeta(\lambda_{\editline{\text{opt}}})-\bbeta^\star)\|^2
     + o_P(1)
\end{align}
where $\editline{\lambda_{\text{opt}}}$ is the oracle parameter achieving the smallest
out-of-sample error in probability in the interval $I$. \editline{See \Cref{subfig:adaptive_tuning_intro} for simulation. The formal statement of the result \eqref{eq:intro_tuning} is in \Cref{th:lambda_tuning} below.}
\subsection{Related work}
Under regularity conditions on the loss $\rho$ and noise distribution $F_\epsilon$, the out-of-sample error of the unregularized M-estimator \eqref{eq:intro_est} is known 
% \cite[among others]{el2013robust,el2018impact, karoui2013asymptotic, donoho2016high, thrampoulidis2018precise}
\editline{\cite{el2013robust,el2018impact, karoui2013asymptotic, donoho2016high, thrampoulidis2018precise}}
to converge in probability to a deterministic value $\alpha^2$, which is a solution to the following nonlinear system of equations
{with positive unknowns} \( (\alpha,\kappa) \):
\begin{align}\label{eq:intro_nonlinear}
    \begin{split}
        \alpha^2\gamma  &= \E[
        ((\alpha Z + W) - \PPox[\kappa \rho] (\alpha Z+ W))^2] \\
        \alpha\gamma &= \E [((\alpha Z + W) - \PPox[\kappa \rho](\alpha Z+ W))\cdot Z]
    \end{split}
\end{align}
where $Z\sim \N(0,1)$, $W\sim F_\epsilon$, and  
 $\PPox[f] (x) := \argmin_{u\in\R} (x-u)^2/2 + f(u)$.
The convergence in probability of the out-of-sample error of \( \hbbeta \) to \( \alpha^2 \) is granted by
     \cite{donoho2016high,thrampoulidis2018precise}
     provided that \eqref{eq:intro_nonlinear} admits a unique solution.
The existence of solutions to \eqref{eq:intro_nonlinear} was established
by \cite{donoho2016high} for strongly convex losses, and in our
recent paper \cite{koriyama2023analysis} for general Lipschitz losses
on the side of the phase transition where exact recovery (i.e., \( \hbbeta=\bbeta^* \)) is not possible. On the other hand,
uniqueness is addressed in
\cite{thrampoulidis2018precise} by strict convexity arguments and
in \cite{koriyama2023analysis}.
The working assumptions of the present paper stated
in \Cref{sec:ofs} ensure that a solution \( (\alpha,\kappa) \) to \eqref{eq:intro_nonlinear}
exists and is unique, so that the convergence in probability
\( \|\bSigma^{1/2}(\hbbeta-\bbeta^*)\|\to^p \alpha \)
as \( n,p\to+\infty \) with \( p/n\to\gamma<1 \)
holds by the
main result of \cite{thrampoulidis2018precise}.

\citet{bean2013optimal}
studied the construction of the optimal loss
that minimizes the solution $\alpha$ of \eqref{eq:intro_nonlinear} when $F_\epsilon$ is log-concave and known. However, constructing the optimal loss in this way requires the knowledge of
$F_\epsilon$ to minimize the corresponding $\alpha$ in
\eqref{eq:intro_nonlinear}. Since $F_\epsilon$ is typically unknown,
this construction cannot be implemented in practice.

In this paper, we focus on the underparametrized regime $p/n\to\gamma$ with $\gamma\in(0,1)$. In the overparametrized regime with $\gamma>1$,
the out-of-sample error of the unregularized M-estimator \eqref{eq:intro_est}
explodes
(cf. \cite[Remark 5.1.1]{thrampoulidis2018precise}). 
This curse of dimensionality
can be overcome by using a penalty function $g$ and computing the corresponding
penalized M-estimator
$\hat{\bbeta}\in \argmin_{\bbeta\in \R^p}  n^{-1} \sum_{i=1}^n \rho(y_i-\bx_i^\top\bbeta) + g(\bbeta)$, for which the out-of-sample error will be finite
under suitable assumptions on the penalty $g$ 
% \cite[among others]{bayati2011lasso,thrampoulidis2018precise}.
\editline{\cite{bayati2011lasso,thrampoulidis2018precise,loureiro2021learning}.}
If $\gamma < 1$, however, some advantages of the unregularized M-estimator\eqref{eq:intro_est} include its invariance with respect to $(\bSigma,\bbeta^\star)$
\cite[Lemma 1]{el2013robust} and 
that confidence intervals for components $\beta_j^\star$ of $\bbeta^\star$
using the asymptotic normality of \( \hat\beta_j \)
do not require knowledge of $\bSigma$; 
on the other hand,
the de-debiasing correction \( \bm e_j^\top\bSigma^{-1/2}\bX^\top\psi(\by-\bX\hbbeta) \) necessary for asymptotic normality
of a regularized M-estimator with penalty \( g(\cdot) \) as above
does require knowledge (or some estimate) of $\bSigma$
\citep{bellec2022asymptotic}.

Most similar works to \eqref{eq:intro_consistency} are  \cite{bellec2023out, bellec2022derivatives,bellec2022observable}; they show that for general pairs of loss $\rho$ and penalty $g$, 
 the out-of-sample error of the penalized M-estimator $\hat{\bbeta}\in \argmin_{\bbeta\in \R^p}  n^{-1} \sum_{i=1}^n \rho(y_i-\bx_i^\top\bbeta) + g(\bbeta)$ 
enjoys the approximation
\begin{align}\label{eq:ofs_strong_convex}
    \|\bSigma^{1/2} (\hat{\bbeta}-\bbeta^\star)\|_2^2 \approx \ 
    \tr[\bV]^{-2}({\|\hat{\bpsi}\|_2^2 (2\hat{\df} - p) + \|\bSigma^{-1/2}\bX^\top \hat{\bpsi}\|_2^2}),
\end{align}
where $\hat{\bpsi} := \psi(\bm y - \bX \hat{\bbeta})\in \R^n$, \editline{$\hat{\df} := \tr[(\partial/\partial \bm{y})(\bX \hat{\bbeta})]$, and $\bV$ as in \eqref{eq:intro_R}}. 
Furthermore, these derivatives have closed forms for a certain choice of $(\rho, g)$; for instance when $\rho$ is the Huber loss and $g$ is the $L_1$ penalty, we have $\hat{\df}=|\hat{S}|$ and $\tr[\bV]=|\hat{I}|-|\hat{S}|$, where $\hat{S} = \{j\in [p] : \hat{\beta}_j \neq 0\}$ is the active set and
$\hat{I} = \{i\in [n]: | y_i - \bx_i^\top \hat\bbeta | \le 1\}$ is the set of inliers \cite[Propositions 2.2 and 2.3]{bellec2023out}. 

With the \editline{KKT} conditions of \eqref{eq:intro_est} giving
$\bX^\top \hat{\bpsi} = \bm0_p$,
we may regard \eqref{eq:intro_R} as a special case of \eqref{eq:ofs_strong_convex} by setting $\hat{\df}=p$.
However, the proof techniques in the aforementioned papers
\cite{bellec2023out, bellec2022derivatives,bellec2022observable}
rely either on the strong convexity of the penalty $g$ (excluding the $g=0$ case we study here),
or on a sufficiently sparse structure in $\bbeta^\star$ in which
case the L1 penalty also allows for approximations of the form \eqref{eq:ofs_strong_convex} \cite{celentano2020lasso,bellec2023out}.

For the unregularized case \eqref{eq:intro_est} studied here,
the results from the papers
\cite{el2013robust,el2018impact, karoui2013asymptotic, donoho2016high, thrampoulidis2018precise}
are not sufficient to establish the approximation \eqref{eq:intro_consistency},
say for the Huber loss, on the one hand because
\cite{el2013robust,el2018impact, karoui2013asymptotic, donoho2016high}
rely on strong convexity on either the loss or the penalty,
but more importantly because these results do not study the trace of the
Jacobian $\tr[\bV]$ in \eqref{eq:intro_R} and how this quantity relates
to the solution $(\alpha,\kappa)$ of the nonlinear system \eqref{eq:intro_nonlinear}.
For instance, the Convex Gaussian Min-max Theorem (CGMT) of \cite{thrampoulidis2018precise} provides the limit in probability of 
$\|\bSigma^{1/2}(\hbbeta-\bbeta^\star)\|^2$, and reversing the argument
(replacing the design matrix by its transpose) provides the limit in
probability of $\|\bpsi\|^2/n$ in \eqref{eq:intro_est}, which
equals $\alpha^2\gamma / \kappa^2$. However,
characterizing a limit in probability for the quantity $\tr[\bV]$ has
so far remained out of reach of the CGMT. The trace of the Jacobian
$\tr[\bV]$ appears in the leave-one-out analysis of \cite{el2013robust,karoui2013asymptotic,el2018impact}, however these works study the regularized
estimate
$\hat{\bbeta}\in \argmin_{\bbeta\in \R^p}  n^{-1} \sum_{i=1}^n \rho(y_i-\bx_i^\top\bbeta) + \mu\|\bbeta\|^2/2$
with an additive Ridge penalty $\mu\|\bbeta\|^2/2$.
Recently, \cite{bellec2022observable} proposed an alternative regularization
technique to study unregularized estimates of the form
\eqref{eq:intro_est}, but this crucially requires
the loss $\rho$ to be twice continuously differentiable and $\rho''(x)>0$ for all $x\in \R$, which rules out the Huber loss \eqref{huber-loss}, one of the most common robust loss functions and a major application of the present paper.

\subsection{Precise analysis of a perturbed M-estimator}
To overcome these difficulties and prove \eqref{eq:intro_consistency}
for the Huber loss and its variants, we use the ``\textit{Ridge-smoothing}'' technique, which is used in previous works
% \cite[among others]{karoui2013asymptotic,  celentano2022fundamental, loureiro2022fluctuations}.
\editline{\cite{karoui2013asymptotic,  celentano2022fundamental, loureiro2022fluctuations}.}
By rotational and translation invariance, assume without loss of generality that $\bbeta^\star = \bm{0}_p$ and $\bSigma=\bm{I}_p$. 
Then, \eqref{eq:intro_consistency} is reduced to
\begin{equation}\label{eq:intro_goal}
    \|\hat\bbeta\|_2^2 = p \|\bpsi\|_2^2 \cdot (\tr[(\partial/\partial \bm{\epsilon})\bpsi])^{-2}  
    {+o_P(1)}
\end{equation}
where $\bpsi: \bm\epsilon \in\R^n \mapsto \psi(\bm\epsilon-\bX\hat{\bbeta})\in\R^n$ is a vector field and $(\partial/\partial \bm{\epsilon})\bpsi$ is the Jacobian matrix. To show \eqref{eq:intro_goal}, we instead consider the ridge-regularized M-estimator $\hat{\bbeta}_{\text{ridge}}$ with a diminishing regularization parameter $n^{-c}$ for some constant $c>0$
\begin{equation}
    \label{eq:b_ridge_intro}
\hat{\bbeta}_{\text{ridge}} = \argmin_{\bbeta\in\R^p} \frac{1}{n}\sum_{i=1}^n \rho(\epsilon_i - \bx_i^\top\bbeta) + \frac{n^{-c}}{2}\|{\bbeta}\|_2^2,
\end{equation}
and define the vector field $\bpsi_{\text{ridge}}: \bm{\epsilon}\mapsto \psi(\bm\epsilon-\bX\hat{\bbeta}_{\text{ridge}})$. We prove the target \eqref{eq:intro_goal} by the following two steps:
\begin{enumerate}[leftmargin=6mm, label=(\Roman*), ref=(\Roman*)]
    \item \label{I} Prove \eqref{eq:intro_goal} for the regularized $\hat{\bbeta}_{\text{ridge}}$, i.e., 
    $\|\hat{\bbeta}_{\text{ridge}}\|_2^2 \approx p \|\bpsi_{\text{ridge}}\|_2^2 \cdot (\tr[(\partial/\partial \bm{\epsilon})\bpsi_{\text{ridge}}])^{-2}$. 
    \item \label{II} Prove that each quantity in \eqref{eq:intro_goal}
            is approximately the same for
        $\hat\bbeta_{\text{ridge}}$ and for $\hat\bbeta$,
        as the Ridge penalty coefficient in \eqref{eq:b_ridge_intro}
        converges to 0 polynomially in \( n \):
        \begin{equation*}
        \|\hat{\bbeta}_{\text{ridge}}\|_2^2 \approx \|\hat{\bbeta}\|_2^2, 
        \qquad \tfrac1n\|\bpsi_{\text{ridge}}\|_2^2 \approx\tfrac1n \|\bpsi\|_2^2, 
        \qquad \tfrac1n\tr[ (\partial/\partial \bm{\epsilon})\bpsi_{\text{ridge}}] \approx \tfrac1n \tr[(\partial/\partial \bm{\epsilon})\bpsi].
    \end{equation*}
\end{enumerate}
We prove \editline{\ref{I}} by a chi-square type moment inequality given in \cite[Section 7]{bellec2023out}. The more subtle and novel part of the proof is to derive
the three approximations in \editline{\ref{II}}, in particular the challenging approximation 
\begin{equation}
    \label{eq:third_approx}
    \tr[(\partial/\partial \bm{\epsilon})\bpsi_{\text{ridge}}] \approx \tr[(\partial/\partial \bm{\epsilon})\bpsi].
\end{equation}
Indeed, the closeness of the two vector fields $(\bpsi, \bpsi_{\text{ridge}})$ in the Euclidean norm does not necessarily imply the closeness of their divergence.
To show {
\eqref{eq:third_approx}, we leverage the assumption that the noise
distribution is sufficiently smooth: concretely, \Cref{as:noise} below grants
}
that the noise distribution is a convolution of two probability distributions $(F, \tilde{F})$, that is,
\begin{equation}
    \label{eq:convolution_intro}
\bm{\epsilon}=\bm{z} + \bm{\delta}, \qquad \bm{z} \indep \bm{\delta}, \qquad (z_i)_{i=1}^n \sim F, \qquad (\delta_i)_{i=1}^n \sim \tilde{F},
\end{equation}
where  $F$ has density $z \mapsto \exp(-\phi(z))$ such that $\phi:\R\to\R$ is twice continuously differentiable with bounded second derivative
{(on the other hand, \( \tilde F \) is unrestricted and may have arbitrarily
fat tails).}
Then, using a variant of
the second order Stein's formula \cite[Section 2.4]{bellec2021second} extended to the distribution with density $z\mapsto \exp(-\phi(z))$ (see \Cref{th:second} below), we will argue in \Cref{lm:diff_V} below
that
$
\E [(\tr[(\partial/\partial \bm{z})(\bpsi_{\text{ridge}} -\bpsi)] -\phi'(\bm{z})^\top (\bpsi_{\text{ridge}} -\bpsi))^2] 
$ is relatively small. As a consequence, using the chain rule for the Jacobian, we have 
\begin{equation}
    \label{eq:approx_divergence}
\tr[(\partial/\partial \bm{\epsilon})\bpsi_{\text{ridge}}] - \tr[(\partial/\partial \bm{\epsilon})\bpsi] = \tr[(\partial/\partial \bm{z})(\bpsi_{\text{ridge}} -\bpsi)] \approx \phi'(\bm{z})^\top (\bpsi_{\text{ridge}} - \bpsi).
\end{equation}
Roughly speaking,
{
\Cref{as:noise} below grants that the noise distribution is sufficiently
smooth, and thanks to this assumption,
}
closeness of the vector fields {in the form} $\|\bpsi-\bpsi_{\text{ridge}}\|^2/n=o_P(1)$ carries over to the two divergences, and implies
\begin{equation}
    \tr[(\partial/\partial \bm{\epsilon}) \bpsi]/n
    =
    \tr[(\partial/\partial \bm{\epsilon}) \bpsi_{\text{ridge}}]/n + o_P(1)
    .
    \label{conclusion_continuity_V}
\end{equation}
It is not clear at this point if
\editline{\eqref{conclusion_continuity_V}} holds without this smoothness
assumption on the noise distribution.

Adaptive tuning of the scale parameter $\lambda$ was investigated before in several papers, including \cite{loh2021scale, wang2021new} \editline{and references therein.}
%The limitation
%of these papers is that the noise variance needs to be bounded and known
% \todo{I think we cannot allow either a noise level that depends on $n$, i.e., $\epsilon=\sigma_n(z+\delta)$ unless the scaled Lasso is used (which we do not know how to analyze yet)} 
%(or at least to be estimated heuristically),
These works do not assume a proportional asymptotics regime
and do not aim for exact multiplicative constants: the resulting risk is only guaranteed to be less than $C \times (\text{optimal risk})$ for a constant $C>1$.
{
On the other hand, our goal in the present paper is to achieve the optimal
risk with multiplicative constant 1.
}
%On the other hand, our tuning method is fully data-driven, i.e., it does not require the knowledge of noise distribution, and the tuned M-estimator achieves the best possible risk in an interval of scale parameter, with $(1+\epsilon)$ multiplicative error for an arbitrary small $\epsilon>0$. 

\subsection{Organization}
%\Cref{sec:prelim} is a preliminary for sections that follow. 
In \Cref{sec:ofs}, we derive several results on the
consistency of an estimator of the out-of-sample error. In \Cref{sec:tuning}, we discuss the adaptive tuning of scale parameters. 
\Cref{sec:numeric} is devoted to numerical simulations, and \Cref{sec:proof_highlight} gives {an outline of the proof}. The rigorous proofs are provided in appendix. 

\subsection{Notation}
For any vector $\bm{u}$, we denote by $\|\bm{u}\|$ the Euclidean norm $\sqrt{\sum_i u_i^2}$. 
For any matrix $\bm X$, let $\|\bm X\|_{op}$ be the operator norm, i.e., the maximum singular value of $\bm X$. Given a vector field $\bm f: \bm{z}\in\R^n\mapsto\bm{f}(\bm{z})\in\R^m$, let $(\partial/\partial \bm{z})\bm{f}\in\R^{m\times n}$ be the Jacobian matrix, and 
 $\|\bm f\|_{\lip}$ be the Lipschitz constant of $\bm f$ induced by the Euclidean norm. 
For any function $\psi: \R \to \R$, let
 $\psi(\bm{r})=(\psi(r_i))_{i=1}^n \in \R^n$ for all $\br\in \R^n$, in other words $\psi:\R\to\R$ acts componentwise on vectors, and let $\|\psi\|_{\infty}=\sup_{x\in \R}|\psi(x)|$ be the sup norm. 
 If we write $C=\C(a, b, c)$, $C$ is a constant depending on $(a, b, c)$ only. If two random vectors $\bm x, \bm y$ are independent, we write $\bm x\indep \bm y$. 
{For a sequence of random variables \( (U_n)_{n\ge 1} \), we write
    \( U_n\to^p U \) to denote convergence in probability to \( U \)
    and \( U_n=o_P(1) \) if \( U_n\to^p 0 \).
    For a sequence of reals \( r_n>0 \), we write \( U_n=O_P(r_n) \) if
    for any \( \epsilon>0 \) there exists \( K_\epsilon \) such that
    \( \sup_{n\ge 1}\PP(|U_n|>K_{\epsilon}r_n)<\epsilon \)
}

\section{{Estimation of the out-of-sample error}}
\label{sec:ofs}
Throughout, we assume that $(y_i, \bx_i, \epsilon_i)_{i=1}^n$ are independently distributed according to
$$
\forall i \in [n], \qquad y_i = \bx_i^\top \bbeta^\star + \epsilon_i, \qquad \bx_i \sim \N(\bm{0}_p, \bSigma), \qquad \epsilon_i \sim F_\epsilon, \quad \bx_i\indep \epsilon_i
$$
where $\bbeta^\star\in\R^p$ is an unknown regression vector, $\bSigma\in\R^{p\times p}$ is some symmetric positive semi-definite matrix, and $F_\epsilon$ is a probability distribution. 
Since our interest is the out-of-sample error $\|\bSigma^{1/2}(\hat{\bbeta}-\bbeta_*)\|^2$, and the out-of-sample error  for the unregularized M-estimator $\hat{\bbeta}\in \argmin_{\bbeta\in\R^p}\sum_{i=1}^n \rho(y_i-\bx_i^\top\bbeta)$ is invariant with respect to $(\bbeta^\star, \bSigma)$ (cf. \cite{karoui2013asymptotic} or \Cref{sec:proof_highlight}), we do not require any assumptions on $(\bbeta^\star, \bSigma)$. 
% In contrast, we make some assumptions on the loss $\rho$ and noise distribution $F_\epsilon$ (see \Cref{as:loss}, \ref{as:noise}, and \ref{as:tuning}). 

When taking a limit as $n\to\infty$, we implicitly assume that the number of features $p$ and the sample size $n$ are increasing such that 
$p/n\to \gamma\in (0,1)$, 
while other quantities such $\rho$ and $F_\epsilon$ are fixed. In this setting, it has been shown by \cite{el2013robust,el2018impact, karoui2013asymptotic, donoho2016high, thrampoulidis2018precise} that if the nonlinear system of equations 
\begin{align}\label{eq:nonlinear}
    \begin{split}
        \alpha^2\gamma  &= \E[
        ((\alpha Z + W) - \PPox[\kappa \rho] (\alpha Z+ W))^2] \\
        \alpha\gamma &= \E [((\alpha Z + W) - \PPox[\kappa \rho](\alpha Z+ W))\cdot Z]
    \end{split}
    \quad \text{ where }
    \begin{cases}
    Z\sim \N(0,1),\\
    W\sim F_\epsilon,
     Z\indep W
    \end{cases}
\end{align}
admits a unique solution $(\alpha, \kappa)$, the out-of-sample error $\|\bSigma^{1/2}(\hat{\bbeta}-\bbeta_*)\|^2$ converges to $\alpha^2$ in probability. 
The existence of solutions to \eqref{eq:nonlinear} was established
in \cite{donoho2016high} for strongly convex \( \rho \), and recently
in our companion paper \cite{koriyama2023analysis}
(see \Cref{sec:details} for details).

\begin{figure}
    \centering
    \begin{subfigure}[b]{0.49\textwidth}
        \centering
        \includegraphics[width=\textwidth]{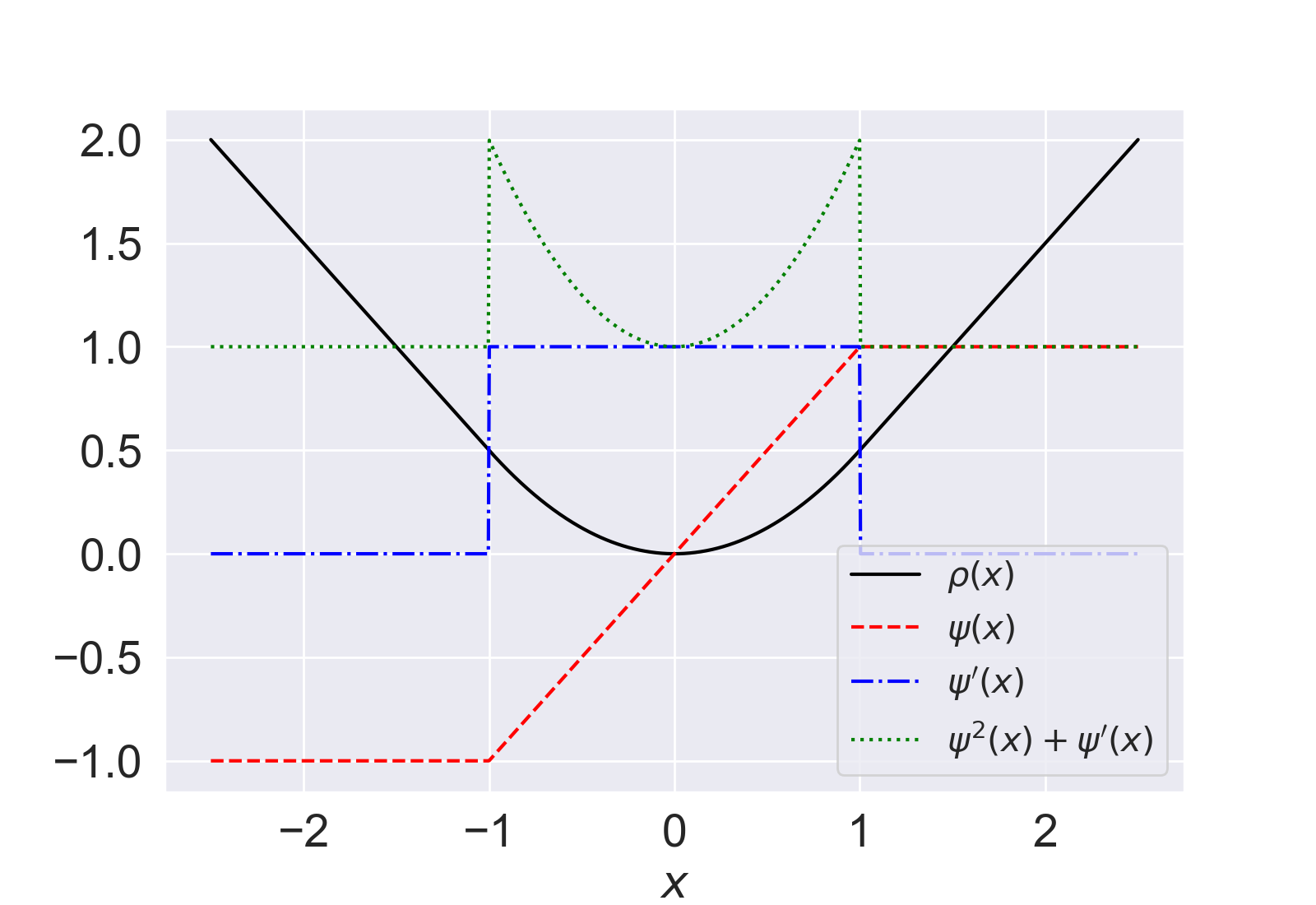}
        \caption{Huber $\rho(x)=\int_{0}^{|x|} \min(1, u)\diff u$}
        \label{subfig:huber}
    \end{subfigure}
    \begin{subfigure}[b]{0.49\textwidth}
        \centering
        \includegraphics[width=\textwidth]{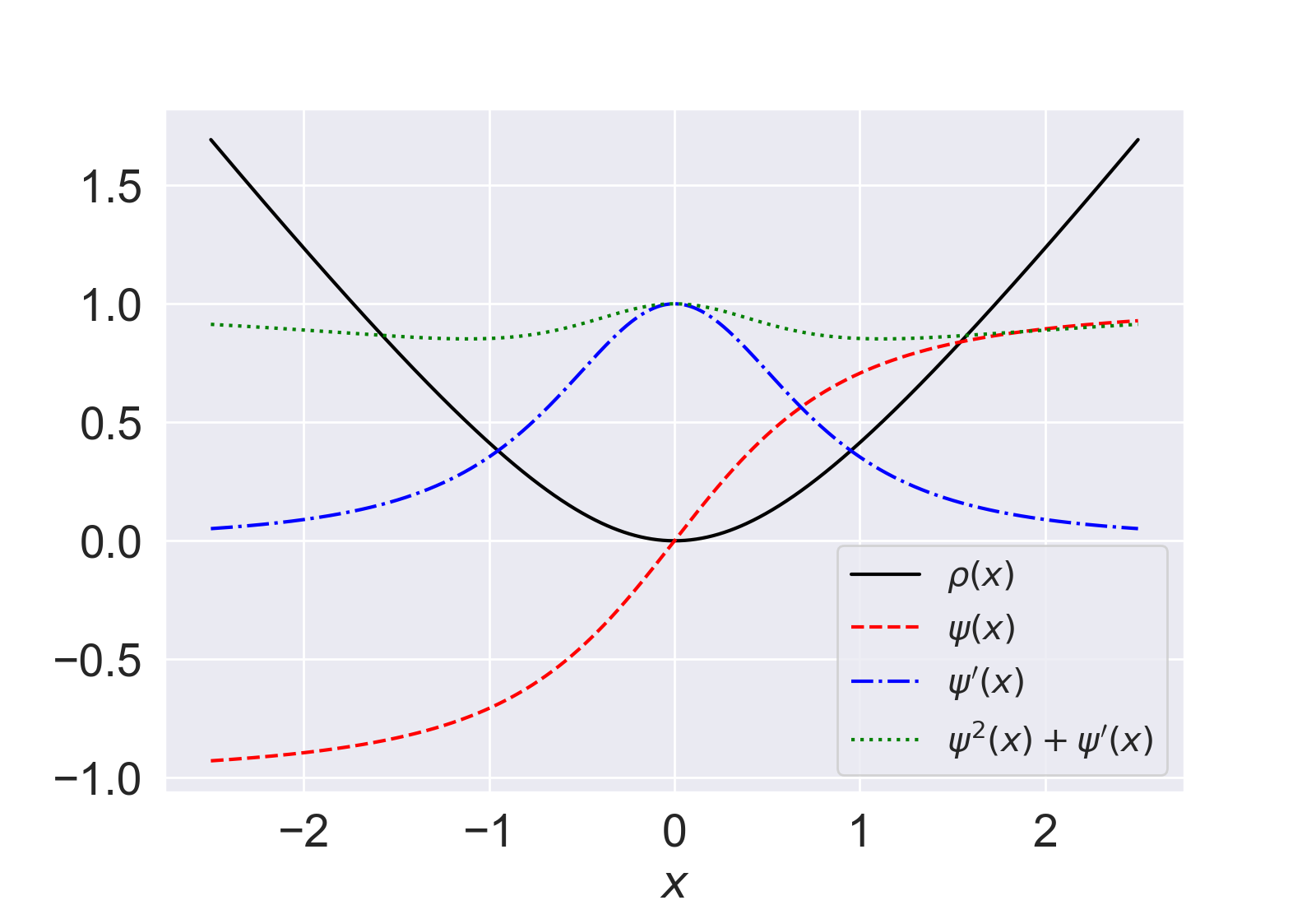}
        \caption{Pseudo Huber $\rho(x) = \sqrt{1+x^2}-1$}
        \label{subfig:pseudo_huber}
    \end{subfigure}
    \caption{Example of loss satisfying \Cref{as:loss}}
    \label{fig:loss_example}
\end{figure}
In this section, we will argue that the random quantity $\hat{R}$ defined by  
\begin{equation}\label{eq:est_ofs}
    \hat R = \frac{p\|\psi(\bm{y}-\bX\hat{\bbeta})\|_2^2}{\tr[\bV]^2}, 
     \text{ where } 
            \bV := \frac{\partial \psi(\bm{y}- \bX\hat{\bbeta})}{\partial \bm{y}}  \in \R^{n\times n}, \quad \psi=\rho':\R\to\R, 
\end{equation}
approximates well the out-of-sample error $\|\bSigma^{1/2}(\hat{\bbeta}-\bbeta_*)\|^2$. Note that 
$\bV\in \R^{n\times n}$ is the Jacobian matrix of the map $\bm{y}\in\R^n \mapsto \psi(\bm{y}-\bX\hat{\bbeta})\in\R^n$. Before moving to the formal statement, we introduce our assumption on the loss $\rho$ and noise distribution $F_\epsilon$.  
\begin{assumption}\label{as:loss}
$\rho$ is convex, differentiable, and $\{0\}=\argmin_{x}\rho(x)$, as well as $\psi=\rho':\R\to\R$ satisfies 
\begin{enumerate}
    \item $\|\psi\|_{\infty} <+\infty$. 
    \item $\|\psi\|_{\lip}=1$.
    \item $\exists \eta>0$ such that $\psi(x)^2/\|\psi\|_{\infty}^2 + \psi'(x) \ge \eta$ for almost every $x\in \R$. 
\end{enumerate}
\end{assumption}
\begin{assumption}\label{as:noise}
    $F_\epsilon$ is a convolution of $(F, \tilde{F})$, where $\tilde{F}$ is arbitrary while $F$ has some density $z\mapsto \exp(-\phi(z))$ 
    for some twice continuously differentiable function $\phi:\R\to\R$ with a bounded second derivative $\sup_{x\in \R} |\phi''(x)| <+\infty.$
\end{assumption}

The Huber loss $\rho(x) = \int_0^{|x|} \min(1, u)\diff u$  and its smooth approximations, such as
the pseudo Huber loss $\rho(x) = \sqrt{1+x^2}-1$,  satisfy \Cref{as:loss} (See \Cref{fig:loss_example}). 
\editline{
\Cref{as:noise} requires that each component $\epsilon_i$ of the noise in the linear model
is equal in distribution to $z_i+\delta_i$ for two independent random variables $z_i\sim F$ and $\delta_i\sim \tilde{F}$,
where $z_i$ has a smooth density $z\mapsto \exp(-\phi(z))$.
}
Typical example of the distribution $F$ in \Cref{as:noise} is the normal distribution $\N(0, \sigma^2)$ for some $\sigma>0$. We emphasize that there is no assumption on
$\tilde{F}$, so that the noise distribution $F_\epsilon=F*\tilde{F}$ can be heavy-tailed. For instance, we allow 
$F_\epsilon = \N(0, 1) * \Cauchy(0,1)$, 
in which $F_\epsilon$ has no finite moments.

\editline{\Cref{as:loss}(1)} is the condition for  \Cref{th:system} to be applicable;
\editline{this is a sufficient condition to guarantee that the system \eqref{eq:intro_nonlinear} has a unique solution}.
\editline{\Cref{as:loss}(2)} is mainly for the Jacobian $\bV$ of the map $\bm y\in\R^n \mapsto \psi(\bm y - \bm X \hat{\bm \beta})\in \R^n$ in \eqref{eq:est_ofs} to be well-defined. Indeed, under this assumption, $\bm y \mapsto \psi(\bm y - \bm X \hat{\bm \beta})$ is \editline{Lipschitz} for almost every $\bm X$ (cf. \Cref{prop:psi_lip} or Proposition 4.1 in \cite{bellec2023out}), so that Rademacher's theorem guarantees the existence of the Jacobian $\bV$ for almost every $(\by, \bX)\in \R^n \times \R^{n\times p}$. 
\editline{
The Lipschitz assumption in \Cref{as:loss}(2) can be relaxed to $\|\psi\|_{\lip} < +\infty$ since the unregularized M-estimator \eqref{eq:intro_est} remains invariant under a rescaling of the loss function, i.e., $\rho \mapsto \rho / \|\psi\|_{\lip}$. For simplicity and to streamline the proof arguments presented in \Cref{sec:proof_highlight} and the Appendix, we assume $\|\psi\|_{\lip} = 1$ throughout the paper.
}
\editline{\Cref{as:loss}(3)} is used to show that $\tr[\bV]$ in the denominator of $\hat{R}$ in \eqref{eq:est_ofs} is bounded from below by a positive constant times $n$ with high probability. 

\Cref{as:noise} on $F_\epsilon$ is a technical condition
to control $\tr[\bV]$;
see the discussion surrounding \eqref{eq:convolution_intro}.
An equivalent formulation of \Cref{as:noise} is that each entry
$\epsilon_i$ of the noise is a sum $z_i + \delta_i$ of
two independent random variables $z_i\sim F$ and $\delta_i\sim \tilde{F}$,
where $z_i$ has a smooth density $z\mapsto \exp(-\phi(z))$.
Numerical simulations in \Cref{subsec:relax_noise_assumption} suggest that our results still hold
for some $F_\epsilon$ without \editline{the presence of the smooth noise component
$z_i$}, which suggests that it is an artifact of the proof.
\editline{Some preliminary theoretical evidence that this assumption
    is not necesary is given in \Cref{lemma_vanishing}
    where this assumption is replaced with $\epsilon_i = \sigma z_i + \delta_i$
    with the same distributions as above but a vanishing
    amplitude $\sigma$ for the smooth part $z_i$.
    However, our results in \Cref{sec:tuning} require
    \Cref{as:noise} (and cannot accomodate a vanishing $\sigma$)
    in \Cref{lm:tr_lambda_lip}.
}

The quantity $\tr[\bm V]$ in \eqref{eq:est_ofs} is observable and
can be computed approximately by
Monte Carlo schemes
(see Section 2.11 in \cite{bellec2023out} and the references therein) 
or off-the-shelf numerical methods to compute derivatives.
For the Huber loss and the pseudo Huber loss, closed-form expressions
for $\tr[\bV]$ are available: 
$\tr[\bV]=|\{i\in[n]:y_i-\bx_i^\top\hbbeta\in[-1,1]\}|-p$
for the Huber loss 
and 
$$
\tr[\bV] = \sum_{i=1}^n\Bigl[
\psi'(\bx_i^\top\hbbeta)  
-\psi'(\bx_i^\top\hbbeta)^2\bx_i^\top \Bigl(\sum_{l=1}^n\bx_l \psi'(\bx_l^\top\hbbeta)\bx_l^\top\Bigr)^{-1} \bx_i
\Bigr]
$$ for the pseudo Huber loss $\rho(x)=\sqrt{1+x^2}-1$ or other twice-continuously differentiable loss functions with $\psi'=\rho''$ positive everywhere.
\begin{remark}\label{rm:system_condition_satisfied}
    \Cref{as:loss}-\ref{as:noise}
    imply that the system \eqref{eq:nonlinear}
    admits a unique solution. Indeed, according to \Cref{th:system}, the sufficient conditions for the system to admit a unique solution are 
    $(1)$ $\rho$ is convex, Lipschitz, and $\{0\}=\argmin_{x\in \R}\rho(x)$, and $(2)$ $\PP(W\ne 0)>0$ and \( \inf_{\lambda>0} \E[\dist(G, \lambda \partial \rho(W) )^2] > 1-\gamma \) for independent $W\sim F_\epsilon$ and $G\sim \N(0,1)$. Here, condition (1) is readily satisfied by \Cref{as:loss}. For condition (2), thanks to \Cref{as:noise}, $F_\epsilon$ does not have any point mass, so in particular $\PP_{W\sim F_\epsilon}(W\ne 0)=1$. Since $\rho$ is differentiable by \Cref{as:loss}, $\partial\rho(W)$ is always the singleton $\{\rho'(W)\}$, which gives $\inf_{\lambda>0}\E[\dist(G, \lambda\partial\rho(W))^2]=\inf_{\lambda>0}\E[(G-\lambda\rho'(W))^2] = 1>1-\gamma$. 
    %Therefore, the condition in \Cref{th:system} is always satisfied under \Cref{as:loss} and \ref{as:noise}.
\end{remark}

Now we claim that the random quantity $\hat{R}$ in \eqref{eq:est_ofs} is a consistent estimate of the out-of-sample error $\|\bSigma^{1/2}(\hat{\bbeta}-\bbeta_*)\|^2$. 
\begin{theorem}\label{th:ofs}
    Assume that $(\rho, F_\epsilon)$ satisfy \Cref{as:loss} and \ref{as:noise}. Let $\psi=\rho':\R\to\R$ be the derivative of the loss, and $\bV$ be the Jacobian matrix $(\partial/\partial \bm{y})\psi(\bm{y}-\bX\hat{\bbeta})\in\R^{n\times n}$. 
    Then, as $n, p\to\infty$ with \editline{$p/n\to \gamma\in (0,1)$}, we have
    \begin{align}
        \|\bSigma^{1/2} (\hat{\bm\beta}-\bm \beta^\star)\|^2 = \hat{R} + o_P(1), \text{ where } \hat{R} := \frac{p\|\psi(\bm{y}-\bX\hat{\bbeta})\|_2^2}{\tr[\bV]^2}. 
        \label{eq:conclusion_Theorem1}
    \end{align}
\end{theorem}

An outline of the proof is given in \Cref{sec:proof_highlight} 
\editline{
and the formal proof is given in
\Cref{sec_proof:th_ofs}.
We are able to relax \Cref{as:noise} on the noise as the next proposition
shows: the result \eqref{eq:conclusion_Theorem1} still holds if the smooth
component of the noise has vanishing amplitude $\sigma_n$.
}

\begin{proposition}
    \label{prop:relaxed}
\editline{
    Let \Cref{as:loss} be fulfilled, and assume that
    for each $i\in[n]$, the noise
    $\epsilon_i$ is equal in distribution to
    $\delta_i + \sigma_n z_i$ where $\delta_i \sim \tilde F$
    and $z_i\sim F$ are independent and $(F,\tilde F)$ are as in
    \Cref{as:noise},
    and where $\sigma_n\to 0$ with $\sigma_n \ge n^{-1/8}$.
    Then \eqref{eq:conclusion_Theorem1} still holds
    as $n,p\to+\infty$ with $p/n\to \gamma \in (0,1)$.
}
\end{proposition}

The proof is given in
\Cref{sec_proof_relaxed}.
\Cref{th:ofs} implies that the random quantity $\hat{R}$ is a consistent estimate of the out-of-sample error. 
Importantly, $\psi(\by - \bX \hat{\bbeta})$ and $\tr[\bV]$ are both observable, i.e., they can be computed by observed data $(\bm{y}, \bX)$ only. Thus, $\hat{R}$ serves as a criterion to select different losses. 
\begin{corollary}\label{cor:tuning}
Assume either that $F_\epsilon$ satisfies \Cref{as:noise},
\editline{
or that each iid component $\epsilon_i$ of the noise satisfies the relaxed condition in \Cref{prop:relaxed}.
}
For fixed integer $K$, consider $K$ different loss functions \editline{$\rho_1, \dots, \rho_K$} that satisfy \Cref{as:loss}. For each $k\in [K]$, let $\hat{\bbeta}_k \in \argmin_{\bbeta\in\R^p} \sum_{i=1}^n \rho_k(y_i-\bx_i^\top \bbeta)$ be the M-estimator computed by the $k$-th loss and $\hat{R}_k$ be the corresponding criterion \eqref{eq:est_ofs}.  
Then, we have
$$
\|\bSigma^{1/2} (\hat{\bbeta}_{\hat{k}} - \bbeta^\star)\|^2 = \min_{k=1, \dots, K} \|\bSigma^{1/2} (\hat{\bbeta}_{k} - \bbeta^\star)\|^2 + o_P(1), \ \text{ where } \hat{k} \in \argmin_{k=1, \dots, K} \hat{R}_k.
$$
\end{corollary}
See \Cref{proof:cor:tuning} for the proof.
Note that the left-hand side is the out-of-sample error of the M-estimator that minimizes the criterion $(\hat{R}_k)_{k=1}^K$ among $K$ different losses, while 
 $\min_{k=1, \dots, K} \|\bSigma^{1/2} (\hat{\bbeta}_{k} - \bbeta^\star)\|^2$ on the right-hand side 
is the optimal out-of-sample error among the candidates. 
Thus, \Cref{cor:tuning} implies that the M-estimator minimizing the criterion achieves the optimal out-of-sample error up to an error term that converges to $0$ in probability. 

\section{Adaptive tuning of scale parameters}\label{sec:tuning}
Let $\rho$ be a fixed robust loss satisfying \Cref{as:loss}. 
For $\lambda>0$, consider \textit{$\lambda$-scaled} loss $\rho_\lambda$ defined as
$$
\forall x\in \R, \quad \rho_\lambda(x) := \lambda^2\rho(x/\lambda).
$$
This $\lambda$ controls the sensitivity to outliers
{
or heavy tails, and different \( \lambda \) lead to significanclty
different performance as seen in \Cref{subfig:ofs_consistency,subfig:ofs_consistency_sigma0}.
}
When {the base loss} $\rho$ is the Huber loss or the pseudo Huber loss (see \Cref{fig:loss_example}), it holds that 
 \begin{align*}
    \forall x\in\R, \quad \rho_\lambda(x) = \left\{\begin{array}{lll}
         x^2 \cdot (x/\lambda)^{-2} \rho(x/\lambda) &  
       \to x^2/2 & \text{ as $\lambda \to +\infty$} \\
        \lambda x \cdot (x/\lambda)^{-1} \rho(x/\lambda)
       &  \sim  \lambda|x| & \text{ as 
      $\lambda\to 0+$}  
    \end{array}   
    \right., 
 \end{align*}
thanks to $\lim_{u\to 0}\rho(u)/u^2 =1/2$ and $ \lim_{u\to \pm \infty} \rho(u)/u=\pm 1$. 
Informally speaking, the scaled loss $\rho_\lambda$ behaves like the square loss for large $\lambda$ and like the absolute loss $|x|$ for small $\lambda$. 
%As we discussed in \Cref{sec:prelim}, this hyperparameter $\lambda$ controls the sensitivity of the loss $\rho_\lambda$ to outliers.

{Previous sections have} so far discussed the consistency of 
%our proposed estimate
{the estimate} $\hat{R}$ given by \eqref{eq:est_ofs}. In this section, we apply this
%this estimate to the adaptive tuning
{
consistency result to perform adaptive tuning of the parameter \( \lambda \).
}
The goal of this section is to select some $\lambda$ in a data-driven way so that the M-estimator $\hat{\bbeta}_\lambda$ {defined with} the scaled loss $\rho_\lambda$,
 \begin{equation}\label{eq:df_M_scaled}
\hat{\bbeta}_\lambda \in \argmin_{\bbeta\in \R^p} \frac{1}{n}\sum_{i=1}^n \rho_\lambda(y_i - \bx_i^\top\bbeta)  \quad \text{with} \quad \rho_\lambda(\cdot) := \lambda^2\rho(\cdot/\lambda)
\end{equation}
achieves an asymptotically optimal risk. 
%Before moving to its formal statement, 
Let us introduce some useful notation. For all $\lambda>0$, let $\psi_\lambda:\R\to\R$ be the derivative of the $\lambda$-scaled loss $\rho_\lambda$:
$$
\psi_\lambda: x\mapsto \rho_\lambda'(x) = \lambda \psi(x/\lambda) \quad \text{with}\quad \psi=\rho'.
$$
Note that {by construction}, $\|\psi_\lambda\|_{\lip}$ is invariant with respect to $\lambda$. We define the random functions ${R}$, $\hat{R}$, and the deterministic function $\alpha$ as
\begin{align}
    &R: (0,\infty) \to \R, &&\lambda \mapsto \|\bSigma^{1/2} (\hat{\bbeta}_\lambda-\bbeta^\star)\|^2 \text{ with $\hat{\bbeta}_\lambda$ being the M-estimator \eqref{eq:df_M_scaled}} \label{eq:df_risk_scaled}\\
    &\hat{R}: (0, \infty)\to\R,  &&\lambda\mapsto p\frac{\|\psi_\lambda(\by - \bX \hat{\bbeta}_\lambda)\|^2}{\tr[\bV_\lambda]^2} \text{ with } \bV_\lambda = \frac{\partial \psi_\lambda(\bm y - \bX\hat{\bbeta}_\lambda)}{\partial\bm{y}} \in \R^{n\times n} \label{eq:df_est_scaled}\\
    &\alpha: (0,\infty) \to \R, \quad &&\lambda\mapsto \alpha(\lambda) \ (\text{the solution to \eqref{eq:nonlinear} with $\rho(\cdot)=\rho_\lambda(\cdot)$}). \label{eq:df_alpha_scaled}
\end{align}
With the above notation,
\cite{donoho2016high,thrampoulidis2018precise}
grants \( R(\lambda)\to^p\alpha^2(\lambda) \) while
\Cref{th:ofs} and \Cref{th:system} yield 
\( \hat R(\lambda)\to^p\alpha^2 \) and an explicit upper bound
on \( \alpha^2(\lambda) \) with respect to $\lambda$.
We summarize these results in the
following proposition. 
\begin{proposition}\label{prop:scaled_control}
    Assume that $(\rho, F_\epsilon)$ satisfy \Cref{as:loss} and \ref{as:noise}. Then, the nonlinear system of equations \eqref{eq:nonlinear} with $\rho=\rho_\lambda$ admits a unique solution for all $\lambda>0$, so that the map $\lambda\mapsto \alpha^2(\lambda)$ in \eqref{eq:df_alpha_scaled} is well-defined. Furthermore, 
    as $n,p\to\infty$ with \editline{$p/n\to\gamma\in(0,1)$}, we have
    $$
    \hat{R}(\lambda) \to^p \alpha^2(\lambda), \quad R(\lambda)\to^p \alpha^2(\lambda), \quad 
    \alpha^2(\lambda) \le  \C(\gamma, F_\epsilon, \rho) (\lambda^2 + 1). 
    $$
    for all $\lambda>0$. 
\end{proposition}
See \Cref{proof:scaled_control} for the proof. 
\Cref{prop:scaled_control} suggests that 
the M-estimator $\hat{\bbeta}_{\hat\lambda}$ in \eqref{eq:df_M_scaled} with $\hat\lambda$ minimizing the criterion $\hat{R}(\lambda)$ over a {discrete grid}
achieves the optimal risk limit, i.e., 
$$
\alpha^2(\hat{\lambda})\approx \min_\lambda \alpha^2(\lambda) \text{ with } \hat{\lambda} \in  \argmin_{\lambda} \hat{R}(\lambda),
$$
as long as the map $\lambda\mapsto \alpha^2(\lambda)$ is smooth enough
{and the grid fine enough, so that at least one element \( \lambda \)
    of the grid
is sufficiently close to the optimal parameter (that may for instance
fall between two consecutive elements of the grid).}
To 
% show
{ensure} such smoothness {condition on the function} \( \alpha^2(\cdot) \), we introduce an additional assumption on the loss $\rho$ via $\psi=\rho'$. 
\begin{assumption}\label{as:tuning}
    $\psi=\rho':\R\to\R$ satisfies the following:
    $$
    \sup_{\lambda, \tilde{\lambda}>0, \lambda\ne\tilde{\lambda}}
    \sup_{x\in \R} \frac{|\lambda \psi(x/\lambda) - \tilde{\lambda}\psi(x/\tilde{\lambda})|}{|\lambda-\tilde{\lambda}|} < + \infty. 
    $$
\end{assumption}
Note that the Huber loss $\rho(x) = \int_{0}^{|x|} \min(1, u) du$ satisfies \Cref{as:tuning}.
If $\rho$ is twice continuously differentiable, the sufficient condition for \Cref{as:tuning} is
\begin{align}\label{eq:sufficient}
    \sup_{u\in \R} |\psi(u) - u\psi'(u)| <+\infty 
\end{align}
Indeed, 
for all $\lambda, \lambda\in (0,\infty)$ with $\lambda\ne\tilde{\lambda}$ and for all $x\in \R$, there exists $\lambda_x>0$ by the mean value theorem, such that 
\begin{edited}
$$
    \frac{|\lambda\psi(x/\lambda) - \tilde{\lambda}\psi(x/\tilde{\lambda})|}{|\lambda-\tilde{\lambda}|} = \Bigl|\frac{d}{d \lambda} \Bigl(\lambda\psi\Bigl(\frac{x}{\lambda}\Bigr)\Bigr)\Bigr|_{\lambda=\lambda_x} = \Bigl|\psi\Bigl(\frac{x}{\lambda_x}\Bigr) - \frac{x}{\lambda_x} \psi'\Bigl(\frac{x}{\lambda_x}\Bigr)\Bigr| \le \sup_{u\in \R} |\psi(u) - u\psi'(u)|,
$$ 
\end{edited}
which implies that  \eqref{eq:sufficient} is a  sufficient condition for \Cref{as:tuning}. We can easily check that \eqref{eq:sufficient} is satisfied by the pseudo Huber loss $\rho(x) = \sqrt{1+x^2} - 1$. 

Now we will claim that the map $\lambda\mapsto \alpha^2(\lambda)$ in \eqref{eq:df_alpha_scaled} is localy $1/2$-H\"older continuous.
See \Cref{proof:alpha_lipschitz} for the proof. 
\begin{theorem}\label{th:alpha_lipschitz}
Assume that $(\rho, F_\epsilon)$ satisfy  \Cref{as:loss}, \ref{as:noise}, and \ref{as:tuning}. 
For each $\lambda>0$, 
let $\alpha^2(\lambda)$ be the solution to the nonlinear system of equations \eqref{eq:nonlinear} with $\rho(\cdot)=\lambda^2\rho(\cdot/\lambda)$. 
Then, the map $\lambda \mapsto \alpha^2(\lambda)$ is locally $1/2$-H\"older continuous in the sense that
$$
\sup_{\lambda_{\min} \le \lambda\le\lambda_{\max} } \frac{|\alpha^2(\lambda)-\alpha^2(\tilde{\lambda})|}{|\lambda-\tilde{\lambda}|^{1/2}} \le \C(\lambda_{\min}, \lambda_{\max}, \rho, F_\epsilon, \gamma)
$$
for all $\lambda_{\min}, \lambda_{\max}\in (0,\infty)$ with $\lambda_{\min}<\lambda_{\max}$. 
\end{theorem}
\begin{remark}
    Similar result to \Cref{th:alpha_lipschitz} is known for the regularization parameter $\lambda$ of the Lasso $\lambda\|\bbeta\|_1$ (see \cite{celentano2020lasso,miolane2021distribution}). This line of work shows a suitable smoothness of the associated nonlinear system with respect to $\lambda$, using the implicit function theorem.
    %In contrast, we use a different proof technique; 
    {We proceed differently here for the proof of \Cref{th:alpha_lipschitz}}: we 
    show the map 
    $$
    \lambda \in (0, \infty) \mapsto \hat{R}(\lambda) = \frac{\|\psi_\lambda(\bm{y}-\bX\hat{\bbeta}_\lambda)\|^2}{\tr[\bV_\lambda]^2}
    $$
    is H\"older continuous.
    %and then trace back to 
    {This H\"older continuity of \( \hat R(\cdot) \) implies that of}
    $\alpha^2(\cdot)$ thanks to the approximation $\alpha^2(\lambda)=\hat{R}(\lambda)+o_P(1)$ by \Cref{prop:scaled_control}. 
    H\"older continuity of \( \hat R(\cdot) \) follows if
    $
    \lambda\mapsto \|\psi_\lambda(\by-\bX\hat{\bbeta}_\lambda)\|^2$ and $\lambda\mapsto \tr[\bV_\lambda]^2
    $
    are both H\"older continuous. 
    {The proof in \Cref{proof:alpha_lipschitz} shows}
    the H\"older continuity of the map $\lambda \mapsto \|\psi_\lambda(\bm{y}-\bm{X}\hat{\bbeta}_\lambda)\|^2$ using \Cref{as:tuning} and the \editline{KKT} conditions. For the map $\lambda\mapsto \tr[\bV_\lambda]^2$, we show that $\tr[\bV_\lambda]$ inherits the H\"older continuity from $\psi_\lambda(\by - \bx_i^\top \hat{\bbeta}_\lambda)$, by
{
    leveraging the assumption that the noise
    distribution is sufficiently smooth;
    see \Cref{th:second} and the discussion around
    \eqref{eq:convolution_intro} on how this smoothness assumption
    is also used for results in \Cref{sec:ofs}.
}
\end{remark}
Finally, we apply \Cref{th:alpha_lipschitz} to the adaptive tuning of the scale parameter $\lambda$. 
Let us take  $\lambda_{\min}$ and $\lambda_{\max}$ such that 
$0 < \lambda_{\min} < \lambda_{\max}$, and take the closed interval $I=[\lambda_{\min}, \lambda_{\max}]$. For this fixed interval $I$, we consider a finite grid $I_N$ of $(N+1)$ points equispaced in log-scale, i.e.,
\begin{align}\label{eq:df_grid}
   I_N := \Bigl\{
   %\lambda_i = 
   \lambda_{\min}\Bigl(\frac{\lambda_{\max}}{\lambda_{\min}}\Bigr)^{i/N}: i=0, \dots, N\Bigr\} \subset I=[\lambda_{\min}, \lambda_{\max}]. 
\end{align}
for all $N\in\mathbb{N}$. 
We define {the oracle optimal parameter} $\lambda_{\opt}\in I$ and {the data-driven selected parameter} $\hat{\lambda}_N\in I_N$ as 
$$
\lambda_{\opt} \in \argmin_{\lambda \in I} \alpha^2(\lambda), \quad  \hat{\lambda}_N \in \argmin_{\lambda \in I_N} \hat{R}(\lambda). 
$$
Note that 
$\lambda_{\opt}$ exists thanks to the continuity of $\lambda\mapsto \alpha^2(\lambda)$ from \Cref{th:alpha_lipschitz} and the compactness of $I$. Here, $\lambda_{\opt}$ is the optimal scale parameter in the closed interval $I$, while
$\hat{\lambda}_N$ is the minimizer of the criterion $\hat{R}(\lambda)$ among the finite grid $I_N$ of $(N+1)$ points.
 With the above notation, we claim that $\hat{\lambda}_N$ achieves 
the theoretically optimal risk limit $\alpha^2(\lambda_{\opt}) = \min_{\lambda\in I}\alpha^2(\lambda)$. 
\begin{theorem}\label{th:lambda_tuning}
Assume that $F_\epsilon$ satisfies \Cref{as:noise}, and $\rho$ satisfies \Cref{as:loss} and \ref{as:tuning}. Let $R$, $\hat{R}$, and $\alpha^2$ be the maps defined by \eqref{eq:df_risk_scaled}, \eqref{eq:df_est_scaled}, and \eqref{eq:df_alpha_scaled}, respectively. 
Then, for any closed interval $I=[\lambda_{\min}, \lambda_{\max}]\subset (0,\infty)$, 
there exists a sequence of integers $(N_n)_{n=1}^\infty=(N_n(\gamma, \rho, F_\epsilon, \lambda_{\min}, \lambda_{\max}))_{n=1}^\infty$ such that as $n, p\to+\infty$ with $p/n\to\gamma\in(0,1)$, we have
$$
R(\hat{\lambda}_{N_n}) =
\alpha^2(\hat{\lambda}_{N_n}) + o_P(1)
= \alpha^2(\lambda_{\opt}) + o_P(1),
$$
where $\lambda_{\opt} \in \argmin_{\lambda \in I} \alpha^2(\lambda)$ and $\hat{\lambda}_N \in \argmin_{\lambda \in I_N} \hat{R}(\lambda)$ with $I_N$ given by \eqref{eq:df_grid}.
\end{theorem}

See \Cref{sec:proof_tuning} for the proof. 
\Cref{th:lambda_tuning} implies that if the 
%grids
{grid} $I_N$ is fine enough, the M-estimator $\hat{\bbeta}({\hat\lambda})$ with ${\hat\lambda}$ minimizing
the criteria $\hat{R}(\lambda)$ {over the grid}
achieves the theoretically optimal risk $\alpha^2(\lambda_{\opt})=\min_{\lambda\in I}\alpha^2(\lambda)$ in the fixed interval $I$ of $\lambda$.  
We will verify \Cref{th:lambda_tuning} by numerical simulations in \Cref{sec:numeric}. 

\section{Numerical simulations}\label{sec:numeric}
 We focus on the Huber loss $\rho(x) = \int_{0}^{|x|}\min(1,t)dt$ and consider the adaptive tuning of the scale parameter $\lambda$ for the scaled loss $\rho_\lambda(\cdot) = \lambda^2\rho(\cdot/\lambda)$. 
Note that the scaled loss $\rho_\lambda$ and its derivative $\psi_\lambda := \rho_\lambda'$ can be written explicitly as 
\begin{align}\label{eq:huber_scaled_definition}
    \forall x\in \R, \quad \rho_\lambda (x) := \left\{
        \begin{array}{ll}
            x^2/2 & |x| \le \lambda\\
            \lambda|x|-\lambda^2/2 & |x|\ge\lambda
        \end{array}
        \right., \quad \psi_\lambda(x) = \left\{
            \begin{array}{ll}
                x & |x|\le \lambda\\
                \lambda \sign(x) & |x| \ge \lambda
            \end{array}
        \right..     
\end{align}
Below, we use the notation in \eqref{eq:df_risk_scaled}, \eqref{eq:df_est_scaled}, and \eqref{eq:df_alpha_scaled}; let $\hat{\bbeta}_\lambda \in \argmin_{\bbeta\in\R^p} \sum_{i=1}^n \rho_\lambda(y_i-\bx_i^\top{\bbeta})$ be the unregularized M-estimator 
computed by the scaled loss $\rho_\lambda$, and let $R(\lambda)$ and $\hat{R}(\lambda)$ be the out-of-sample error and our proposed estimate
\begin{align*}
    R(\lambda) := \|\bSigma^{1/2}(\hat{\bbeta}_\lambda-\bbeta_*)\|, \qquad 
   \hat{R}(\lambda):= \frac{p \sum_{i=1}^n \psi_\lambda (y_i-\bx_i^\top\hat{\bbeta}_\lambda)^2}{|\{i\in [n]: |y_i-\bx_i^\top\hat{\bbeta}_\lambda| \le \lambda\}|^2},
\end{align*}
where we have used the simplification $\tr[\bV_\lambda]  = \sum_{i=1}^n \bm{1}\{|y_i-\bx_i\hat{\bbeta}_\lambda|\le \lambda\}$ in \eqref{eq:df_est_scaled}. Let $\alpha^2(\lambda)$ be the solution to the nonlinear system \eqref{eq:nonlinear} with $\rho=\rho_\lambda$. 

With the above notation, we first verify 
\begin{align}\label{eq:conjecture_1}
    \hat{R}(\lambda) \approx R(\lambda) \approx \alpha^2(\lambda).
\end{align}
\editline{
We set $(n, p) = (4000, 1200)$, $F_\epsilon = \tdist (\df=2)$, $\bSigma=\bm I_p$ and $\bbeta^\star=\bm 0_p$. 
Once we generate $(\bm{y}, \bX)$,  
we compute $(R(\lambda), \hat{R}(\lambda))$ for each $\lambda$ in a finite grid. We repeat the above procedure $100$ times and 
plot $(R(\lambda)$, $\hat{R}(\lambda))$ in \Cref{subfig:ofs_consistency}, and the relative error $|\hat{R}(\lambda)/R(\lambda)-1|$ in \Cref{sub@subfig:ofs_relative_error}. 
We also plot $\alpha^2(\lambda)$ in \Cref{subfig:ofs_consistency} by solving the nonlinear system of equations \eqref{eq:nonlinear} \tk{(see \Cref{rm:how_to_solve_system} for details)}. 
\editline{
\Cref{subfig:risk_estimate_intro} in the introduction features the same
experiment as \Cref{subfig:ofs_consistency}, without the
theoretical curve $\lambda \to \alpha^2(\lambda)$. 
}
%\footnote{\editline{Rearranging the equations we solve it as a two-dimensional fixed point equation}}
These figures are consistent with \eqref{eq:conjecture_1}. 
}
Next, we conduct the adaptive tuning of the scale parameter $\lambda>0$. 
Let us take $I=[1,10]$ and the finite grid $I_N$ as $\{\lambda_i=10^{i/100}: i\in 0, 1, \dots, 100\}\subset I$. Then, \Cref{cor:tuning} and \Cref{th:lambda_tuning} implies 
\begin{align}\label{eq:conjecture_2}
    \min_{\lambda\in I_N} R(\lambda) \approx R(\hat{\lambda}) \approx \min_{\lambda\in I}\alpha^2(\lambda) \quad \text{where}\quad \hat{\lambda}\in \argmin_{\lambda\in I_N} \hat{R}(\lambda). 
\end{align}
Below, we verify \eqref{eq:conjecture_2}
as we change the scale of noise distribution $F_\epsilon$  in the following way:
$$
\editline{
F_\epsilon := \sigma \cdot \tdist(\df=2) \text{ where  } \sigma\in [1, 3].
}
$$
For each $\sigma$ in a finite grid over \editline{$[1,3]$}, we generate dataset $(\bX, \bm{y})$ and calculate $R(\hat{\lambda}_N)$ and $\min_{\lambda \in I_N} R(\lambda)$.
We repeat the above procedure \editline{$100$} times and plot $R(\hat{\lambda})$ and $\min_{\lambda \in I} R(\lambda)$ in \Cref{subfig:tuning_oracle}. 
We also plot $\min_{\lambda \in I} \alpha(\lambda)^2$ in the same figure.
\editline{
\Cref{subfig:adaptive_tuning_intro} in the introduction has the same
experiment as \Cref{subfig:tuning_oracle}, without the
theoretical limit.
}

In \Cref{subfig:tuning_relative_error}, we plot the two relative errors:
$|\frac{R(\hat{\lambda}_N)}{\min_{\lambda\in I_N} R(\lambda)}-1|$ and $|\frac{R(\hat{\lambda}_N)}{\min_{\lambda\in I}\alpha^2(\lambda)}-1|$.
These Figures are consistent with \eqref{eq:conjecture_2}.

\begin{figure}[htpb]
    \centering
    \begin{subfigure}[b]{0.49\textwidth}
        \centering
        \includegraphics[width=\textwidth]{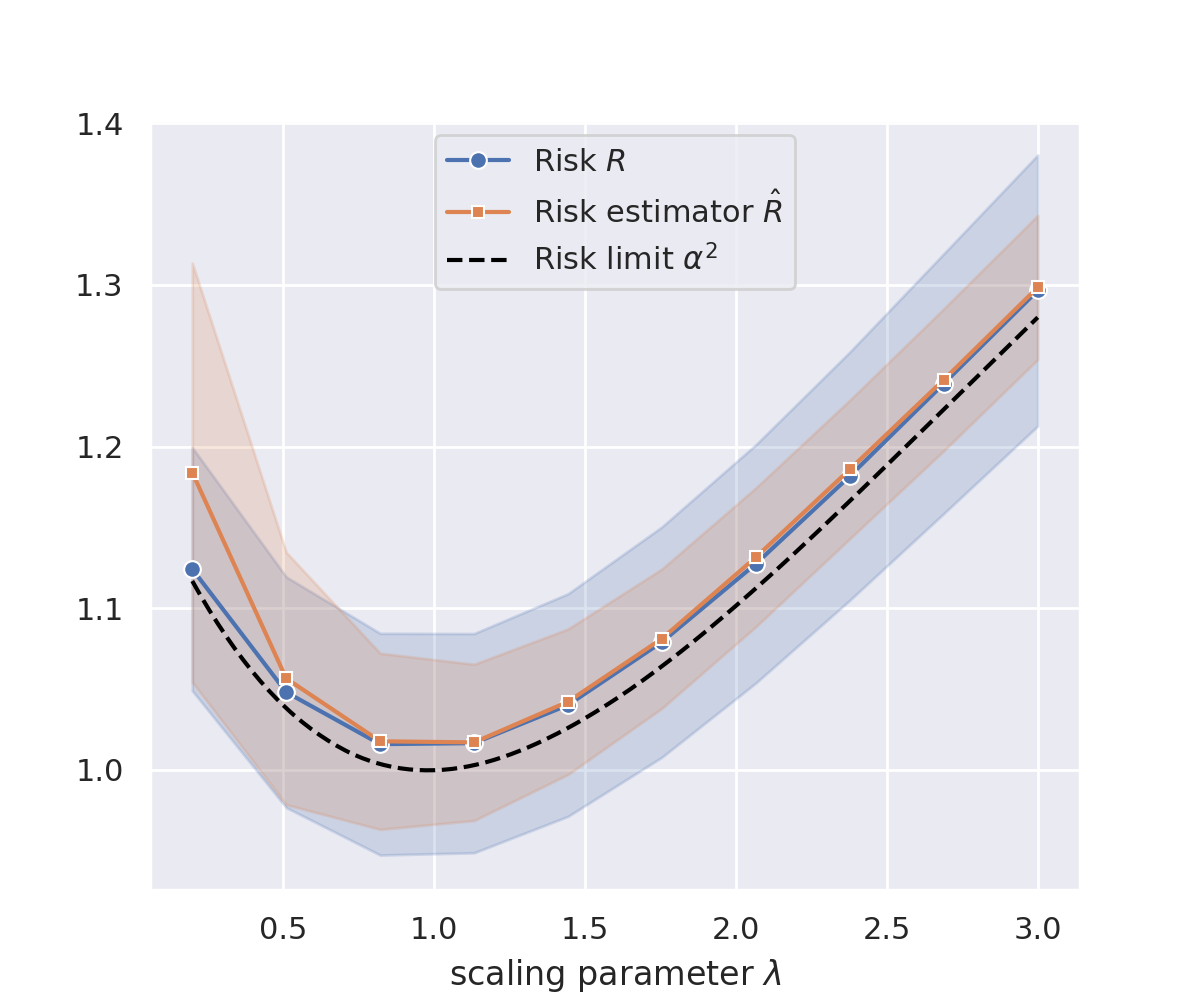}
        \caption{$R(\lambda)$, $\hat{R}(\lambda)$ and $\alpha^2(\lambda)$.}
        \label{subfig:ofs_consistency}
    \end{subfigure}
    \begin{subfigure}[b]{0.49\textwidth}
        \centering
        \includegraphics[width=\textwidth]{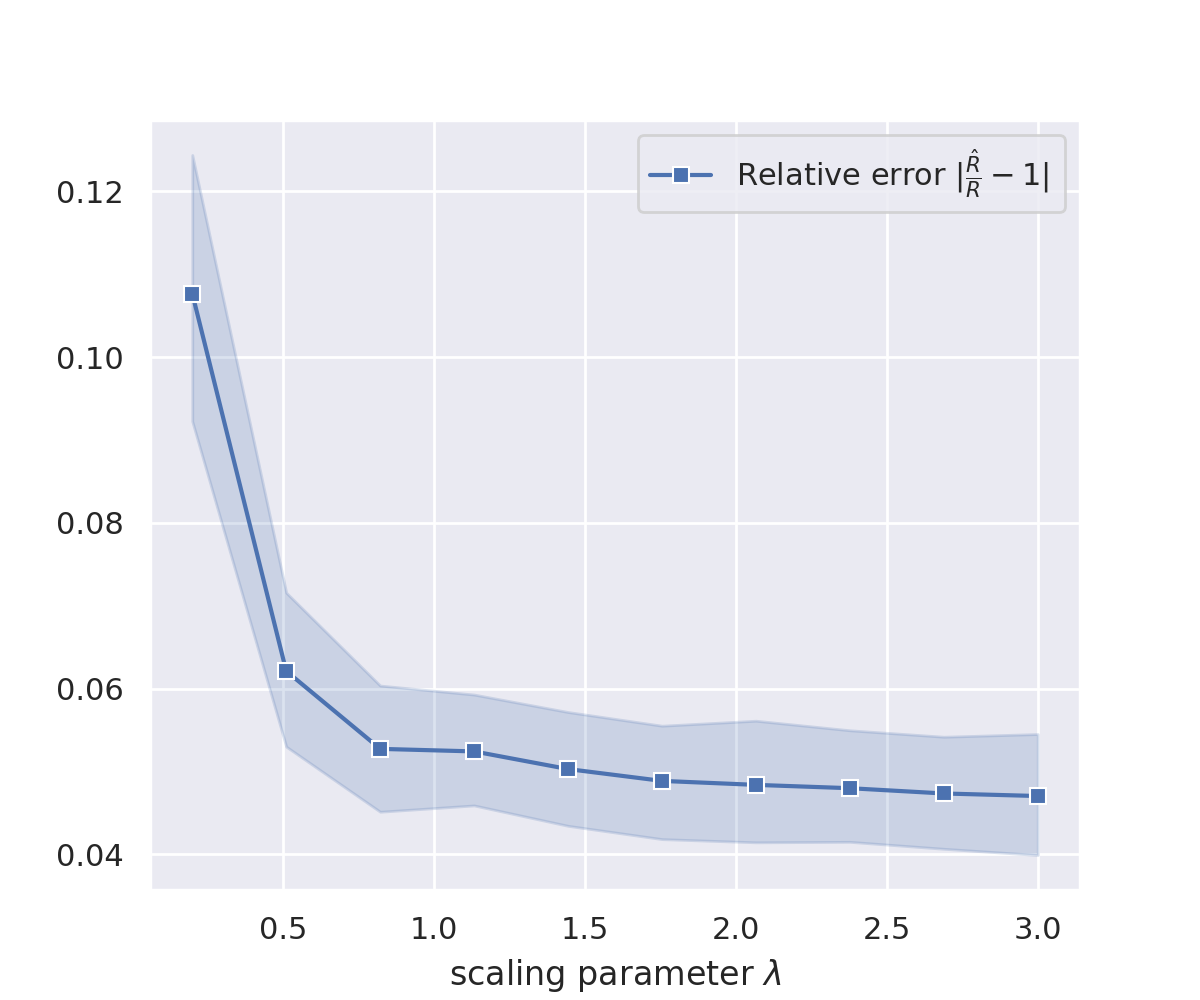}
        \caption{Relative error $|\hat{R}(\lambda)/R(\lambda)-1|$}
        \label{subfig:ofs_relative_error}
    \end{subfigure}
    \caption{
        \editline{Plot of the out-of-sample error $R(\lambda)$ and estimator $\hat{R}(\lambda)$ over 100 repetitions, with $n=4000$, $p=1200$, for the Huber loss for different values of the scale parameters $\lambda$. The noise distribution is $\text{t-dist}(\df=2)$. 
    $\alpha^2(\lambda)$ is the solution to the nonlinear system \eqref{eq:nonlinear}.}}
    \label{fig:ofs}
\end{figure}

\begin{figure}[htpb]
    \centering
    \begin{subfigure}[b]{0.49\textwidth}
        \centering
        \includegraphics[width=\textwidth]{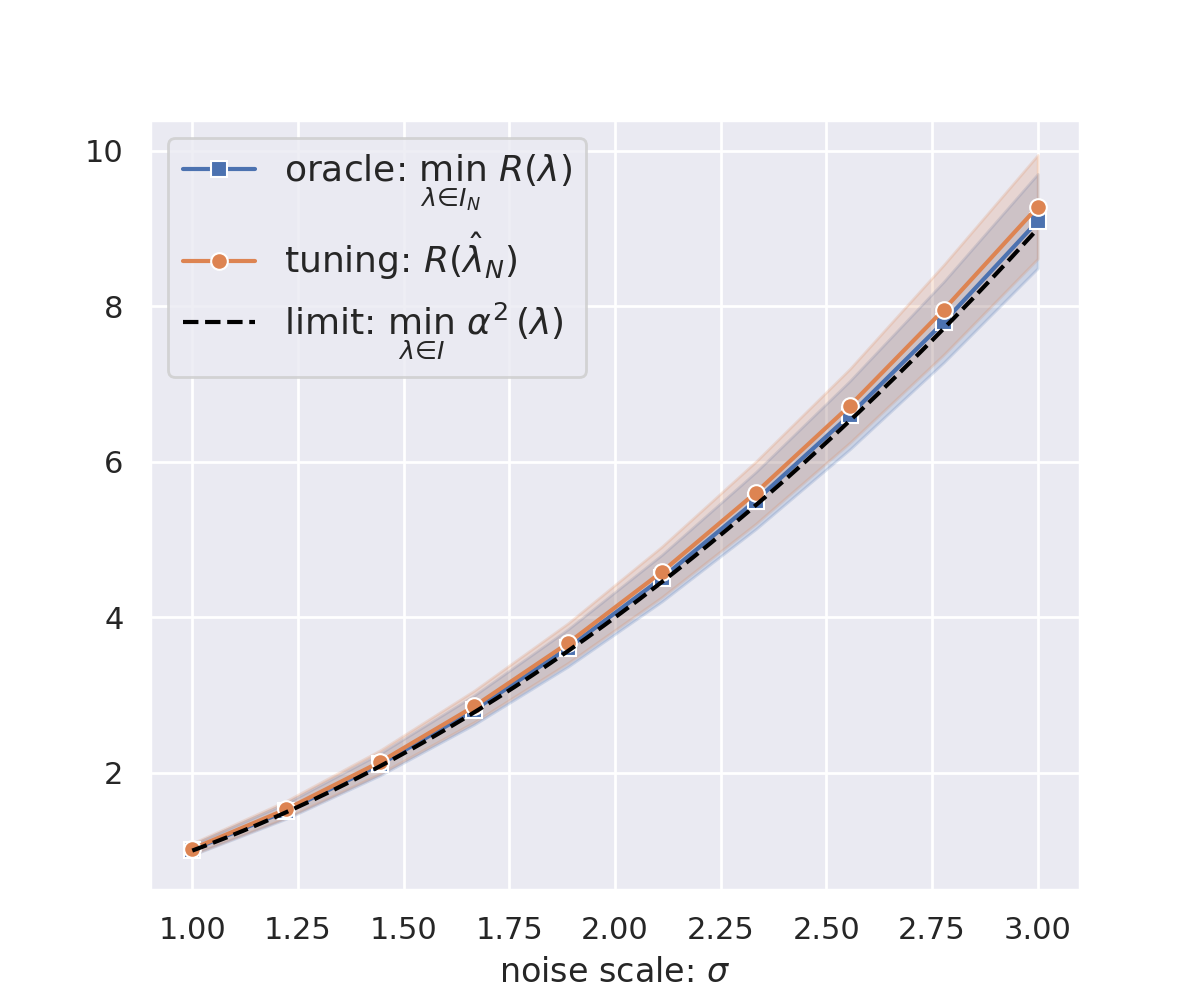}
        \caption{Adaptive tuning vs oracles}
        \label{subfig:tuning_oracle}
    \end{subfigure}
    \begin{subfigure}[b]{0.49\textwidth}
        \centering
        \includegraphics[width=\textwidth]{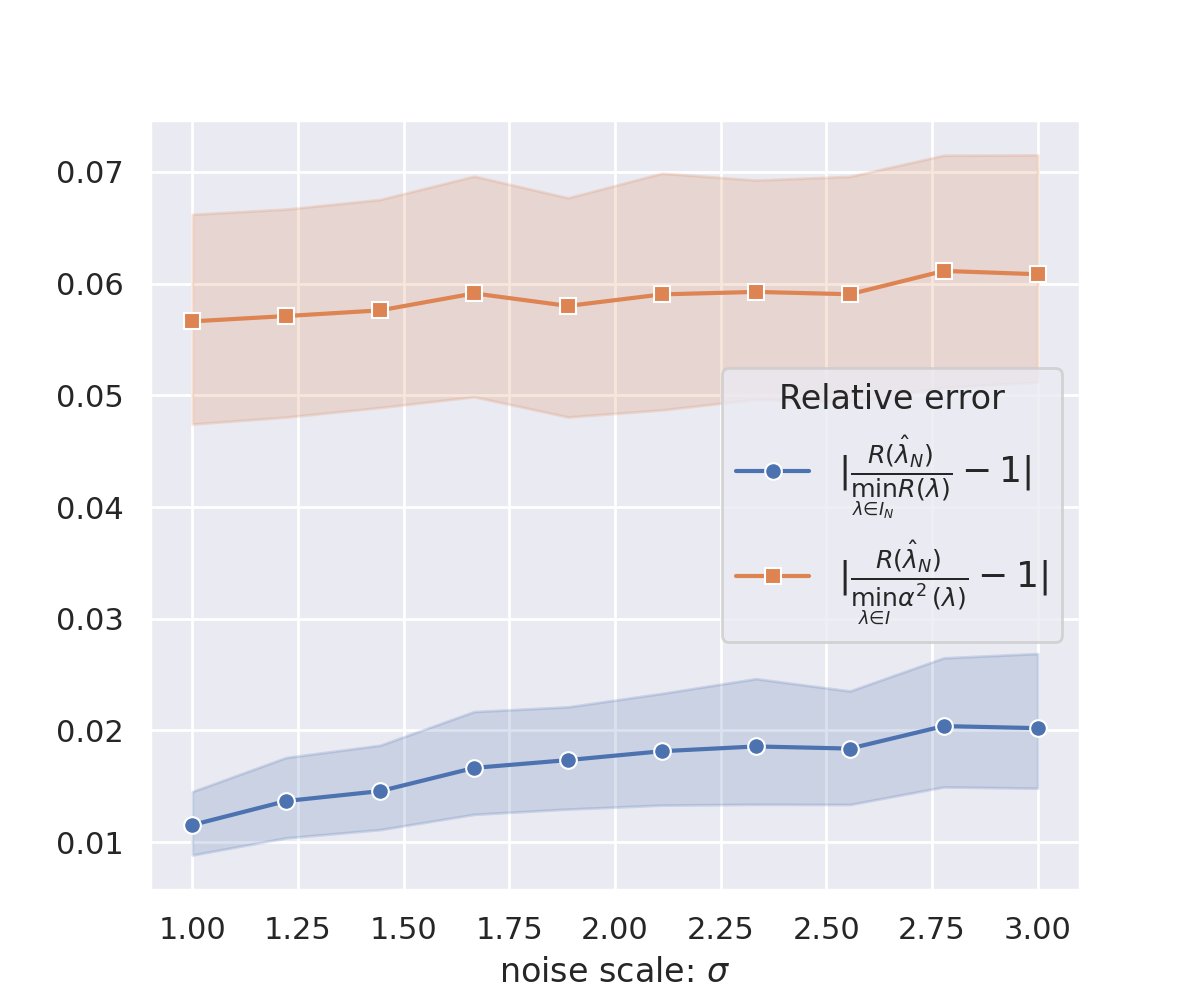}
        \caption{Relative error}
        \label{subfig:tuning_relative_error}
    \end{subfigure}
    \caption{
    \editline{
    Adaptive tuning with the scale of noise $\sigma$ changing as $F_\epsilon = \sigma \cdot \text{t-dist}(\df=2)$.
    {Here,}
    {$I=[1,10]$}, $n=4000$, $p=1200$, and $I_N$ is the uniform grid in log-scale of length $101$. 
    ${R}(\hat{\lambda}_N)$ is the out-of-sample error with $\lambda$ selected by $\hat{R}$, $\min_{\lambda\in I_N} R(\lambda)$ is the optimal out-of-sample error among $I_N$, and $\min_{\lambda\in I}\alpha^2(\lambda)$ is the theoretically optimal risk limit. We repeat $100$ times.
    }
    }
    \label{fig:tuning}
\end{figure}

\section{Outline of the proof}\label{sec:proof_highlight}
In this section, we give a sketch of proof of \Cref{th:ofs}. See \Cref{sec:proof_ofs} for rigorous proofs. 
Let $\bm G=\bm X \bm \Sigma^{-1/2} \in \mathbb{R}^{n\times p}$ so that $\bm G$ has i.i.d $\mathcal{N}(0,1)$ entries, and we denote by $\bm g_i$ the $i$-th row of $\bm G$. 
By the change of variable $\bm \beta \mapsto \bm h= \bm\Sigma^{1/2} (\bm\beta-\bm\beta^\star)$, 
we write the M-estimator $\hat{\bbeta}$ in \eqref{eq:intro_est} by
 $\hat{\bm \beta} = \bm \Sigma^{-1/2} \hat{\bm h} + \bm \beta^\star$, where  $\hat{\bm h}$ is defined as
\begin{align}
  \hat{\bm h} (\bm \epsilon, \bG) &\in \argmin_{\bm h\in \mathbb{R}^p} \frac{1}{n} \sum_{i} \rho(y_i - \bm x_i^\top (\bm \Sigma^{-1/2} \bm h + \bm \beta^\star)) = 
    \argmin_{\bm h\in \mathbb{R}^p} \frac{1}{n}\sum_{i=1}^n \rho(\epsilon_i - \bm g_i^\top \bm h). \label{eq:df_h}
\end{align}
Throughout, we regard $\hat{\bm h}$ as a function of $(\bm \epsilon, \bm G)$. Then, quantities of interest such as
$\|\bm \Sigma^{1/2} (\hat{\bm \beta}-\bm\beta_\star)\|$, $\psi(\bm y - \bm X\hat{\bm\beta})$, and $\bm V = (\partial/\partial \bm{y})\psi(\bm y-\bm X\hat{\bm\beta})$
can be  expressed as functions of $(\bm \epsilon, \bm G)$:
\begin{align*}
    \|\bm \Sigma^{1/2} (\hat{\bm \beta}-\bm\beta^\star)\| = \|\hat{\bm h}\|, \quad
    \psi(\bm y - \bm X\hat{\bm\beta}) =\psi(\bm\epsilon-\bm G \hat{\bm h})\in\R^{n}, \quad
    \bm V =(\partial/\partial \bm{\epsilon})\psi(\bm\epsilon - \bG \hat{\bbh})\in \R^{n\times n}.
\end{align*}
With the above notation, the statement of \Cref{th:ofs} is simplified to
\begin{align}\label{eq:target}
\|\hat \bbh \|^2= \frac{p}{\tr[\bm V]^2} \|\psi(\bm \epsilon - \bm G \hat{\bm h})\|^2 + o_P(1).
\end{align}
To prove \eqref{eq:target}, 
we consider the ridge-regularized M-estimator $\hat{\bbh}_\mu$
\begin{align}\label{eq:df_h_mu}
   \hat{\bm h}_\mu (\bm \epsilon, \bm G) := \argmin_{\bm h\in \mathbb{R}^p} \frac 1 n \sum_{i=1}^n \rho(\epsilon_i - \bm g_i^\top \bm h) + \frac{\mu}{2} \|\bm h\|^2.
\end{align}
for a diminishing regularization parameter $\mu=\mu_n\to 0$. 
Note that $\hat{\bm h}_\mu$ coincides with $\hat{\bm h}$ defined by \eqref{eq:df_h} when $\mu=0$. 
The benefit of considering $\hat{\bm h}_\mu$ 
is its rich differentiability structures that will be explained in \Cref{subsec:diff} ahead. 
Such derivative structure is helpful when we use probabilistic tools that involve derivatives, 
e.g., Stein's formula \editline{\citep{stein1981estimation,bellec2021second}}, the Gaussian Poincaré inequality \cite[Theorem 3.20]{boucheron2013concentration}, 
and normal approximations \citep{chatterjee2009fluctuations, bellec2023debiasing}, etc.
\subsection{Differentiability of M-estimators}\label{subsec:diff}
In this section, we introduce a useful differentiable structure of
the regularized M-estimator $\hat{\bm h}_\mu$ defined in \eqref{eq:df_h_mu}.  
To begin with, we discuss the differentiability of $\psi(\bm\epsilon-\bm G \hat{\bm h}_\mu)$ with respect to $\bm\epsilon$. 
\begin{proposition}[Proposition 4.1 in \cite{bellec2023out}]\label{prop:psi_lip}
Suppose that $\rho$ is convex and differentiable and that $\psi=\rho'$ is $1$-Lipschitz. Then, for any $M$-estimator of the form
$\hat{\bbeta} \in \argmin_{\bbeta\in \R^p} \sum_{i=1}^n \rho(\bm y_i - \bm x_i^\top\bbeta) + g(\bbeta)$ where $g: \R^p \to \R$ is a convex penalty, the map $\bm y \in \R^n \mapsto \psi(\bm y - \bm X \hat{\bbeta})\in \R^n$ is $1$-Lipschitz for almost every $\bm X\in \R^{n\times p}$.
\end{proposition}
Since $\psi=\rho'$ is 1-Lipschitz,
\Cref{prop:psi_lip} implies that the map $\bm \epsilon \mapsto \psi(\bm \epsilon -\bm G \hat{\bm h}_\mu (\bm \epsilon, \bm G))$ is $1$-Lipschitz for all $\mu\geq 0$. As a consequence, Rademacher's theorem gives the following:
\begin{align}\label{eq:df_Vmu}
 \forall \mu \geq 0, \quad \bV_\mu := (\partial/\partial \bm{\epsilon}) \psi(\bm \epsilon - \bm G \hat{\bm h}_\mu) \in \R^{n\times n} \text{ exists almost everywhere and }
 \|\bV_\mu\|_{op}\leq 1.
\end{align}
We emphasize that \eqref{eq:df_Vmu} also holds for $\mu=0$, so 
the original Jacobian matrix $\bm V = (\partial/\partial \bm{\epsilon})\psi(\bm \epsilon - \bm G\hat{\bm h})$ is also well-defined and $\|\bV\|_{op} \le 1$. 

The next theorem introduces a detailed differentiable structure of $\hat{\bbh}_\mu$ 
with respect to $\bG$ when $\mu$ is strictly positive.
\begin{theorem}[Theorem 1 in \cite{bellec2022derivatives}]\label{th:diff_smoothed}
Suppose $\mu > 0$. Then,
for almost every $(\bm\epsilon, \bm G)$, the map $(\bm\epsilon, \bm G)\mapsto \hat{\bm h}_\mu(\bm\epsilon, \bm G)$ defined by \eqref{eq:df_h_mu} is differentiable at almost every $(\bm \epsilon, \bm G)$, and 
the derivatives are given by 
\begin{align*}
    \forall i \in [n], \ \frac{\partial \hat{\bm h}_\mu}{\partial \epsilon_i} (\bm\epsilon,\bm G) &= \hat{\bm A}_\mu\bm G^\top \bm{e}_i \psi'(r_i),\\
    \forall i\in [n], \forall j\in [p], \ \frac{\partial \hat{\bm{h}}_\mu}{\partial g_{ij}} (\bm\epsilon, \bm G) &=\hat{\bm A}_\mu\bm{e}_j\psi(r_i)- \hat{\bm A}_\mu\bm G^\top \bm{e}_i\psi'(r_i) \bm{e}_j^\top \hat{\bm h}_\mu,
\end{align*}
where $\hat{\bm A}_\mu = (\bm G^\top \diag\{\psi'(\bm r)\} \bm G + n \mu I_p)^{-1} \in \mathbb{R}^{p\times p}$ 
and $\bm r = \bm\epsilon - \bm G \hat{\bm h}_\mu\in \mathbb{R}^n$. 
\end{theorem}
See \cite{bellec2022derivatives} for the differentiability of regularized M-estimators for general strongly convex penalties.  
Thanks to the strict positiveness of $\mu$, the matrix $\hat{A}_\mu$ is always well-defined and its operator norm is bounded from above by $(n\mu)^{-1}$.
By \Cref{th:diff_smoothed} and the chain rule, we find that $\psi(\bm \epsilon - \bm G\hat{\bm h}_\mu)$ is also differentiable at $(\bm \epsilon, \bm G)$ when $\mu >0$, 
with the derivatives given by
\begin{align*}
  (\partial/\partial g_{ij}) \psi(\bm \epsilon - \bm G \hat{\bm h}_\mu) &= -\diag\{\psi'(\bm r)\} \bm G \hat{\bm A}_\mu \bm{e}_j\psi(r_i) - \bm{V}_\mu \bm{e}_i \bm{e}_j^\top \hat{\bm h}_\mu \in \R^n,\\
    \bm V_\mu = (\partial/\partial \bm{\epsilon}) \psi(\bm \epsilon- \bm G \hat{\bm h}_\mu) 
    &= \diag\{\psi'(\bm r)\} - \diag\{\psi'(\bm r)\} \bG \hat{\bm A}_\mu \bG^\top \diag\{\psi'(\bm r)\} \in \R^{n\times n}.
\end{align*}
Note that the derivative formula in  \Cref{th:diff_smoothed} is only guaranteed for the smoothed M-estimator $\hat{\bbh}_\mu$ with a positive ridge parameter $\mu$, and not for 
the original M-estimator $\hat{\bbh}$ ($\mu=0$). This is the benefit of considering the smoothed M-estimator $\hat{\bbh}_\mu$. 

Next, we will use those derivative formulae to prove  the
interplay between  $\|\hat{\bbh}_\mu\|^2$, $\tr[\bV_\mu]$, and 
$\|\psi(\bm\epsilon-\bG \hat{\bbh}_\mu)\|^2$. 
\subsection{{Consistency of the risk estimate for} the smoothed M-estimator}\label{subsec:rep_smoothed}
We have discussed the differentiable structure of the regularized M-estimator $\hat{\bm h}_\mu$.
  In this section, we 
{
  derive the key relationship \eqref{eq:key_mu} below
}
  between
  $\|\hat{\bm h}_\mu\|^2$, $\tr[\bm V_\mu]$, 
  and $\|\psi(\bm \epsilon-\bm G \hat{\bm h}_\mu)\|^2$,
  {which can be seen as an extension of
      \eqref{eq:ofs_strong_convex} to M-estimators
      with a Ridge penalty function that vanishes polynomially
      in \( n \) as \( n\to+\infty \).
  }
\begin{lemma}\label{lm:str_smoothed}
\editline{Let \Cref{as:loss} and \ref{as:noise} be fulfilled.}
Let $\hat{\bh}_\mu$ be the regularized M-estimator with a ridge penalty $\mu\|\bh\|^2/2$ defined in \eqref{eq:df_h_mu}. 
If $\mu = n^{-c}$ for some $c\in (0, 1/4]$, we have
\begin{equation}
    \label{eq:key_mu}
    \frac{p}{n} \|\bm\psi_\mu\|^2 -  \frac{\tr[\bm V_\mu]^2}{n} \|\hat{\bm h}_\mu\|^2= O_P(n^{1-c}),
\end{equation}
where $\bm\psi_\mu$ is the vector $\psi(\bm\epsilon - \bm G \hat{\bm h}_\mu)\in \R^n$ and 
$\bV_\mu$ is the Jacobian matrix $(\partial/\partial \bm{\epsilon})\psi(\bm{\epsilon}-\bG\hat{\bh}_\mu)\in \R^{n\times n}$.
\end{lemma}
The proof is given in \Cref{proof_lemma_1}.
It is based on the following moment inequality. 

\begin{theorem}[Theorem 7.2 in \cite{bellec2023out}]\label{th:chi_square}
Assume that $\bG = (g_{ij})_{ij} \in \mathbb{R}^{n\times p}$ has i.i.d $\mathcal{N}(0,1)$ entries, 
and that $\bm f: \mathbb{R}^{n\times p} \to \mathbb{R}^n$ is a differentiable function with 
$\|\bm f\|^2 \leq 1$. Then, there exists an absolute constant $C>0$ such that
\begin{align}
\E
\Bigl[\Big|p \|\bm f\|^2 - \sum_{j=1}^p (\bm f^\top \bG \bm{e}_j - \sum_{i=1}^n\frac{\partial f_i}{\partial g_{ij}})^2 \Big|\Bigr]
\leq C \Bigl[ \sqrt{p} \E [(1+ \sum_{ij} \|\frac{\partial \bm f}{\partial g_{ij}}\|^2)^{1/2}
]+ \E [\sum_{ij} \|\frac{\partial \bm f}{\partial g_{ij}}\|^2]\Bigr]\label{eq:chi_square}
\end{align}
\end{theorem}
Above, we have used $\sum_{ij}=\sum_{i=1}^n\sum_{j=1}^p$ for brevity.
The proof is based on the Gaussian Poincaré inequality (cf. \cite[Theorem 3.20]{boucheron2013concentration}) and the second order Stein's formula 
(see \cite{bellec2021second} or \Cref{th:second}). 

Below, we describe a sketch of proof of \Cref{lm:str_smoothed}. To begin with, we take 
$\bm f= \bm \psi_\mu/(n^{1/2} \|\psi\|_{\infty})$ in \Cref{th:chi_square} so that 
$\|\bm f\|^2 \leq 1$.  
By the derivative formula in \Cref{th:diff_smoothed} with $\|\hat{\bm A}_\mu\|_{op} = \|(\bm G^\top \diag\{\psi'(\bm r)\} \bm G + n \mu I_p)^{-1}\|_{op} \le (n\mu)^{-1} = n^{-1+c}$, we can show that the right-hand side  of \eqref{eq:chi_square} 
is negligible compared to the left-hand side. As a consequence, we have
\begin{align*}
     \frac{p}{n}\|\bpsi_\mu\|^2 \approx \frac{1}{n} \sum_{j=1}^p \Bigl(
        \bpsi_\mu \bG^\top \bm e_j - \sum_{i=1}^n \bm e_i^\top \frac{\partial \bpsi_\mu}{\partial g_{ij}}
    \Bigr)^2.
\end{align*}
Using \Cref{th:diff_smoothed} and the \editline{KKT} condition $\bG^\top \bpsi_\mu = n\mu \hat{\bbh}_\mu$, the terms inside the square can be written explicitly by 
\begin{align*}
    \bpsi_\mu \bG^\top \bm e_j - \sum_{i=1}^n \bm e_i^\top \frac{\partial \bpsi_\mu}{\partial g_{ij}}
    =  \left(n\mu \hat{\bbh}_\mu^\top +
    \bpsi_\mu^\top 
    \diag{\psi'(\bm r)} \bG \hat{\bm A}_\mu  + \tr(\bV_\mu) \hat{\bbh}_\mu^\top\right) \bm e_j,
\end{align*} 
for all $j\in [p]$, 
where $\hat{\bm A}_\mu = (\bm G^\top \diag\{\psi'(\bm r)\} \bm G + n \mu I_p)^{-1}$. 
Combining the above two displays, we obtain
\begin{align*}
    \frac{p}{n} \|\bpsi_\mu\|^2
    \approx \frac{1}{n} \left\|n\mu \hat{\bbh}_\mu  +
    \hat{\bm A}_\mu^\top \bG^\top \diag{\psi'(\bm r)} \bpsi_\mu + \tr(\bV_\mu) \hat{\bbh}_\mu\right\|^2.
\end{align*}
Thanks to $\mu = n^{-c}$ and $\|\hat{\bm A}_\mu\|_{op} \leq (n\mu)^{-1}$, $\|n\mu \hat{\bbh}_\mu\|$ and 
$\|\hat{\bm A}_\mu^\top \bG^\top \diag{\psi'(\bm r)} \bpsi_\mu\|$ are negligible compared to $\tr[\bV_\mu]\hat{\bbh}_\mu$, so that we obtain
$p/n \cdot \|\bpsi_\mu\|^2 \approx n^{-1} \|\tr[\bV_\mu] \hat{\bbh}_\mu\|^2$ as desired. 

\subsection{%
\editline{Back} to the original M-estimator}
Finally, we derive the target \eqref{eq:target} from \Cref{lm:str_smoothed}. Toward that, we will show that the quantities
of smoothed M-estimator,  i.e., $(\hat{\bbh}_\mu, \bpsi_\mu, \tr[\bV_\mu])$,
are close to $(\hat{\bbh}, \bpsi, \tr[\bV])$ under $\mu=\mu_n =n^{-c}$ for some $c\in(0,1/4]$.  
More precisely, we prove the following:
\begin{align}
    &{\|\hat{\bh}_\mu-\hat{\bh}\|^2 = o_P(1)}, &&\|\hat{\bh}\|^2 = O_P(1), 
    \label{eq:h_diff}\\
    &\|\bpsi_\mu\|^2 - \|\bpsi\|^2 = {o_P}(n^{1-\frac{c}{2}}), &&\|\bpsi_\mu-\bpsi\|_2 = O_P(n^{\frac{1-c}{2}}), 
    \label{eq:psi_diff}\\ 
    &\tr[\bV_\mu]^2 - \tr[\bV]^2 = {o_P}(n^{2-\frac{c}{2}})
    %,
    {.}
    \label{eq:V_diff}
\end{align}
Now we assume that the above displays are justified. Then, combined with \Cref{lm:str_smoothed}, we have
\begin{align*}
    \frac{p}{n^2} \|\bm \psi\|^2 - \frac{\tr[\bm V]^2}{n^2}\|\hat{\bm h}\|^2 &= \left(\frac{p}{n^2} \|\bm \psi_\mu\|^2 - \frac{\tr[\bm V_\mu]^2}{n^2}\|\hat{\bm h}\|^2\right)
    +\frac{p}{n^2} (\|\bm \psi|_2^2 - \|\bm \psi_\mu\|^2)  \\
    &\qquad + \frac{\tr[\bm V_\mu]^2 - \tr[\bm V]^2}{n^2}\|\hat{\bm h}_\mu\|^2 - \frac{\tr[\bm V]^2}{n^2}(\|\hat{\bm h}\|^2 - \|\hat{\bm h}_\mu\|^2)  \\
    &= O_P(n^{-c}) + {o_P}(n^{-\frac{c}{2}}) +  {o_P}(n^{-\frac{c}{2}}) - \frac{\tr[\bm V]^2}{n^2}(\|\hat{\bm h}\|^2 - \|\hat{\bm h}_\mu\|^2)\\
    &= {o_P}(n^{-\frac{c}{2}}) -\frac{\tr[\bm V]^2}{n^2} o_P(1).
\end{align*}
Furthermore, we will argue that $n^2/\tr[\bm V]^2 = O_P(1)$, i.e.,  $\tr(\bV)/n$ is not too small (see \Cref{lm:lb_trace}).
Thus, dividing by $\tr[\bm V]^2/n^2$ on the above display, we obtain the target \eqref{eq:target}.

Below, we describe a sketch of proof of \eqref{eq:h_diff}, \eqref{eq:psi_diff}, and \eqref{eq:V_diff}. 
For \eqref{eq:h_diff}, we will show that $\|\hat{\bm h}\|^2$ and $\|\hat{\bm h}_\mu\|^2$ converge to the same finite limit
by the main theorem in \cite{thrampoulidis2018precise} and \Cref{th:system}. 
\eqref{eq:psi_diff} easily follows from the two \editline{KKT} conditions and the Lipschitz property of $\psi$. 
The most technical part is \eqref{eq:V_diff}. Thanks to \Cref{as:noise}, the noise vector $\bm{\epsilon}$ can be represented as
$$
\bm{\epsilon} = \bm{z} + \bm{\delta}, \quad \bm{z}\indep \bm{\delta}, \quad (z_i)_{i=1}^n \iid F, \quad (\delta_i)_{i=1}^n \iid \tilde{F}, 
$$
where $F$ has the density $z\mapsto\exp(-\phi(z))$.
By the chain rule, $ \bV - \bV_\mu $ can be written as 
\begin{equation}\label{eq:identity}
    \bV - \bV_\mu = \frac{\partial}{\partial \bm{\epsilon}} (\bpsi - \bpsi_\mu) = \frac{\partial\bm{z}}{\partial\bm{\epsilon}}\cdot \frac{\partial}{\partial \bm{z}} (\bpsi - \bpsi_\mu) = \bm{I}_n \cdot \frac{\partial}{\partial \bm{z}}(\bpsi-\bpsi_\mu) = \frac{\partial}{\partial \bm{z}}(\bpsi-\bpsi_\mu).     
\end{equation}
Here, by the second order Stein's formula extended to the density of the form $\exp(-\phi(z))$, we will show that $\tr[\bV]-\tr[\bV_\mu]$ concentrates around $\phi'(\bm{z})^\top (\bpsi-\bpsi_\mu)$. 
\begin{theorem}[Equation (2.15) in \cite{bellec2021second}]\label{th:second}
Assume that the random vector $\bm{z}\in\R^n$ has density $\bm{z}\mapsto \exp(-\phi(\bm{z}))$ where $\phi:\R^n\to\R$ is twice continuously differentiable function with Hessian $H(\bm z)=(\partial/\partial \bm{z})^2 \phi(\bm{z})\in \R^{n\times n}$. 
Then, for any differentiable function $\bm f: \R^n \to \R^n$, we have
\begin{align*}
    \E\Bigl[\Bigl( \frac{\partial\phi(\bm z)}{\partial \bm{z}}  \bm f(\bm{z}) - \tr[\frac{\partial \bm f(\bm z)}{\partial \bm z}]\Bigr)^2\Bigr] = \E\Bigl[\bm{f}(\bm{z})^\top \bm{H}(\bm{z}) \bm{f}(\bm{z}) + \tr[(\frac{\partial \bm f(\bm z)}{\partial \bm z})^2]\Bigr],
\end{align*}
provided that the expectation on the right-hand side exists.
\end{theorem}
Applying \Cref{th:second} with $\phi(\bm{z}) = \sum_{i=1}^n \phi(z_i)$ and $\bm{f}=\bpsi-\bpsi_\mu$,  using the identity \eqref{eq:identity} and $\bm{H}(\bm{z}) = \text{diag} \{\phi''(z_i): i\in [n]\}$, we have
$$
\E[( \phi'(\bm{z})^\top (\bm \psi_\mu - \bm \psi) - \tr[\bV-\bV_\mu])^2]= 
\E\Bigl[\sum_{i=1}^n\phi''(z_i) (\psi_i-(\psi_\mu)_i)^2 + \tr[(\bm V-\bm V_\mu)^2]\Bigr],
$$
and we can easily show that the right-hand side is $O(n)$. Hence, by the Markov inequality, we have 
$$
\tr[\bV_\mu] - \tr[\bV] = \phi'(\bm{z})^\top (\bm \psi_\mu - \bm \psi) + O_P(n^{1/2}).
$$
By the \editline{Cauchy--Schwarz} inequality, we bound $|\phi'(\bm{z})^\top (\bm \psi_\mu - \bm \psi)|$ from above as 
$$
|\phi'(\bm{z})^\top (\bm \psi_\mu - \bm \psi)| \le \|\phi'(\bm{z})\| \|\bpsi_\mu-\bpsi\| \overset{(*)}{=} O_P(n^{1/2}) \cdot O_P(n^{\frac{1-c}{2}}) = O_P(n^{1-\frac{c}{2}}), 
$$
where $(*)$ follows from $\|\phi'(\bm{z})\| = O_P(n^{1/2})$ (see \Cref{lem:phi_concentrates}) and 
$\|\bpsi-\bpsi_\mu\|=O_P(n^{1/2-c})$
by \eqref{eq:psi_diff}. 
As a consequence, we obtain
$$
\tr[\bV_\mu]-\tr[\bV] = O_P(n^{1-\frac{c}{2}}) + O_P(n^{1/2}) = O_P(n^{1-\frac{c}{2}})
$$
thanks to $c\le 1$. 
Finally, $\|\bV\|_{op}$ and $\|\bV_\mu\|_{op} \le 1$ lead to
$$|\tr[\bV_\mu]^2 - \tr[\bV]^2| \leq |\tr[\bV_\mu] + \tr[\bV]||\tr[\bV_\mu] - \tr[\bV]|\leq 2n O_P(n^{1-\frac{c}{2}}) = O_P(n^{2-\frac{c}{2}}),
$$
which is exactly \eqref{eq:V_diff}. 

\acks{
P. C. Bellec's research was partially supported by 
the National Science Foundation award DMS-1945428.}

\bibliography{reference}

\newpage
\appendix
\section*{SUPPLEMENT}

\section{Existence and uniqueness of solutions to \eqref{eq:nonlinear}}
\label{sec:details}
The following theorem from our companion paper \cite{koriyama2023analysis} guarantees the uniqueness and the existence of the solution $(\alpha, \kappa)$ to \eqref{eq:nonlinear}, with an explicit upper bound of $\alpha$. 
{We include it here for convenience}.
\begin{theorem}[Section 2 \cite{koriyama2023analysis}]\label{th:system}
    Suppose that $\gamma\in(0,1)$ and 
    $(\rho, F_\epsilon)$ satisfy the conditions
    \begin{enumerate}
        \item $\rho$ is convex, Lipschitz, and $\{0\}=\argmin_{x\in \R}\rho(x)$. 
        \item $\PP_{W\sim F_\epsilon} (W\ne 0)>0$ and \( \inf_{\lambda>0} \E_{W\sim F_\epsilon}[\dist(G, \lambda \partial \rho(W) )^2] > 1-\gamma \). 
    \end{enumerate}
    Then, the nonlinear system {of equations} \eqref{eq:nonlinear}
    has a unique solution $(\alpha, \kappa)\in \R_{>0}^2$, and this $\alpha$ is bounded from above as
    \begin{equation}\label{eq:localization}
        \alpha \le \frac{Q_{F_\epsilon}(r^2 c_\gamma)}{r c_\gamma} + \frac{b}{c_\gamma}
    \end{equation}
    where $c_\gamma$ is a positive constant depending on $\gamma$ only, 
    $Q_{F_\epsilon} (x):=\inf\{q>0: \PP_{W\sim F_\epsilon}(|W|\ge q)\le x\}$, and $r\in (0, 1]$ and $b\ge 0$ are any constant such that the coercivity condition
$$
\|\rho_\lambda\|_{\lip}^{-1}(\rho(x)-\rho(0)) \ge r (|x|-b) \quad \text{for all $x\in\R$}
$$
is satisfied.
% and 
%     $(r,b)$ are any constants such that
%     $$
%     r\in(0,1], \quad b\ge 0, \quad \text{and} \quad \|\rho\|_{\lip}^{-1}\bigl(\rho(u) - \rho(0)\bigr) \ge r (|u| - b) \quad \text{for all $u\in\R$}.
%     $$ 
\end{theorem}
The upper bound \eqref{eq:localization} is helpful for \Cref{sec:tuning}, in which we will consider the scaled loss $\rho_\lambda:\R\to\R$, $x\mapsto \lambda^2\rho(x/\lambda)$ for some scale parameter $\lambda>0$. 
Let $\alpha(\lambda)$ be the solution to the nonlinear system \eqref{eq:nonlinear} with $\rho=\rho_\lambda$. Then, 
\eqref{eq:localization} with $\rho=\rho_\lambda$ gives the following
explicit upper bound of $\alpha(\lambda)$ with respect to $\lambda>0$:
$$
\forall \lambda>0, \quad \alpha^2(\lambda) \le \C(\rho, F_\epsilon, \gamma) (\lambda^2+1).
$$
See \Cref{prop:scaled_control} and its proof for its derivation. We observe that the upper bound explodes as $\lambda\to+\infty$, which reflects the fact that the out-of-sample error of the ordinary least square (OLS) estimator is unbounded when the noise distribution has no second moment.  

\tk{
\begin{remark}\label{rm:how_to_solve_system}
    Suppose that the loss $\rho$ is the Huber loss $\rho_\lambda$, parametrized by  a scale parameter $\lambda>0$. This loss function was introduced in \eqref{eq:huber_scaled_definition} and used for numerical simulations in \Cref{sec:numeric} and \Cref{sec:additional_simulation}. For convenience, we recall its definition:
\begin{align*}
    \rho_\lambda (x) := \left\{
        \begin{array}{ll}
            x^2/2 & |x| \le \lambda\\
            \lambda|x|-\lambda^2/2 & |x|>\lambda
        \end{array}\right.   
\end{align*}
In this case, the proximal operator of $\rho=\rho_{\lambda}$ takes a closed form and satisfies 
\begin{align*}
    x-\prox[\kappa\rho_{\lambda}](x) = \kappa\lambda \zeta\Bigl(\frac{x}{\lambda(1+\kappa)}\Bigr) \quad \text{where} \quad \zeta(u) : = \begin{cases}
        u & |u|\le 1\\
        \sign(u) & |u|>1
    \end{cases}   
\end{align*}
for all $\kappa>0$ and $x\in\R$. Thus, the nonlinear system \eqref{eq:nonlinear} with $\rho=\rho_\lambda$ can be written as 
\begin{align*}
        \alpha^2\gamma  - \kappa^2\lambda^2 \E\Bigl[\zeta\Bigl(\frac{aG+W}{\lambda(1+\kappa)}\Bigr)^2\Bigr] = 0, \quad \alpha\gamma - \kappa \lambda \E \Bigl[\zeta\Bigl(\frac{aG+W}{\lambda(1+\kappa)}\Bigr) \cdot Z\Bigr] = 0
\end{align*}
with positive unknown $(\alpha, \kappa)$. In the numerical simulations presented  in \Cref{sec:numeric} and \Cref{sec:additional_simulation}, this system was solved using the solver $\mathsf{scipy.optimize.fsolve}$ from scipy \citep{2020SciPy-NMeth}.
\end{remark}
}

\section{Proof of \Cref{th:ofs}}\label{sec:proof_ofs}
\subsection{Proof of \Cref{lm:str_smoothed}}
\label{proof_lemma_1}
The claim of \Cref{lm:str_smoothed} is recalled here for convenience:
$$
\frac{p}{n} \|\bm\psi_\mu\|^2 -  \frac{\tr[\bm V_\mu]^2}{n} \|\hat{\bm h}_\mu\|^2= O_P(n^{1-c}) \text{ under $\mu=n^{-c}$ for all $c\in (0,1/4]$. }
$$
To prove this, we use a variant of \Cref{th:chi_square}:
\begin{lemma}[Theorem 7.4 in \cite{bellec2022derivatives}]\label{lm:chi_square_bound}
    Let $h: \mathbb{R}^{n\times p} \to \mathbb{R}^p$ and $\psi : \mathbb{R}^{n\times p} \to \mathbb{R}^n$ be locally Lipschitz functions. Suppose $\bm G\in \mathbb{R}^{n\times p}$ has i.i.d $\mathcal{N}(0,1)$ entries, and we denote $h(\bm G)$ by $\bm h$ and $\psi(\bm G)$ by $\bm\psi$. Then, we have
\begin{align}
        \mathbb{E} \Bigl[
        \frac{|\frac{p}{n} \|{\bm \psi}\|^2 - \frac{1}{n}\sum_{j=1}^p ({\bm \psi}^\top \bm G \bm{e}_j - \sum_{i=1}^n \frac{\partial {\psi}_i}{\partial g_{ij}})^2|}{
        \|{\bm h}\|^2 + \|{\bm \psi}\|^2/n
        }
        \Bigr] \leq C (\sqrt{n+p}(1+\Xi^{1/2}) + \Xi), \label{eq:chi_square_bound}
\end{align}
where 
$\Xi = \mathbb{E}[
    \frac{1}{\|{\bm h}\|^2 + \|{\bm \psi}\|^2/n} \sum_{i=1}^n \sum_{j=1}^p (
    \|\frac{\partial {\bm h}}{\partial g_{ij}}\|^2 + \frac{1}{n}\|\frac{\partial {\bm \psi}}{\partial g_{ij}}\|^2
    )]$
\end{lemma}
\Cref{lm:chi_square_bound} is obtained from \Cref{th:chi_square} with 
$\bm f=  n^{-1/2} \bm \psi / (\|\bpsi\|/n + \|\bbh\|)$. 
Below, we prove \Cref{lm:str_smoothed} using \Cref{lm:chi_square_bound} with 
$(\bpsi, \bbh) = (\bpsi_\mu, \bbh_\mu) = (\psi(\bm\epsilon - \bG \hat{\bbh}_\mu), \hat{\bbh}_\mu)$. 
First, we bound $\Xi$ in \eqref{eq:chi_square_bound}. 
By the derivative formula in \Cref{th:diff_smoothed}, the derivatives inside the expectation can be written as
\begin{align*}
    \sum_{i=1}^n \sum_{j=1}^p \|\frac{\partial \hat{\bm h}_\mu}{\partial g_{ij}}\|^2 &=  \sum_{i=1}^n \sum_{j=1}^p\|
    \hat{\bm A}_\mu\bm{e}_j \bm{e}_i^\top \bm\psi_\mu - \hat{\bm A}_\mu\bm G^\top \bm D_\mu \bm{e}_i \bm{e}_j^\top \hat{\bm h}_\mu
 \|^2, \\
 \sum_{i=1}^n \sum_{j=1}^p \|
    \frac{\partial \hat{\bm\psi}_\mu}{\partial g_{ij}} 
    \|^2 &= \sum_{i=1}^n \sum_{j=1}^p \|-\bm D_{\mu} \bm G \hat{\bm A}_\mu \bm{e}_j \bm{e}_i^\top \bm \psi_\mu - \bm{V}_\mu \bm{e}_i \bm{e}_j^\top \hat{\bm h}_\mu\|^2,
\end{align*}
where 
$\bm D_\mu = \diag\{\psi' (\bm \epsilon - \bm G \hat{\bm h}_\mu)\}$, $\hat{\bm A}_\mu = (\bm G^\top \diag\{\psi'(\bm r)\} \bm G + n \mu I_p)^{-1}$, and $\bm V_\mu = (\partial/\partial \bm{\epsilon}) \psi(\bm \epsilon- \bm G \hat{\bm h}_\mu)$.
Note that the operator norm of these matrices is bounded from above as 
\begin{align*}
    \|\hat{\bm{A}}_\mu\|_{op} &\le (n\mu)^{-1} && \text{ by $\psi'(r_i) \ge 0$ for all $i\in[n]$}\\ 
    \|\bm{D}_\mu\|_{op} &\le 1 && \text{ by $\|\psi\|_{\lip}=1$}\\
    \|\bV_\mu\|_{op} &\le 1 && \text{ by \eqref{eq:df_Vmu}}
\end{align*}
Then, $\frac{1}{2}\sum_{ij} \|\frac{\partial \hat{\bm h}_\mu}{\partial g_{ij}} \|^2$ can be bounded from above as
\begin{align*}
    \frac{1}{2}\sum_{i=1}^n \sum_{j=1}^p \|\frac{\partial \hat{\bm h}_\mu}{\partial g_{ij}}\|^2
    &\leq \sum_{ij} \|
    \hat{\bm A}_\mu\bm{e}_j\bm{e}_i^\top \bm\psi_\mu
    \|^2 + \|\hat{\bm A}_\mu\bm G^\top  \bm D_\mu \bm{e}_i 
    \bm{e}_j^\top \hat{\bm h}_\mu\|^2
    &&\text{ by $(a+b)^2\le2(a^2+b^2)$}
    \\
    &= \|\hat{\bm A}_\mu\|_F^2 \|\bm \psi_\mu\|^2 + \|\hat{\bm A}_\mu\bm G^\top \bm D_\mu \|_F^2 \|\hat{\bm h}_\mu\|^2 
    &&\text{ by $\textstyle\sum_i \|\bm M e_i\|^2=\|\bm M\|_F^2$}
    \\
    &\leq \|\hat{\bm A}_\mu\|_F^2 ( \|\bm\psi_\mu\|^2 +  \|\bm{G}^\top \bm{D}_\mu\|_{op}^2 \|\hat{\bm h}_\mu\|^2)
    && \text{ by $\|\hat{\bm{A}}_\mu \bm{G}^\top \bm{D}_\mu\|_F^2 \le \|\hat{\bm{A}}_\mu\|_F^2 \|\bm{G}^\top \bm{D}_\mu\|_{op}^2$ }
    \\
    &\leq \|\hat{\bm A}_\mu\|_F^2 ( \|\bm\psi_\mu\|^2 +  \|\bm{G}\|_{op}^2 \|\hat{\bm h}_\mu\|^2) && 
    \text{ by $\|\bm{D}_\mu\|_{op}\le 1$}
    \\
    &\leq \|\hat{\bm A}_\mu\|_F^2 (n + \|\bm G\|_{op}^2)( \|\bm\psi_\mu\|^2/n + \|\hat{\bm h}_\mu\|^2)\\
    &\leq p (n\mu)^{-2} (n + \|\bm G\|_{op}^2)( \|\bm\psi_\mu\|^2/n + \|\hat{\bm h}_\mu\|^2) &&\text{ by $\|\hat{\bm{A}}_\mu\|_{F}^2 \le p\|\hat{\bm{A}}_\mu\|_{op}^2 \le p (n\mu)^{-2}$}.
\end{align*}
By the same algebra, we bound $\frac{1}{2n}\sum_{ij} \|
\frac{\partial \hat{\bm\psi}_\mu}{\partial g_{ij}} 
\|^2$ from above as 
\begin{align*}
\frac{1}{2n} \sum_{i=1}^n \sum_{j=1}^p\|
        \frac{\partial \hat{\bm\psi}_\mu}{\partial g_{ij}} 
        \|^2 &\leq 
        \frac{1}{n} \sum_{ij} \Bigl(\|\bm D_{\mu} \bm G \hat{\bm A}_\mu \bm{e}_j \bm{e}_i^\top \bm \psi_\mu\|^2 + \|\bm{V}_\mu \bm{e}_i \bm{e}_j^\top \hat{\bm h}_\mu\|^2\Bigr)
        \\
        &=  n^{-1} \|\bm D_\mu \bm G \hat{\bm A}_\mu \|_{F}^2 \|\bm\psi_\mu\|^2 +  n^{-1} \|\bm V_\mu\|_{F}^2 \|\hat{\bm h}_\mu\|^2\\
        &\leq  \|\hat{\bm A}_{\mu}\|_F^2 \|\bm G\|_{op}^2 \|\bm\psi_\mu\|^2/n +  \|\hat{\bm h}_\mu\|^2 \\
        &\leq  (\|\hat{\bm A}_{\mu}\|_F^2 \|\bm G\|_{op}^2  + 1) (\|\bm\psi_\mu\|^2/n + \|\hat{\bm h}_\mu\|^2)\\
        &\leq  (p (n\mu)^{-2} \|\bm G\|_{op}^2  + 1) (\|\bm\psi_\mu\|^2/n + \|\hat{\bh}_\mu\|^2)
\end{align*}
By the above displays, we obtain the upper bound of $\Xi = \mathbb{E}[
    \frac{1}{\|{\bm h}\|^2 + \|{\bm \psi}\|^2/n} \sum_{ij} (
    \|\frac{\partial {\bm h}}{\partial g_{ij}}\|^2 + \frac{1}{n}\|\frac{\partial {\bm \psi}}{\partial g_{ij}}\|^2
    )]$:
\begin{align*}
    \Xi \leq \mathbb{E}\left[
     2 p(n\mu)^{-2} (n + \|\bm G\|_{op}^2) + 2 (p(n\mu)^{-2} \|\bm G\|_{op}^2  + 1)
    \right]
    \overset{(*)}{=} O(n^{2c}),
\end{align*}
where $(*)$ follows from $\mu=n^{-c}$ with $c>0$ and $\E[\|\bG\|_{op}^2] = O(n)$ for the matrix $\bG\in \R^{n\times p}$ with iid standard normal entries as $\editline{p/n}\to \gamma \in (0,\infty)$. 
Substituting this bound to the right-hand side of \Cref{lm:chi_square_bound}, we have
\begin{align}
    \mathbb{E} \Bigl[
    \frac{|\frac{p}{n} \|\bm \psi_\mu \|^2 - \frac{1}{n}\sum_{j=1}^p (\bm \psi_\mu^\top \bm G e_j - \sum_{i=1}^n \frac{\partial \bm{e}_i^\top \bm \psi_\mu}{\partial g_{ij}})^2|}{
     \|\bm\psi_\mu\|^2/n + \|\hat{\bm h}_\mu\|^2
    }
    \Bigr]
    \le C (\sqrt{n+p}(1+\Xi^{1/2}) + \Xi) 
    \overset{(**)}{=} O(n^{c+1/2}),
    \label{eq:xi_bound}
\end{align}
where we have used $0<c<1/2$ and $\Xi = O(n^{2c})$ for $(**)$.  
It remains to bound the error $\sum_{j=1}^p (\bm \psi_\mu^\top \bm G e_j - \sum_{i=1}^n  (\partial/\partial g_{ij})  \bm{e}_i^\top \bm \psi_\mu )^2-\| \tr[\bm V_\mu] \hat{\bm h}_\mu\|^2$ inside the expectation on the left-hand side. 
Now the derivatives formula and the  \editline{KKT} condition $\bm G^\top \bm{\psi}_\mu  = n\mu \hat{\bm h}_\mu$ yield
\begin{align*}
\sum_{j=1}^p (\bm \psi_\mu^\top \bm G e_j - \sum_{i=1}^n \frac{\partial \bm{e}_i^\top \bm \psi_\mu}{\partial g_{ij}})^2 
=\| n\mu \hat{\bbh}_\mu + \hat{\bm A}_\mu^\top \bm G^\top \bm D_\mu \bm\psi_\mu + \tr[\bm V_\mu] \hat{\bm h}_\mu\|^2,
\end{align*}
so that we have
\begin{align*}
& \frac{1}{2} \Bigl| \sum_{j=1}^p (\bm \psi_\mu^\top \bm G e_j - \sum_{i=1}^n \frac{\partial \bm{e}_i^\top \bm \psi_\mu}{\partial g_{ij}})^2 - \| \tr[\bm V_\mu] \hat{\bm h}_\mu\|^2\Bigr|\\
&= 2^{-1} \| n\mu \hat{\bm h}_\mu + \hat{\bm A}_\mu^\top \bm G^\top \bm D_\mu \bm\psi_\mu + \tr[\bm V_\mu] \hat{\bm h}_\mu\|^2 - \|\tr[\bm V_\mu] \hat{\bm h}_\mu\|^2 |\\
 &\leq 2^{-1} \|n\mu \hat{\bm h}_\mu + \hat{\bm A}_\mu \bm G^\top \bm D_\mu \bm \psi_\mu \|^2 + 
\|n\mu \hat{\bm h}_\mu + \bm A_\mu \bm G^\top \bm D_\mu \bm \psi_\mu \| \|\tr[\bm V_\mu] \hat{\bm h}_\mu\| && \text{ by $|(a+b)^2-b^2| \le a^2 + 2|ab|$}\\
 &\leq   n^{2} \mu^2 \|\hat{\bm h}_\mu\|^2 +  \| \bm A_\mu \bm G^\top \bm D_\mu \bm \psi_\mu\|^2 + |\tr[\bV_\mu]|\|n\mu \hat{\bm h}_\mu + \bm A_\mu \bm G^\top \bm D_\mu \bm \psi_\mu \| \|\hat{\bm h}_\mu\|  && \text{ by $2^{-1}(a+b)^2 \le a^2 + b^2$}\\
 & \le  n^{2} \mu^2 \|\hat{\bm h}_\mu\|^2 +  \| \bm A_\mu \bm G^\top \bm D_\mu \bm \psi_\mu\|^2 + n \|n\mu \hat{\bm h}_\mu + \bm A_\mu \bm G^\top \bm D_\mu \bm \psi_\mu \| \|\hat{\bm h}_\mu\| && \text{ by $|\tr[\bV_\mu]| \le n \|\bV_\mu\|_{op} \le n$}\\
 & \leq  n^{2} \mu^2 \|\hat{\bm h}_\mu\|^2 + \|\hat{\bm{A}}_\mu \bG^\top \bm{D}_\mu \bm \psi_\mu\|^2 + n^{2}\mu \|\hat{\bm h}_\mu\|^2 + 
 \|\hat{\bm{A}}_\mu \bG^\top \bm{D}_\mu \bm \psi_\mu\|\|\hat{\bh}_\mu\| && \text{ by triangle inequality}\\
 & \le (n^{2} \mu^2 + n^2\mu + 2^{-1}) \|\hat{\bm h}_\mu\|^2 + (1+2^{-1})\|\hat{\bm{A}}_\mu \bG^\top \bm{D}_\mu \bm \psi_\mu\|^2 && \text{ by $ab \le 2^{-1}(a^2 + b^2)$}\\
 & \le (\|\hat{\bm h}_\mu\|^2 + n^{-1} \|\bm \psi_\mu\|^2)
 (n^2 \mu^2 + n^2 \mu + 2^{-1} +  2^{-1} 3 n\|\hat{\bm{A}}_\mu \bG^\top \bm{D}_\mu\|_{op}^2) \\
 &\le (\|\hat{\bm h}_\mu\|^2 +n^{-1} \|\bm \psi_\mu\|^2)
 (n^2 \mu^2 + n^2\mu + 2^{-1} +  2^{-1}3 n (n\mu)^{-2} \|\bG\|_{op}^2) && \text{ by $\|\hat{\bm{A}}_\mu\|_{op} \le (n\mu)^{-1}$}\\
 & && \text{ and $\|\bm{D}_\mu\|_{op} \le 1$}.
\end{align*}
Thus, we have
\begin{align*}
    &\E\Bigl[
    \frac{\frac{1}{n}\left|\sum_{j=1}^p (\bm \psi_\mu^\top \bm G e_j - \sum_{i=1}^n \frac{\partial \bm{e}_i^\top \bm \psi_\mu}{\partial g_{ij}})^2  - \| \tr[\bm V_\mu] \hat{\bm h}_\mu\|^2 \right|}{\|\bm\psi_\mu\|^2/n + \|\hat{\bm h}_\mu\|^2}\Bigr]\\
    &\le \C \E\bigl[n \mu^2 + n \mu + n^{-1} + (n\mu)^{-2} \|\bG\|_{op}^2\bigr] \\
    &= O(n^{1-2c} + n^{1-c} + n^{-1} + n^{-1+2c}) &&\text{ by $\mu=n^{-c}$ and $\E[\|\bG\|_{op}^2] = O(n)$, } \\
    &=O(n^{1-c}) && \text{ by $0 < c \le 1/4 < 2/3$}
\end{align*}
Combined with \eqref{eq:xi_bound}, by the triangle inequality, 
\begin{align}
    \label{UNION_1}
    \mathbb{E} \left[
    \frac{|\frac{p}{n} \|\bm \psi_\mu \|^2 - \frac{1}{n}
    \|\tr[\bm V_\mu] \hat{\bm h}_\mu\|^2
    |}{
    \|\hat{\bm h}_\mu\|^2 + \|\bm \psi_\mu\|^2/n
    }
\right] &= O(n^{c+1/2}) + O(n^{1-c}) = O(n^{1-c})
\end{align}
by $c\in (0,1/4]$. Thus, we obtain
\begin{align*}
    \Bigl|\frac{p}{n} \|\bm \psi_\mu \|^2 - \frac{1}{n}
    \|\tr[\bm V_\mu] \hat{\bm h}_\mu\|^2
    \Bigr| =   \Bigl(\|\hat{\bm h}_\mu\|^2 + \frac{\|\bm \psi_\mu\|^2}{n}\Bigr)
    O_P(n^{1-c}) \le (\|\hat{\bh}_\mu\|^2 + \|\psi\|_{\infty}^2) O_P(n^{1-c}) = O_P(n^{1-c}),
\end{align*}
where we have used $\|\hat{\bh}_\mu\|_{2}^2 = O_P(1)$ (see \Cref{lm:h_diff}). This finishes the proof.  
\subsection{Convergence of smoothed quantities}
The lemmas below are the key to relating \Cref{lm:str_smoothed} to \Cref{th:ofs}.
\begin{lemma}\label{lm:h_diff}
Suppose that $(\rho, F_\epsilon)$ satisfies \Cref{as:loss} and \Cref{as:noise}. Then, under $\mu=n^{-c}$ for any $c>0$, we have
$$
\|\hat{\bh}\|^2\to^p \alpha^2 , \qquad \|\hat{\bh}_\mu\|^2 \to^p \alpha^2, \qquad \alpha^2 <+\infty, {\qquad \|\hat{\bh}-\hat{\bh}_\mu\| = o_p(1)}
$$
where $\alpha$ is the unique solution to the nonlinear system \eqref{eq:nonlinear}. 
\end{lemma}

\begin{proof}
The definition of $\hat\bh$ and $\hat{\bh}_\mu$ are recalled here for convenience:
$$
\hat{\bh} \in \argmin_{\bh\in\R^p} \frac{1}{n} \sum_{i=1}^n\rho(\epsilon_i-\bg_i^\top\bh), \quad \hat{\bh}_\mu \in  \argmin_{\bh\in\R^p} \frac{1}{n} \sum_{i=1}^n\rho(\epsilon_i-\bg_i^\top\bh) + \frac{\mu\|\bh\|^2}{2},
$$
where $(\epsilon_i)_{i=1}^n \iid F_\epsilon$, $(\bg_i)_{i=1}^n \iid \N(\bm{0}_p, \bm{I}_p)$, and $(\epsilon_i)_{i=1}^n\indep (\bg_i)_{i=1}^n$. 
Since \Cref{as:loss}-\ref{as:noise} {imply} the assumption in \Cref{th:system} (see \Cref{rm:system_condition_satisfied}), the nonlinear system of equations \eqref{eq:nonlinear} has {a} unique solution
$\alpha$ and $\|\hat{\bh}\|^2\to^p \alpha^2$. It remains to show $\|\hat{\bh}_\mu\|^2\to^p\alpha^2$ under the diminishing regularization parameter $\mu=n^{-c}$. 
Define 
$$
f:\qquad
\R^p\to \R,
\qquad \bm{x}\mapsto n^{-c}\|\bm x\|^2/2
$$ so that $\hat{\bh}_\mu$ is the regularized M-estimator with the Lipschitz convex loss $\rho$ and penalty $f$. Suppose the following condition is satisfied:
\begin{align}\label{eq:moreau_condition}
    \forall a \in \R, \forall\tau>0, \quad  n^{-1} (e_{f} (a\bm g; \tau) - f(\bm 0_p)) \to^p 0, \text{ where } \bm g\sim \mathcal{N}(\bm 0_p, \bm I_p). 
\end{align}
where $e_f(x;\tau) = \argmin_{u}\frac{1}{2{\tau}}(x-u)^2 + f(u)$ is the Moreau envelope of $f$. 
{
Let \( L(c, \tau) := \E[e_\rho(cZ+W; \tau)-\rho(W)]\)
and assume \( Z\sim N(0,1), \ W\sim F_\epsilon, \ Z\indep W \).
}
Then, by \cite[Theorem 3.1]{thrampoulidis2018precise} (with $F=0$ and 
{taking for instance the signal distribution to be a single point mass at 0}), if the convex-concave optimization
\begin{equation}
\inf_{\alpha\ge 0, \tau_g>0} \sup_{\beta\ge 0, \tau_h>0} \frac{\beta\tau_g}{2} + \frac{1}{\gamma} \cdot L(\alpha, \frac{\tau_g}{\beta}) - \frac{\alpha \tau_h}{2} - \frac{\alpha^2\beta^2}{2\tau_h},
\label{CGMT}
\end{equation}
admits a unique solution $\alpha_*$, we have $\|\hat{\bh}_\mu\|^2 \to^p \alpha_*^2$. Note that the stationary condition of the above convex-concave optimization is the nonlinear system of equations \eqref{eq:nonlinear} (cf. \cite[Section 5.1]{thrampoulidis2018precise}), while \Cref{th:system} guarantees the uniqueness and existence of the solution to \eqref{eq:nonlinear}. Therefore, we obtain $\|\hat{\bm h}_\mu\|_2^2\to^p\alpha_*^2$ 
{ where \( \alpha_* \) is the unique solution to \eqref{eq:nonlinear}
(we add a star to denote the unique solution here, to avoid confusion
with the variable \( \alpha \) of the minimization in \eqref{CGMT}).}
It remains to show \eqref{eq:moreau_condition}. Note that the Moreau envelope $e_{f}(\bm x;\tau)$ of the convex function $\bx\in \R^p \mapsto n^{-c}\|\bx\|^2/2$ is given by 
\begin{align*}
   \forall \bx\in\R^p, \forall\tau>0, \quad e_{f}(\bm x;\tau) = \min_{\bm u\in \R^p} \frac{\|\bm x-\bm u\|^2}{2\tau} + \frac{n^{-c}}{2}\|\bm u\|^2
    =  \frac{n^{-c}}{2(1+n^{-c}\tau)} \|\bm x\|^2,
\end{align*}
so that for all $a\in \R$ and $\tau>0$, we have
\begin{align*}
\frac{e_{f_\mu} (a\bm g; \tau) - f_\mu(\bm 0_p)}{n}
      = \frac{n^{-c}}{1+n^{-c}\tau} \frac{a^2 \|\bm g\|^2}{n} = \frac{n^{-c}}{1+n^{-c}\tau}  \frac{a^2}{n} \sum_{i=1}^n g_i^2  \to^p 0 \cdot a^2 = 0,
\end{align*}
where we have used the weak law of large number to the iid sum $n^{-1} \sum_{i=1}^n g_i^2$ with $\E[g_i^2]=1$. This finishes the proof of \eqref{eq:moreau_condition}. 

It remains to show $\|\hat{\bh} - \hat{\bh}\| = o_p(1)$. Now we verify $\|\hat{\bh}-\hat{\bh}_\mu\|^2 \le\|\hat{\bh}\|^2 - \|\hat{\bh}_\mu\|^2$. Letting $\mathcal{L}:\R^p\to\R$ be the convex function $\mathcal{L}(\bh) := \sum_{i=1}^n \rho(\epsilon-\bm{g}_i^\top \bh)$
then $\hat{\bh}$ and $\hat{\bh}_\mu$ solve
$$
\hat{\bh} \in \argmin_{\bh\in\R^p} \mathcal{L}(\bh), \qquad \hat{\bh}_\mu \in\argmin_{\bh\in\R^p} \mathcal{L}(\bh) + \frac{\mu}{2}\|\bh\|^2 
$$
Note in passing that $\hat{\bh}_\mu$ is also a minimizer of the convex function $\mathcal{F}:\R^p\to\R$
$$
\hat{\bh}_\mu \in \argmin_{\bh\in\R^p} \mathcal{F}(\bh)\quad  \text{where}\quad \mathcal{F}(\bh):= \mathcal{L}(\bh) + \frac{\mu}{2}\Bigl(\|\bh\|^2 - \|\bh-\hat{\bh}_\mu\|^2\Bigr), 
$$
since the gradient of $\bh\mapsto \|\bh-\hat{\bh}_\mu\|\editline{^2}$ is $\bm{0}_p$ at $\hat{\bh}_\mu$ and $\hat{\bh}_\mu$ satisfies the new \editline{KKT} condition $\nabla\mathcal{L}(\hat{\bh}_\mu) = \bm{0}_p$. Then, $\mathcal{F}(\hat{\bh}_\mu) \le \mathcal{F}(\hat{\bh})$ and $\mathcal{L}(\hat
\bh) \le \mathcal{L}(\hat{\bh}_\mu)$ yield
\begin{align*}
    0 \ge \mathcal{F}(\hat{\bh}_\mu) - \mathcal{F}(\hat{\bh}) &= \frac{\mu}{2} \bigl(
        \|\hat{\bh}_\mu\|^2 - \|\hat{\bh}\|^2 + \|\hat{\bh}-\hat{\bh}_\mu\|^2
    \bigr) + \mathcal{L}(\hat{\bh}_\mu) - \mathcal{L}(\hat{\bh}) \\
    &\ge \frac{\mu}{2} \bigl(
        \|\hat{\bh}_\mu\|^2 - \|\hat{\bh}\|^2 + \|\hat{\bh}-\hat{\bh}_\mu\|^2
    \bigr) + 0,
\end{align*}
so that $\|\hat{\bh}-\hat{\bh}_\mu\|^2 \le  \|\hat{\bh}\|^2 - \|\hat{\bh}_\mu\|^2.$ Combined with $\|\hat{\bh}\|^2, \|\hat{\bh}_\mu\|^2 \to^p \alpha^2<+\infty$, we have $\|\hat{\bh}-\hat{\bh}_\mu\|^2 = o_p(1)$ and conclude the proof. 
\end{proof}

\begin{lemma}\label{lm:psi_diff}
    Suppose that $(\rho, F_\epsilon)$ satisfies \Cref{as:loss}-\ref{as:noise}. Then, under $\mu=n^{-c}$ for some $c>0$, we have
    $$
    \|\bpsi_\mu - \bpsi\|_2 = {o_P}(n^{\frac{1-c}{2}}) \quad \text{ and } \quad 
    \|\bpsi\|^2 - \|\bpsi_\mu\|^2 = {o_P}(n^{1-\frac{c}{2}}),
    $$
    where $\bpsi = \psi(\bm\epsilon-\bG\hat{\bbh})$ and
    $\bpsi_\mu = \psi(\bm\epsilon-\bG \hat{\bbh}_\mu)$.
\end{lemma}

\begin{proof}
    Note that $\bG^\top \bpsi_\mu=n\mu \hat{\bbh}_\mu$ and $\bG^\top \bpsi = \bm 0_p$  by the \editline{KKT} conditions. Then, using \Cref{lm:psi_increasing} that is introduced later, $\|\bpsi_\mu-\bpsi\|^2$ can be bounded from above as 
    \begin{align}
        \|\bpsi_\mu-\bpsi\|^2 &\le (\bpsi_\mu-\bpsi)^\top (\bG\hat{\bh} - \bG\hat{\bh}_\mu) &&\text{ by \Cref{lm:psi_increasing} with $\bm{u} = \bm{\epsilon}-\bm{G}\hat{\bh}_\mu$ and $\bv = \bm{\epsilon}-\bm{G}\hat{\bh}$}
        \nonumber
        \\
        &\le (n\mu \hat{\bh}_\mu - \bm{0}_p)^\top (\hat{\bh}-\hat{\bh}_\mu) && \text{ by $\bG^\top \bpsi_\mu=n\mu \hat{\bbh}_\mu$ and $\bG^\top \bpsi = \bm 0_p$} 
        \nonumber
        \\
        &\le n \mu \|\hat{\bh}_\mu\| \|\hat{\bh}-\hat{\bh}_\mu\|  && \text{ by the \editline{Cauchy--Schwarz} inequality} 
        \label{UNION_2}
        \\
        % &= n\mu \|\hat{\bh}_\mu\| (\|\hat{\bh}\| + \|\hat{\bh}_\mu\|) && \text{ by the triangle inequality}\\
        & = {o_P}(n^{1-c}) && \text{ by \Cref{lm:h_diff} and $\mu=n^{-c}$},
        \nonumber
    \end{align}
    which finishes the proof for $\|\bpsi_\mu-\bpsi\|^2$. For the bound of $\|\bpsi_\mu\|^2-\|\bpsi\|^2$, the 
     \editline{Cauchy--Schwarz} inequality implies 
    \begin{align*}
        |\|\bm \psi\|^2 - \|\bm \psi_\mu\|^2|= |(\bm \psi- \bm \psi_\mu)^\top (\bm \psi + \bm \psi_\mu)|
        \leq  \|\bm \psi + \bm \psi_\mu\|\|\bm \psi- \bm \psi_\mu\| 
        \leq 2\sqrt{n} \|\psi\|_{\infty} \|\bm \psi- \bm \psi_\mu\|.
    \end{align*}
    Since $\|\psi\|_{\infty} <+\infty$ by \Cref{as:loss} and we have shown that $\|\bpsi-\bpsi_\mu\|^2 = {o_P}(n^{1-c})$, we obtain $ |\|\bm \psi\|^2 - \|\bm \psi_\mu\|^2| = {o_P}(n^{1-\frac{c}{2}})$. This finishes the proof. 
\end{proof}

\begin{lemma}\label{lm:diff_V}
Suppose that $(\rho, F_\epsilon)$ satisfies \Cref{as:loss}-\ref{as:noise}. Then, under $\mu=n^{-c}$ for some $c\in(0,1)$, we have
$$
\tr[\bm V_\mu]^2 - \tr[\bm V]^2 = {o_P}(n^{2-\frac{c}{2}}),
$$
where $\bV = (\partial/\partial \bm{\epsilon})\bpsi$ and $\bV_\mu = (\partial/\partial\bm{\epsilon}) \bpsi_\mu$. 
\end{lemma}
\begin{proof}
By \Cref{as:noise}, we can write 
$\bm{\epsilon} = \bm{z} + \bm{\delta}$ where  $\bm{z}\indep \bm{\delta}$ and 
$(z_i)_{i=1}^n$ has i.i.d. density $\exp(-\phi(z))$. 
By the chain rule, $\bV-\bV_\mu$ can be written as 
$$
\bV - \bV_\mu = \frac{\partial}{\partial \bm{\epsilon}} (\bpsi - \bpsi_\mu) = \frac{\partial\bm{z}}{\partial\bm{\epsilon}}\cdot \frac{\partial}{\partial \bm{z}} (\bpsi - \bpsi_\mu) = \bm{I}_n \cdot \frac{\partial}{\partial \bm{z}}(\bpsi-\bpsi_\mu) = \frac{\partial}{\partial \bm{z}}(\bpsi-\bpsi_\mu).      
$$
Then, \Cref{th:second} applied with $\bm{f}=\bpsi-\bpsi_\mu$ yields
\begin{align}
    &\E\bigl[( \phi'(\bm{z})^\top (\bm \psi_\mu - \bm \psi) - \tr[\bV-\bV_\mu])^2\bigr]
    \label{appli_Theorem6}
    \\
    &= 
    \E \Bigl[\sum_{i=1}^n\phi''(z_i) (\psi_i-(\psi_\mu)_i)^2 + \tr[(\bm V-\bm V_\mu)^2]\Bigr] && \text{ by \Cref{th:second}}
    \nonumber
    \\
    &\le \|\phi''\|_{\infty}^2 \E[\|\bpsi-\bpsi_\mu\|^2] + n \E[\|\bV-\bV_\mu\|_{op}^2] && \text{ using $\tr[\bm{M}] \le n \|\bm{M}\|_{op}$ for $\bm{M}\in \R^{n\times n}$}
    \nonumber
    \\
    &\le 4 \|\phi''\|_{\infty}^2 \|\psi\|_{\infty}^2 n + 4n && \text{ using  $\|\bV\|_{op}, \|\bV_\mu\|_{op} \le 1$ from \eqref{eq:df_Vmu}}
    \nonumber
    \\
    &= \C(\rho, F_\epsilon) n, &&\text{ since $\|\phi''\|_{\infty}, \|\psi\|_{\infty} <+\infty$ by \Cref{as:loss}-\ref{as:noise}}, 
    \nonumber
\end{align}
so that $ \bigl| \tr[\bV-\bV_\mu] - \phi'(\bm{z})^\top (\bm \psi_\mu - \bm \psi)\bigr| = O_P(n^{1/2})$. From this bound and the triangle inequality, it follows that
\begin{align}
    \bigl|\tr[\bm V-\bm V_\mu]\bigr|  &\le \bigl|\phi'(\bm{z})^\top (\bm \psi_\mu - \bm \psi)\bigr| + \bigl| \tr[\bV-\bV_\mu] - \phi'(\bm{z})^\top (\bm \psi_\mu - \bm \psi)\bigr|
    \label{UNION_3}
    \\
    \nonumber
      &\le \|\phi'(\bm{z})\| \|\bpsi-\bpsi_\mu\| + O_P(n^{1/2}) && \text{ by the \editline{Cauchy--Schwarz}}\\
    \nonumber
    & = O_P(n^{1/2}) {o_P}(n^{\frac{1-c}{2}}) + O_P(n^{1/2}) &&\text{ by \Cref{lm:psi_diff} and \ref{lem:phi_concentrates}}\\
    &= {o_P}(n^{1-\frac{c}{2}}) && \text{ by $c {<} 1$}. 
    \nonumber
\end{align}
Finally, using $\|\bV\|_{op}, \|\bV_\mu\|_{op} \le 1$ from \eqref{eq:df_Vmu}, we have
\begin{align*}
    |\tr[\bV]^2-\tr[\bV_\mu]^2| = |\tr[\bV]-\tr[\bV_\mu]|\cdot |\tr[\bV]+\tr[\bV_\mu]| \le  {o_P}(n^{1-\frac{c}{2}}) \cdot 2n = {o_P}(n^{2-\frac{c}{2}}), 
\end{align*}
which concludes the proof. 
\end{proof}

\editline{
In order to relax the assumption on the noise and prove
\Cref{prop:relaxed},
we now provide here a modification of this argument to allow
for a vanishing smooth noise component $z_i$:
the conclusion $|\tr[\bV-\bV_\mu]|/n \to^P 0$ still holds
if the linear model noise is 
$\epsilon_i = \sigma_n z_i + \delta_i$
with vanishing $\sigma_n$ (depending on $n$)
provided that $\sigma_n \ge \sqrt \mu$
and $\sigma_n \sqrt n \to +\infty$.

\begin{lemma}[Vanishing noise component]
    \label{lemma_vanishing}
    Assume that the noise $\bm \epsilon$
    in the linear model has iid coordinates of the form
    $\epsilon_i = \sigma_n z_i + \delta_i$
    where $\delta_i$ and $z_i$ are independent 
    and $z_i$ has density $z\mapsto\exp(-\phi(z))$
    with twice-continuously differentiable $\phi$
    and $\sup_{x\in\R}|\phi''(z)|<+\infty$.
    If $\sigma_n \ge \sqrt \mu$, we have
\begin{equation}
    \label{concl_vanishing_noise}
    \E\Bigl[\frac{|\tr[\bV-\bV_\mu]|}{n}\Bigr]
\le 
\E\Bigl[
\min\Bigl(2,
    \frac{\|\phi'(\bm z)\|}{\sqrt n}
    \|\bm h_\mu\|
    \|\bm h - \bm h_\mu\|
\Bigr)
\Bigr]
+ \frac{{2}
\|\phi''\|_{\infty} \|\psi\|_{\infty}}{\sigma_n \sqrt n}
+
\frac{{2}}{\sqrt n}.
\end{equation}
\editline{
If we take $\mu=n^{-c}$ with $c\in (0,1)$ as in \Cref{lm:h_diff}, then $\|\phi'(\bm z)\|n^{-\frac1 2}
\|\bm h_\mu\|
\|\bm h - \bm h_\mu\|
\to^P 0$ by \Cref{lm:h_diff} and \Cref{lem:phi_concentrates}
for the first term,
while $\sigma_n \sqrt{n} \ge \sqrt{n^{1-c}} \to +\infty$ for the second term. This implies that $|\tr[\bV-\bV_\mu]|/n$ converges to 0 in L1 and in probability for all $\sigma_n \ge n^{-c/2}$. 
}
% Since 
% $\frac{\|\phi'(\bm z)\|}{\sqrt n}
% \|\bm h_\mu\|
% \|\bm h - \bm h_\mu\|
% \to^P 0$ by \Cref{lem:phi_concentrates} and \eqref{eq:h_diff},
% the first term in the right-hand side converges to 0.
% If additionally $\sigma \sqrt n\to+\infty$, then
% $|\tr[\bV-\bV_\mu]|/n$ converges to 0 in L1 and in probability.
\end{lemma}
\begin{proof}
    This is a modification of the argument of
    \Cref{lm:diff_V} to obtain a condition on how small
    the amplitude $\sigma_n$ of the smooth noise $z_i$ is allowed.
    Similarly to \eqref{appli_Theorem6}, by \Cref{th:second}
    we have
\begin{align*}
    \E\bigl[( \phi'(\bm{z})^\top (\frac{\bm \psi_\mu - \bm \psi}{\sigma_n}) - \tr[\bV-\bV_\mu])^2\bigr]
    &= 
    \E \Bigl[\sum_{i=1}^n\phi''(z_i) \frac{(\psi_i-(\psi_\mu)_i)^2}{\sigma_n^2} + \tr[(\bm V-\bm V_\mu)^2]\Bigr]
    \\
    &\le 4 n(\|\phi''\|_{\infty}^2 \|\psi\|_{\infty}^2/\sigma_n^2 + 1 )
\end{align*}
using  $\|\bV\|_{op}, \|\bV_\mu\|_{op} \le 1$ from \eqref{eq:df_Vmu}.
By the triangle inequality
combined with
$0\le \tr[\bm V]\le n$
and $0\le \tr[\bm V_\mu]\le n$,
$$
|\tr[\bV-\bV_\mu]|
\le 
\min\Bigl(2n,
    \frac{\|\phi'(\bm z)\| \|\bm\psi_\mu - \bm\psi\|}{\sigma_n}
\Bigr)
+
\Big|\phi'(\bm{z})^\top (\frac{\bm \psi_\mu - \bm \psi}{\sigma_n}) - \tr[\bV-\bV_\mu]
\Big|
$$
Recalling \eqref{UNION_2} we have
$\|\bm\psi_\mu - \bm\psi\|^2 \le n \mu \|\bm h_\mu\| \|\bm h - \bm h_\mu\|$
if $\sigma_n \ge \sqrt \mu$ we obtain
$$
\E\bigl[|\tr[\bV-\bV_\mu]|/n\bigr]
\le 
\E\bigl[
\min(2,
    (\|\phi'(\bm z)\|n^{-1/2})
    \|\bm h_\mu\|
    \|\bm h - \bm h_\mu\|
)
\bigr]
+
\editline{2}
(\|\phi''\|_{\infty}^2 \|\psi\|_{\infty}^2/\sigma_n^2 + 1 )^{1/2}\editline{/\sqrt{n}}.
$$
Using $\sqrt{a+b}\le \sqrt a + \sqrt b$, we obtain
\eqref{concl_vanishing_noise}.
\end{proof}
} % end \editline

\begin{lemma}\label{lm:psi_increasing}
If $\psi:\R\to\R$ is $1$-Lipschitz and nondecreasing, it holds that 
$$
\|\psi(\bm{u})-\psi(\bm{v})\|^2 \le (\psi(\bm{u})-\psi(\bm{v}))^\top (\bm{u}-\bm{v})
$$    
for all $\bm{u}, \bm{v}\in \R^n$. 
\end{lemma}
\begin{proof}
    We have $0 \le \psi(x)-\psi(y) \le (x-y)$ for all $x, y\in \R$ such that $x\le y$. Switching the role of $x$ and $y$ in this inequality, we have $(\psi(x)-\psi(y))^2 \le (\psi(x)-\psi(y))(x-y)$ for all $x, y\in \R$. Applying this inequality with $(x, y) = (u_i, v_i)$ for each $i\in [n]$, we conclude the proof. 
    % $$
    %   \|\psi(\bm{u})-\psi(\bm{v})\|^2 = \sum_{i=1}^n (\psi(u_i)-\psi(v_i))^2 \le
    %    (\psi(\bm{u})-\psi(\bm{v}))^\top (\bm{u}-\bm{v}).
    % $$
\end{proof}
\begin{lemma}\label{lem:phi_concentrates}
    Suppose $(z_i)_{i=1}^n$ has iid density $\exp(-\phi(z))$ where $\phi:\R\to\R$ is twice continuously differentiable and $\|\phi''\|_{\infty} <+\infty$. Then \editline{$\E[\phi'(z_i)^2]=\E[\phi''(z_i)]$} and
    $\lim_{n\to\infty}\PP(\|\phi'(\bm{z})\|^2 \le n(\|\phi''\|_{\infty}+1)) = 1$.
\end{lemma}
\begin{proof}
    Note $\E[\phi'(z_i)^2]=\E[\phi''(z_i)] \le \|\phi''\|_{\infty} <+\infty$ by integration part.  By the weak law of large number, we have
    $
    n^{-1} \|\phi'(\bm{z})\|^2 = n^{-1} \sum_{i=1}^n \phi'(z_i)^2 \to^p \E[\phi'(z)^2] = \E[\phi''(z_i)]
    $. Thus, we obtain
    $$
    \PP(n^{-1} \|\phi'(\bm{z})\|^2 > \|\phi''\|_{\infty}+1) \le \PP(n^{-1}\|\phi'(\bm{z})\|^2 > \E[\phi''(z_i)] + 1) \to 0,
    $$
    which completes the proof. 
\end{proof}

\subsection{Lower bound on the trace of V}
In this section, we derive a lower bound of $\tr[\bm V]/n$.
\begin{lemma}\label{lm:lb_trace}
Suppose that $(\rho, F_\epsilon)$ satisfy \Cref{as:loss}-\ref{as:noise}.  
Let $\alpha(\rho, F_\epsilon, \gamma)$ be the unique solution to the nonlinear system of equations \eqref{eq:nonlinear} for $(\rho, F_\epsilon, \gamma)$
and $\eta=\eta(\rho)$ be some positive constant such that 
$\psi(x)^2/\|\psi\|_{\infty}^2 +  \psi'(x) \geq \eta^2$ for almost every $x\in \R$ (we can always take such $\eta>0$ thanks to \editline{\Cref{as:loss}(3)}). 
Then, we have 
$$
\PP\left(
    \frac{\tr[\bV]^2}{n^2} \ge  
    \text{B}(\gamma, \rho, F_\epsilon)
\right)\to 1, \text{ where } B(\gamma, \rho, F_\epsilon) =  C(\gamma) \cdot \min(\eta(\rho)^4, \frac{\|\psi\|_\infty^2\cdot \eta(\rho)^2}{\alpha^2(\rho, F_\epsilon, \gamma)})>0. 
$$
\end{lemma}
\begin{proof}
    Recall that we have shown in 
    \Cref{lm:diff_V} that 
    $$
    \tr(\bV)^2/n^2 - \tr(\bV_\mu)^2/n^2 = {o_P (n^{-\frac{c}{2}})} = o_P(1) \text{ under $\mu=n^{-c}$ for all $c\in(0,1)$}. 
    $$
    Thus, it suffices to show the lower bound for $\tr[\bm V_\mu]^2/n^2$ for some $c\in (0,1)$, say $c=1/4$.  
    Define the matrices $\bm{D}$ and $\tilde{\bm{D}}_\mu$ as
    $$
    \bm{D}_{\mu} := \diag\{\psi'(\bm{\epsilon}-\bG \hat{\bm{h}}_\mu)\}, \quad \bm{\tilde{D}_\mu} := \|\psi\|_{\infty}^{-2} \diag\{\psi^2(\bm{\epsilon}-\bG \hat{\bm{h}}_\mu)\}.
    $$
    Thanks to $\|\psi\|_{\lip}\le 1$ and $\psi(x)^2/\|\psi\|_{\infty}^2 + \psi'(x)\ge \eta^2$, we have
    \begin{align*}
        \tr(\bm D_\mu) \leq n, \quad \tilde{\bm D}_\mu +\bm D_\mu \succeq \eta^2 \bm I_n
    \end{align*}
     where $\succeq$ is the positive semi-definite order. 
    By the derivative formula (\Cref{th:diff_smoothed}), we have
    \begin{align*}
        \bm V_\mu &= \bm D_\mu - \bm D_\mu \bm G (\bm G^\top \bm D_\mu \bm G + n\mu \bm I_p)^{-1} \bm G^\top \bm D_\mu\\
        &= \bm D_\mu^{1/2} (\bm I_n - \bm D_\mu^{1/2}\bm G (\bm G^\top \bm D_\mu \bm G + n\mu \bm I_p)^{-1} \bm G^\top \bm D_\mu^{1/2}) \bm D_\mu 
        \\
        &=\bm D_\mu^{1/2} \bm H_\mu \bm D_\mu^{1/2},
    \end{align*}
    where $\bm H_\mu := \bm I_n - \bm D_\mu^{1/2}\bm G (\bm G^\top \bm D_\mu \bm G + n\mu \bm I_p)^{-1} \bm G^\top \bm D_\mu^{1/2}$ is positive semi-definite, and satisfies $\|\bm H_\mu\|_{op} \leq 1$ and $\tr(\bm H_\mu) \geq n-p$. Then, simple algebra yields
    \begin{align*}
        \tr[\bm V_\mu] &= \tr(\bm D_\mu^{1/2} \bm H_\mu \bm D_\mu^{1/2}) \\
        &= \tr(\bm H_\mu^{1/2} \bm D_\mu \bm H_\mu^{1/2}) && \text{ by $\tr(\bm D_\mu^{1/2} \bm H_\mu \bm D_\mu^{1/2}) = \tr(\bm H_\mu \bm D_\mu)=\tr(\bm H_\mu^{1/2} \bm D_\mu \bm H_\mu^{1/2})$} \\
        &\geq \tr(\bm H_\mu^{1/2} (\eta^2 \bm I_n - \tilde{\bm D}_\mu)\bm H_\mu^{1/2}) && \text{ by $\tilde{\bm D}_\mu +\bm D_\mu \succeq \eta^2 \bm I_n$}\\
        &= \eta^2 \tr(\bm H_\mu) -\tr(\bm H_\mu^{1/2} \tilde{\bm D}_\mu \bm H_\mu^{1/2})\\
        &\geq\eta^2(n-p) -\tr(\tilde{\bm D}_\mu)  \|\bm H_\mu\|_{op} && \text{ by $\tr[\bm{H}_\mu]\ge n-p$ and $\tr[\bm{AB}] \le \tr[\bm A] \cdot \|\bm B\|_{op}$}\\
        &\geq
        \eta^2 (n-p) - \|\psi\|_{\infty}^{-2} \|\bm \psi_\mu\|^2 \cdot 1.
    \end{align*}
    Thus, under the event $\Omega:=  \{\|\bm \psi_\mu\|^2 \leq n\|\psi\|_{\infty}^2 \eta^2 (1-\gamma)/2\}$, we have
    \begin{align}\label{eq:omega}
        \bm{1}_{\Omega} \cdot n^{-2} \tr[\bm V_\mu]^2 \geq \bm{1}_{\Omega} \cdot 4^{-1} \eta^4 (1-\gamma)^2
    \end{align}
    It remains to bound $n^{-2}\tr[\bV_\mu]^2$ from below under the complement $\Omega^c$. 
    \Cref{lm:str_smoothed} implies 
    $$ n^{-2} {\tr(\bV_\mu)^2} \|\hat{\bbh}_\mu\|^2 - n^{-2} p \|\bpsi_\mu\|^2 = O_P(n^{-c}) = o_P(1),$$
    where $n^{-2} {\tr(\bV_\mu)^2} \|\hat{\bbh}_\mu\|^2 = n^{-2} {\tr(\bV_\mu)^2} \alpha^2 + o_P(1)$ from
    $\tr[\bm V_\mu]^2/n^2 \leq \|\bV_\mu\|_{op}^2 \le 1$ and $\|\hat{\bm h}_\mu\|^2 \to^p \alpha^2>0$ by \Cref{lm:h_diff}. Thus, we have 
    $
    n^{-2} {\tr[\bm V_\mu]^2} \alpha^2 - \gamma n^{-1}\|\bm\psi_\mu\|^2 = o_P(1)
    $, so that 
    $$
    \PP\Bigl(\frac{\tr[\bm V_\mu]^2}{n^2} - \frac{\gamma}{n\alpha^2}\|\bm\psi_\mu\|^2 \geq -\frac{\|\psi\|_{\infty}^2 \eta^2(1-\gamma)\gamma}{4\alpha^2}\Bigr) \to 1. 
    $$
    Thereby, under $\Omega^c=\{\|\bm \psi_\mu\|^2 \geq n \|\psi\|_{\infty}^2 \eta^2(1-\gamma)/2\}$, we have 
    \begin{align}
        \PP\Bigl(\bm{1}_{\Omega^c} \frac{\tr(\bV_\mu)^2}{n^2}
        \geq \bm{1}_{\Omega^c} \Bigl(\frac{\|\psi\|_{\infty}^2 \eta^2 (1-\gamma)\gamma}{2\alpha^2} - \frac{\|\psi\|_{\infty}^2\eta^2 (1-\gamma)\gamma}{4\alpha^2}\Bigr) = \bm{1}_{\Omega^c} \frac{\|\psi\|_{\infty}^2\eta^2 (1-\gamma)\gamma}{4\alpha^2}\Bigr) \to 1 \label{eq:omega_comp}
    \end{align}
    Consequently, \eqref{eq:omega} and \eqref{eq:omega} yield 
    $$
    \PP\Bigl(\frac{\tr(\bV_\mu)^2}{n^2} \geq \min\bigl\{\frac{\eta^4(1-\gamma)^2}{4}, \frac{\|\psi\|_{\infty}^2 \eta^2(1-\gamma)\gamma}{4\alpha^2}\bigr\} \Bigr)\to 1, 
    $$
    which completes the proof. 
\end{proof}

\subsection{Proof of \Cref{th:ofs}}
\label{sec_proof:th_ofs}
Recall that the goal is to show \eqref{eq:target}. Putting \Cref{lm:str_smoothed}, \Cref{lm:h_diff}, and \Cref{lm:psi_diff}, \Cref{lm:diff_V}
 together, we have 
\begin{align}
    \label{decomposition}
    &\frac{p}{n^2} \|\bm \psi\|^2 - \frac{\tr[\bm V]^2}{n^2}\|\hat{\bm h}\|^2 \\
    &= \Bigl(\frac{p}{n^2} \|\bm \psi_\mu\|^2 - \frac{\tr[\bm V_\mu]^2}{n^2}\|\hat{\bm h}\|^2\Bigr)
    +\frac{p}{n^2} (\|\bm \psi|_2^2 - \|\bm \psi_\mu\|^2) + \frac{\tr[\bm V_\mu]^2 - \tr[\bm V]^2}{n^2}\|\hat{\bm h}_\mu\|^2 - \frac{\tr[\bm V]^2}{n^2}(\|\hat{\bm h}\|^2 - \|\hat{\bm h}_\mu\|^2)
    \nonumber
    \\
    &= O_P(n^{-c}) + {o_P}(n^{-c/2}) +  {o_P}(n^{-c/2}) - \frac{\tr[\bm V]^2}{n^2}{o_P(1)} 
    \nonumber
    \\
    &= {o_P(n^{-c/2})} -\frac{\tr[\bm V]^2}{n^2} o_P(1). 
    \nonumber
\end{align}
Multiplying $(\tr(\bV)^2/n^2)^{-1}$, which is $O_P(1)$ by \Cref{lm:lb_trace}, we conclude the proof.

\subsection{Proof of \Cref{prop:relaxed}}
\label{sec_proof_relaxed}

\editline{
The strategy is exactly the same as for \Cref{th:ofs},
with \Cref{lm:diff_V} replaced by \Cref{lemma_vanishing}
to allow for a vanishing smooth component in the noise.
More precisely, set $c=1/4$ so that $\mu=n^{-1/4}$.
The conclusions of \Cref{lm:str_smoothed},
\Cref{lm:h_diff} and \Cref{lm:psi_diff} still hold for $\epsilon_i = \sigma_n z_i + \delta_i$
with $\sigma_n\to 0$, where $\alpha$ is now the solution to the
system \eqref{eq:intro_nonlinear}
with $W\sim \tilde F$, since the contribution of terms involving $\sigma_n z_i$
in \eqref{CGMT} is 0 due to $\sigma_n\to 0$ and the fact that the loss
is Lipschitz by \Cref{as:loss}.
The condition $\sigma_n \ge \sqrt{\mu}$ of \Cref{lemma_vanishing}
is satisfied thanks to the assumption $\sigma_n\ge n^{-1/8}$ in
\Cref{prop:relaxed}.
The argument and conclusion of \Cref{lm:lb_trace} still hold as well,
with $B(\gamma,\rho,F_\epsilon)$ replaced by $B(\gamma, \rho, \tilde F)$,
again because $\alpha$ is now the solution to the system
\eqref{eq:intro_nonlinear} with $W\sim \tilde F$.
The decomposition \eqref{decomposition}, and the same argument
as for \Cref{th:ofs}, thus yield the conclusion of \Cref{prop:relaxed}.
}

\subsection{Proof of \Cref{cor:tuning}}\label{proof:cor:tuning}
Fix $\epsilon>0$. 
Letting $R_k = \|\bSigma^{1/2} (\hat{\bbeta}_k-\bbeta^\star)\|^2$ for each $k\in [K] := \{1, \dots, K\}$, 
\Cref{th:ofs} \editline{or \Cref{prop:relaxed}} imply $\PP(|R_k - \hat{R}_k| \geq \epsilon/2) \to 0$
for all $k$. Note that $\hat{R}_{\hat{k}} < \hat{R}_k$ for all $k$ by the definition of $\hat{k}$. 
Then, we have
\begin{align*}
    \PP\bigl(R_{\hat{k}} > \min_{k\in [K]} R_k + \epsilon\bigr) &\leq \PP(\exists k\in [K],  R_{\hat{k}} > R_k + \epsilon)\\
    &\leq \sum_{k=1}^K \PP(R_{\hat{k}} > R_k + \epsilon)
    &&\text{by the union bound}
    \\
    & = \sum_{k=1}^K \PP(R_{\hat{k}} > R_k + \epsilon, \hat{R}_k \ge \hat{R}_{\hat{k}})
    &&\text{by definition of $\hat k$}
    \\
    &\le  \sum_{k=1}^K \sum_{l=1}^K \PP(R_{l} > R_k + \epsilon, \hat{R}_k > \hat{R}_{l}) && \text{by the union bound}.
\end{align*}
Here,  $\PP(R_{l} > R_k + \epsilon, \hat{R}_k > \hat{R}_{l})\to0$ as $n\to\infty$ by the following argument: 
\begin{align*}
    \PP(R_{l} > R_k + \epsilon, \hat{R}_k > \hat{R}_{l}) &\leq  \PP(R_{l} - \hat{R}_l - (R_k - \hat{R}_k) > \hat{R}_k - \hat{R}_l+ \epsilon > \epsilon)\\
    &\leq \PP(|R_l-\hat{R}_l| > \epsilon/2) + \PP(|R_k-\hat{R}_k| > \epsilon/2) \to 0. 
\end{align*}
Since $K$ is finite, we conclude $\PP(R_{\hat{k}} > \min_{k\in [K]} R_k + \epsilon)\to 0$. Since $\epsilon>0$ is taken arbitrarily, this finishes the proof.

\section{Proofs for \Cref{sec:tuning}}\label{sec:proof_tuning}
Throughout of this section, we fix $(\rho, F_\epsilon)$ that satisfy \Cref{as:loss}, \Cref{as:noise} and \Cref{as:tuning}. 
For all $\lambda>0$, define $\rho_\lambda$ and $\psi_\lambda$ as 
$$
\forall\lambda>0, \quad \rho_\lambda(\cdot) := \lambda^2 \rho(\cdot/\lambda), \quad \psi_\lambda (\cdot) := \rho_\lambda'(\cdot)=\lambda\psi(\cdot/\lambda).
$$
Note in passing that $\|\psi_\lambda\|_{\lip}=\|\psi\|_{\lip}=1$ and  $\|\psi_\lambda\|_{\infty} = \lambda\|\psi\|_{\infty} <+\infty$. Define $R(\lambda)$, $\hat{R}(\lambda)$, and $\alpha(\lambda)$ as
\begin{align*}
    R(\lambda) &:= \|\bSigma^{1/2} (\hat{\bbeta}_\lambda-\bbeta^\star)\|^2  &&\text{with $\hat{\bbeta}_\lambda \in \argmin_{\bbeta\in\R^p}\sum_{i=1}^n \rho(y_i-\bx_i^\top \bbeta)$ } \\
    \hat{R}(\lambda) &:= p\frac{\|\bpsi_\lambda\|^2}{\tr[\bV_\lambda]^2} &&\text{with } \bV_\lambda = \frac{\partial \bpsi_\lambda}{\partial\bm{y}} \in \R^{n\times n}, \ \bpsi_\lambda := \psi_\lambda(\by - \bX \hat{\bbeta}_\lambda) \\
    \alpha(\lambda) &:= \alpha(\rho_{\lambda}, F_\epsilon, \gamma)  &&(\text{the solution to \eqref{eq:nonlinear} with $\rho=\rho_\lambda$}).
\end{align*}

\subsection{Proof of \Cref{prop:scaled_control}}\label{proof:scaled_control}
$\hat{R} (\lambda) = R(\lambda) + o_P(1)$ by \Cref{th:ofs}, while $R(\lambda)\to^p \alpha^2(\lambda)$ by \cite{thrampoulidis2018precise} and \Cref{th:system}. 
{
It remains to show the upper bound of $\alpha^2(\lambda)$. 
Inequality \eqref{eq:localization} in \Cref{th:system} with $\rho=\rho_\lambda$ implies that $\alpha(\lambda)$ is bounded from above as
$$
\alpha(\lambda) \le \frac{Q_{F_\epsilon}(r^2 c_\gamma)}{r c_\gamma} + \frac{b}{c_\gamma}
$$
where $r\in (0, 1]$ and $b\ge 0$ are any constant such that the coercivity condition
$$
\|\rho_\lambda\|_{\lip}^{-1}(\rho_\lambda(x)-\rho_\lambda(0)) \ge r (|x|-b) \quad \text{for all $x\in\R$}
$$
is satisfied. Now we verify that 
$$
r=\frac{\min\{|\psi(\pm 1)|\}}{\|\rho\|_{\lip}}\quad  \text{ and }\quad  b=\lambda \cdot \frac{\max\{|(\mp1)\psi(\pm1) -\rho(\pm1)+\rho(0)|\}}{\min\{|\psi(\pm 1)|\}}
$$
satisfy the coercivity condition. 
% \begin{align*}
%     a_{\rho_\lambda}(\lambda) &=\min(|\psi_\lambda(\pm \lambda)|)^{-1}  \|\rho_\lambda\|_{\lip}  = \min(\lambda |\psi(\pm 1)|)^{-1} \lambda \|\rho\|_{\lip} =   \min(|\psi(\pm 1)|)^{-1} \|\rho\|_{\lip} = a_\rho(1), \\
%     b_{\rho_\lambda}(\lambda) &= \frac{\max\{|(\pm \lambda) \psi_\lambda(\pm \lambda) - \rho_\lambda(\pm \lambda) + \rho_\lambda(0)|\}}{\min(|\psi_\lambda(\pm \lambda)|)} =  \frac{\lambda^2 \max\{|(\pm 1) \psi(\pm 1) - \rho(\pm 1) + \rho(0)|\}}{\lambda\min(|\psi(\pm 1)|)} = \lambda b_\rho(1),
% \end{align*}
By the convexity, evaluating the derivative at $\lambda$, using $\rho_\lambda(\cdot)=\lambda^2\rho(\cdot/\lambda)$ and $\psi_\lambda(\cdot)=\lambda\psi(\cdot/\lambda)$, 
we have
\begin{align*}
    \rho_\lambda(x) \ge \rho_\lambda(\lambda) + (x-\lambda) \psi_\lambda(\lambda) = \lambda \psi(1) x - \lambda^2 \psi(1) + \lambda^2\rho(1)
\end{align*}
for all $x\in \R$. By the same argument, evaluating the derivative at $-\lambda$ gives
$$
\rho_\lambda(x) \ge \rho_\lambda(-\lambda) + (x+\lambda) \psi_\lambda(-\lambda) = \lambda \psi(-1)x+\lambda^2\psi(-1) + \lambda^2\rho(-1). 
$$
Thanks to $\{0\}=\argmin_x \rho(x)$ by \Cref{as:loss}, $\psi(-1) < 0 < \psi(1)$ holds, and hence, 
\begin{align*}
    \rho_\lambda(x)-\rho_\lambda(0) &\ge \lambda \min\Bigl(|\psi(1)|, |\psi(1)|\Bigr)|x| - \lambda^2 \max\Bigl(|\psi(1)-\rho(1) + \rho(0)|, |-\psi(-1)-\rho(-1)+\rho(0)|\Bigr).
\end{align*}
Dividing the both sides by $\|\rho_\lambda\|_{\lip}=\lambda\|\rho\|_{\lip}$, we obtain
$$
\frac{\rho_\lambda(x)-\rho_\lambda(0)}{\|\rho_\lambda\|_{\lip}} \ge \frac{\min(|\psi(\pm1)|)}{\|\rho\|_{\lip}} |x| - \lambda \frac{\max(|(\mp)\psi(\pm)-\rho(\pm)+\rho(0)|)}{\|\rho\|_{\lip}}.
$$
This means that the pair of $(r, b)$ specified above satisfies the coercivity condition. Substituting this to the upper bound of $\alpha(\lambda)$, since the dependence on $\lambda$ only comes from $b$, we obtain
$$
\alpha(\lambda) \le \C(\rho, F_\epsilon, \gamma) (1+\lambda) , 
$$
which finishes the proof. 
}

\subsection{Proof of \Cref{th:alpha_lipschitz}}\label{proof:alpha_lipschitz}
Thanks to \Cref{prop:scaled_control}, we have
$$
\alpha^2(\lambda)-\alpha^2(\tilde{\lambda}) = \hat{R}(\lambda)-\hat{R}(\tilde{\lambda}) + o_P(1),
$$
so it suffices to show a suitable smoothness of the map $\lambda \mapsto \hat{R}(\lambda)$. 
\begin{lemma}\label{lm:tr_lambda_bound}
We have
    \begin{align*}
        \forall \lambda>0, \quad  |\tr[\bV_\lambda]| \le n, \quad 
        \PP\left(n^2 \cdot [\tr \bV_\lambda]^{-2} \leq \C(\rho, F_\epsilon, \gamma) (1+\lambda^{-2}) \right) \to 1.
    \end{align*}
\end{lemma}
\begin{proof}
    $\tr[\bV_\lambda]/n\le 1$ immediately follows from \Cref{prop:psi_lip} and $\|\psi_\lambda\|_{\lip} = \|\psi\|_{\lip}= 1$. Note in passing that $\psi_\lambda(\cdot)=\lambda\psi(\cdot/\lambda)$ satisfies $\|\psi_\lambda\|_{\infty}=\lambda\|\psi\|_\infty <+\infty$, and for almost every $x$,  
    \begin{align*}
        \psi_\lambda^2(x)/\|\psi_\lambda\|_{\infty}^2 + \psi_\lambda'(x) 
        &= \psi(x/\lambda)^2/\|\psi\|_{\infty} + \psi'(x/\lambda) 
      \\&\ge \eta^2>0
    \end{align*}
    thanks to \editline{\Cref{as:loss}(3)}
    where $\eta$ is a constant that only depends on $\rho$. Thus, 
    \Cref{lm:lb_trace} with $\rho=\rho_\lambda$ implies
    $$
    \PP\Bigl(\frac{\tr[\bV_\lambda]^2}{n^2} \ge C(\gamma) \min\Bigl(\eta^4, \frac{\lambda^2 \|\psi\|_{\infty}^2 \eta^2}{\alpha^2(\lambda)}\Bigr)\Bigr)\to 1.
    $$
    Putting this lower bound and 
    $
    \alpha^2(\lambda) \le C'(\gamma, \rho, F_\epsilon)(\lambda^2 + 1)
    $ by \Cref{prop:scaled_control} together, we have 
    $$
    \frac{\tr[\bV_\lambda]^2}{n^2} \ge
    C(\gamma) \min\Bigl(\eta^4, \frac{\lambda^2 \|\psi\|_{\infty}^2 \eta^2}{C'(\rho, F_\epsilon, \gamma)(1+\lambda^2)}\Bigr)
    \ge
    C(\gamma) \min\Bigl(\eta^4, \frac{\eta^2\|\psi\|_{\infty}^2}{C'(\rho, F_\epsilon, \gamma)}\Bigr)
     \frac{\lambda^2}{1+\lambda^2}
    $$
    with high probability. This finishes the proof.
\end{proof}

\begin{lemma}\label{lm:psi_lambda_lip}
    We have the followings for all $\lambda, \tilde{\lambda}>0$:
\begin{align*}
        \PP\Bigl(\|\bpsi_{\lambda} - \bpsi_{\tilde \lambda}\| \leq \sqrt{n} \C(\gamma, \rho, F_\epsilon) |\lambda-\tilde{\lambda}|^{1/2} (1+\lambda^{1/2}+\tilde{\lambda}^{1/2})\Bigr) &\to 1, \\
        \PP\Bigl(|\|\bpsi_{\lambda}\|^2 - \|\bpsi_{\tilde{\lambda}}\|^2| \leq n  \C(\gamma, \rho, F_\epsilon) |\lambda-\tilde\lambda|^{1/2} (1+\lambda^{3/2} + \tilde{\lambda}^{3/2})\Bigr) &\to 1.            
\end{align*}
\end{lemma}
\begin{proof}
Note that the second display immediately follows from the first display since
\begin{align*}
    |\|\bpsi_{\lambda}\|^2 - \|\bpsi_{\tilde \lambda}\|^2| &= |\|\bpsi_{\lambda}\| - \|\bpsi_{\tilde\lambda}\|| (\|\bpsi_{\lambda}\| + \|\bpsi_{\tilde\lambda}\|)\\
    &\le |\|\bpsi_{\lambda}\| - \|\bpsi_{\tilde\lambda}\||  \sqrt{n} (\|\psi_\lambda\|_{\infty} + \|\psi_{\tilde{\lambda}}\|_{\infty})\\
    &=|\|\bpsi_{\lambda}\| - \|\bpsi_{\tilde\lambda}\||  \sqrt{n} \|\psi\|_{\infty}(\lambda+\tilde{\lambda}) \ && \text{ by $\|\psi_\lambda\|_{\infty}=\lambda\|\psi\|_{\infty}$}
\end{align*}
It remains to bound $|\|\bpsi_{\lambda}\| - \|\bpsi_{\tilde{\lambda}}\||$. Below, we write 
$\br_\lambda = \bepsilon - \bG \hat{\bh}_\lambda$ with $
\hat{\bh}_\lambda = \bSigma^{1/2}(\hat{\bbeta}_\lambda-\bbeta_*)$ so that
$\bpsi_\lambda = \psi_\lambda(\br_\lambda)$ and $\|\hat{\bh}_\lambda\|^2 = R(\lambda)$. 
By \Cref{as:tuning}, there exists some constant $c(\rho)>0$ that only depends on $\rho$ such that for all $\lambda, \tilde{\lambda}>0$,
$$
\sup_{\bu\in \R^n}\|\psi_\lambda(\bu) - \psi_{\tilde{\lambda}} (\bu)\|_2 \le \sqrt{n} \cdot \sup_{x\in \R}|\psi_\lambda(x)-\psi_{\tilde{\lambda}}(x)|\le \sqrt{n} c(\rho) |\lambda-\tilde{\lambda}|,
$$
so that $\|\bpsi_{\lambda} - \bpsi_{\tilde{\lambda}}\|$ can be upper bounded from above as 
$$
\|\bpsi_{\lambda} - \bpsi_{\tilde{\lambda}}\| 
\leq \|\psi_\lambda (\br_\lambda) -\psi_\lambda(\br_{\tilde{\lambda}})\| + \|\psi_\lambda(\br_{\tilde{\lambda}}) - \psi_{\tilde{\lambda}} (\br_{\tilde{\lambda}})\| \le \|\psi_\lambda (\br_\lambda) -\psi_\lambda(\br_{\tilde{\lambda}})\| + \sqrt{n}c(\rho)|\lambda-\tilde\lambda|,
$$
Here, using the \editline{KKT} conditions $\bG^\top \psi_\lambda(\br_\lambda) = \bG^\top \psi_{\tilde\lambda}(\br_{\tilde\lambda})=\bm{0}_p$, we bound the first from above as
\begin{align*}
    \|\psi_\lambda (\br_\lambda) -\psi_\lambda(\br_{\tilde{\lambda}})\|^2 &\le (\psi_\lambda(\br_\lambda)-\psi_{\lambda}(\br_{\tilde{\lambda}}))^\top(\br_\lambda - \br_{\tilde{\lambda}}) && \text{ by \Cref{lm:psi_increasing} and $\|\psi_\lambda\|_{\lip}=1$}\\
    &= (\psi_\lambda(\br_\lambda)-\psi_{\lambda}(\br_{\tilde{\lambda}}))^\top \bG (\hat{\bh}_{\tilde{\lambda}} - \hat{\bh}_{\lambda})\\
    &= (\psi_{\tilde \lambda}(\br_{\tilde \lambda})-\psi_{\lambda}(\br_{\tilde{\lambda}}))^\top \bG (\hat{\bh}_{\tilde{\lambda}} - \hat{\bh}_{\lambda}) && \text{ by $\bG^\top \psi_\lambda(\br_\lambda) = \bG^\top \psi_{\tilde\lambda}(\br_{\tilde\lambda})=\bm{0}_p$}\\
    &\le \|\psi_{\tilde \lambda}(\br_{\tilde \lambda})-\psi_{\lambda}(\br_{\tilde{\lambda}})\| \|\bG\|_{op} (\|\hat{\bh}_\lambda\| + \|\hat{\bh}_{\tilde\lambda}\|)\\
    &\le \sqrt{n} c(\rho) |\lambda-\tilde{\lambda}| \|\bG\|_{op} (\|\hat{\bh}_\lambda\| + \|\hat{\bh}_{\tilde\lambda}\|)
    % &\le \|\psi_{\tilde \lambda}(\br_{\tilde \lambda})-\psi_{\lambda}(\br_{\tilde{\lambda}})\| \cdot 2 \sqrt{n}(1-\sqrt{\gamma}) \cdot 2 (\alpha(\lambda)+\alpha(\tilde{\lambda})) && \text{ $\|\bG/\sqrt{n}\|_{op}\to^p (1-\sqrt{\gamma})$ and $\|\hat{\bh}_\lambda\|\to^p \alpha(\lambda)>0$} \\
    % &\le \C(\rho, \gamma) n |\lambda-\tilde{\lambda}| (\alpha(\lambda)+\alpha(\tilde{\lambda}))
\end{align*}
Therefore, using $\|\hat{\bm{h}}_\lambda\|\to^p\alpha(\lambda)$ and $\alpha(\lambda)\le \C(\rho, \gamma F_\epsilon) (1+\lambda)$ by \Cref{prop:scaled_control}, we have 
\begin{align*}
    \|\bpsi_{\lambda} - \bpsi_{\tilde{\lambda}}\|^2 &\le 2 \sqrt{n} c(\rho) |\lambda-\tilde{\lambda}| \|\bG\|_{op} (\|\hat{\bh}_\lambda\| + \|\hat{\bh}_{\tilde\lambda}\|) + 2 n c(\rho)^2 |\lambda-\tilde{\lambda}|^2\\
    &\le \C(\rho) \cdot n \cdot (n^{-1/2}\|\bG\|_{op}+1) \bigl(|\lambda-\tilde{\lambda}| (\|\hat{\bh}_\lambda\| + \|\hat{\bh}_{\tilde\lambda}\|) +  |\lambda-\tilde{\lambda}|^2 \bigr)\\
    &\le \C(\rho, \gamma, F_\epsilon) \cdot n \bigl(|\lambda-\tilde{\lambda}| (1+\lambda + \tilde{\lambda}) + |\lambda-\tilde{\lambda}|^2\bigr) \\
    &\le \C(\rho, \gamma, \epsilon) \cdot n |\lambda-\tilde{\lambda}|\cdot (1+\lambda + \tilde{\lambda})
\end{align*}
with high probability. This finishes the proof.
\end{proof}

\begin{lemma}\label{lm:tr_lambda_lip}
Suppose that $(\rho, F_\epsilon)$ satisfy \Cref{as:loss}, \Cref{as:noise} and \Cref{as:tuning}. 
Then, we have
for all $\lambda, \tilde{\lambda}>0$ that
\begin{align*}
    \PP\Bigl(n^{-2} |\tr [\bV_\lambda]^2 - \tr [\bV_{\tilde{\lambda}}]^2| \leq \C(\rho, F_\epsilon, \gamma) |\lambda-\tilde{\lambda}|^{1/2} (1+\lambda^{1/2} + \tilde{\lambda}^{1/2}) \Bigr) \to 1,
\end{align*}
\end{lemma}
\begin{proof}
By the same argument in the proof of \Cref{lm:diff_V}, the variant of second order Stein's formula leads to 
\begin{align*}
    \tr[\bV(\lambda)]-\tr[\bV(\tilde\lambda)] = \partial_{\bm{z}} (\bpsi_\lambda-\bpsi_{\tilde\lambda}) = \phi'(\bm{z})^\top (\bpsi_{\lambda}-\bpsi_{\tilde{\lambda}}) +  O_P(n^{1/2}), 
\end{align*}
where $\phi$ is the density function in \Cref{as:noise} such that $\|\phi''\|_{\infty} <+\infty$.
By \Cref{lm:psi_lambda_lip}, \Cref{lem:phi_concentrates}, and \editline{Cauchy--Schwarz}, we have 
$$
|\tr[\bV_\lambda]-\tr[\bV_{\tilde\lambda}]|  \le  \|\phi'(\bm{z})\| \|\bpsi_{\lambda}-\bpsi_{\tilde{\lambda}}\| + O_P(\sqrt{n}) \le n  \C(\rho, F_\epsilon, \gamma)|\lambda-\tilde{\lambda}|^{1/2} (1+\lambda^{1/2} + \tilde{\lambda}^{1/2})
$$
with high probability. Combined with $\tr[\bV_\lambda] \in [0,n]$ from \Cref{lm:tr_lambda_bound}, we conclude the proof.
\end{proof}

\begin{proof}[Proof of \Cref{th:alpha_lipschitz}]
    \Cref{lm:tr_lambda_bound}, \Cref{lm:psi_lambda_lip}, and \Cref{lm:tr_lambda_lip} imply that the followings events hold simultaneously with high probability:
    \begin{align*}
        n^2 \tr [\bV_{\lambda'}]^{-2} 
        &\leq
        \C(\rho, F_\epsilon, \gamma) (1+\lambda'^{-2}) 
        &&\text{ for }\lambda'\in\{\lambda, \tilde{\lambda}\}, \\
        n^{-1} |\|\bpsi_{\lambda}\|^2-\|\bpsi_{\tilde{\lambda}}\|^2| 
        &\leq
        \C(\rho, F_\epsilon, \gamma)|\lambda-\tilde{\lambda}|^{1/2} (1+\lambda^{1/2} + \tilde{\lambda}^{1/2}), \\
        n^{-2}\|\tr [\bV_\lambda]^2 - \tr [\bV_{\tilde{\lambda}}]^2| 
        &\leq
        \C(\rho, F_\epsilon, \gamma) |\lambda-\tilde{\lambda}|^{1/2} (1+\lambda^{1/2} + \tilde{\lambda}^{1/2}).
    \end{align*}
    Thus, we have
    \begin{align*}
        |\hat{R}(\lambda) - \hat{R}(\tilde \lambda)|
        &= \Bigl|\frac{p \|\bpsi_{\lambda}\|^2}{\tr [\bV_\lambda]^2} - \frac{p\|\bpsi_{\tilde{\lambda}}\|^2}{\tr[\bV_{\tilde{\lambda}}]^2}\Bigr| \\
        &=
        \frac{n^2}{\tr [\bV_{\lambda}]^2}\frac{n}{p} \Bigl|\frac{\|\bpsi_\lambda\|^2}{n} - \frac{\tr[\bV_\lambda]^2}{\tr[\bV_{\tilde{\lambda}}]^2} \frac{\|\bpsi_{\tilde{\lambda}}\|^2}{n}\Bigr|
        \\
        &\le
        \frac{n^2}{\tr [\bV_{\lambda}]^2}\frac{n}{p} \Bigl(\frac{|\|\bpsi_\lambda\|^2-\|\bpsi_{\tilde{\lambda}}\|^2|}{n} + \frac{n^2}{\tr [\bV_{\tilde{\lambda}}]^2}\frac{|\tr [\bV_{\lambda}]^2 - \tr [\bV_{\tilde{\lambda}}]^2|}{n^2}\frac{\|\bpsi_{\tilde{\lambda}}\|^2}{n} \Bigr)
        \\
     &\le C(\rho, \gamma, F_\epsilon) |\lambda-\tilde{\lambda}|^{1/2}(
        1+\lambda^{-2})
        \Bigl(
            1+\lambda^{1/2} + \tilde{\lambda}^{1/2} + (1+\tilde{\lambda}^{-2})(1+\lambda^{1/2} + \tilde{\lambda}^{1/2})\tilde{\lambda}
        \Bigr)
    \end{align*}
    Combined with $\hat{R}(\lambda)\to^p \alpha^2(\lambda)$ for each $\lambda\in (0,\infty)$, we have the following for all $\lambda, \tilde{\lambda}>0$:
    $$
    |\alpha^2(\lambda)-\alpha^2(\tilde{\lambda})| \le C(\rho, \gamma, F_\epsilon) |\lambda-\tilde{\lambda}|^{1/2}L(\lambda, \tilde{\lambda}),
    $$
    where $L(\lambda, \tilde{\lambda}):= (
        1+\lambda^{-2})
        \bigl(
            1+\lambda^{1/2} + \tilde{\lambda}^{1/2} + (1+\tilde{\lambda}^{-2})(1+\lambda^{1/2} + \tilde{\lambda}^{1/2})\tilde{\lambda}
        \bigr)$. Since $\sup_{\lambda, \tilde{\lambda} \in [\lambda_{\min}, \lambda_{\max}]}L(\lambda, \tilde{\lambda})$ is bounded from above by some positive constant $C(\lambda_{\min}, \lambda_{\max})$, we conclude the proof.
\end{proof}

\subsection{Proof of \Cref{th:lambda_tuning}}
\begin{lemma}\label{lm:tuning_delta}
    Fix $I=[\lambda_{\min}, \lambda_{\max}]$ for some $0<\lambda_{\min}<\lambda_{\max}<+\infty$. For all $N\in \mathbb{N}$, let $I_N=\{\lambda_{\min}^{i/N}\cdot \lambda_{\max}^{1-i/N}: i=0, 1, \dots N\}$ be the finite grid over $I$ and take $\hat{\lambda}_N \in \argmin_{\lambda\in I_N} \hat{R}(\lambda)$. 
    Then, for all $\delta>0$, there exists $N_\delta=N(\rho, F_\epsilon, \gamma, \lambda_{\min}, \lambda_{\max}, \delta)\in \mathbb{N}$ such that 
    $$
    \lim_{n\to\infty} \PP(|\alpha^2(\hat{\lambda}_{N_\delta}) - \alpha^2(\lambda_{\opt})| < \delta) = \lim_{n\to\infty} \PP(|R(\hat{\lambda}_{N_\delta}) - \alpha^2(\lambda_{\opt})| < \delta)\to 1,
    $$
    where $\hat{\lambda}_{N_\delta}\in \argmin_{\lambda\in I_{N_\delta}} \hat{R}(\lambda)$ and $\lambda_{\opt}\in \argmin_{\lambda\in I}\alpha^2(\lambda)$. 
\end{lemma}

\begin{proof}[Proof of \Cref{lm:tuning_delta}]
First, we prove $\PP(|\alpha^2(\hat{\lambda}_{N_{\delta}}) - \alpha^2(\lambda_{\opt})| < \delta)\to 1$. 
Let us take some $N\in\mathbb{N}$ to be specified later. Consider the decomposition
\begin{align*}
    0 < \alpha^2(\hat{\lambda}_N)-\alpha^2(\lambda_{\opt}) = \alpha^2(\hat{\lambda}_N) - \min_{\lambda\in I_N} \alpha^2(\lambda)  + \min_{\lambda\in I_N} \alpha^2(\lambda)- \alpha^2(\lambda_{\opt}).
\end{align*}
Now, the cardinality of $I_N$ is finite and $\hat{R}(\lambda) = \alpha^2(\lambda) + o_P(1)$ for all $\lambda>0$ by \Cref{prop:scaled_control}. Then, by the same argument of 
\Cref{proof:cor:tuning}, where the role of $R$ is replaced by $\alpha^2$, we have
$$
\alpha^2(\hat{\lambda}_N) - \min_{\lambda\in I_N} \alpha^2(\lambda) = o_P(1), 
$$ 
so in particular $\lim_{n\to\infty}\PP(\alpha^2(\hat{\lambda}_N) - \min_{\lambda\in I_N} \alpha^2(\lambda)  < \delta/2) = 1$.
Below, we show $\min_{\lambda\in I_N} \alpha^2(\lambda) -\alpha^2(\lambda_{\opt}) \le \delta/2$ for some $N=N_\delta$. 
By the definition of $I_N$, we can take $\lambda^{\star}_N \in I_N$ such that
$\lambda_N^\star \leq \lambda_{\opt} < (\lambda_{\max}/\lambda_{\min})^{1/N} \lambda^{\star}_N$, so that 
$$|\lambda_{\opt} - \lambda^\star_N| \leq [(\lambda_{\max}/\lambda_{\min})^{1/N} - 1] \lambda^{\star}_N \leq 
[(\lambda_{\max}/\lambda_{\min})^{1/N} - 1] \lambda_{\max}.
$$
On the other hand, \Cref{th:alpha_lipschitz} implies that 
$$
\text{for all $\lambda, \tilde{\lambda}\in I = [\lambda_{\min}, \lambda_{\max}]$}, \quad 
|\alpha^2(\lambda)-\alpha^2(\tilde{\lambda})| \le L |\lambda-\tilde{\lambda}|^{1/2},
$$
where $L=L(\rho, F_\epsilon, \gamma, \lambda_{\min}, \lambda_{\max})>0$ is some positive constant. From the above displays, it follows that
\begin{align*}
    0 &< \min_{\lambda\in I_{N}} \alpha^2(\lambda) - \alpha^2(\lambda_{\opt}) && \text{ by $\lambda_{\opt}\in \argmin_{\lambda\in I} \alpha^2(\lambda)$ and $I_N \subset I$}\\
    &\leq \alpha^2(\lambda^\star_{N}) - \alpha^2(\lambda_{\opt}) &&\text{ by $\lambda_{N}^* \in I_{N}$}\\
    &\leq L |\lambda^\star_{N} -\lambda_{\opt}|^{1/2} \quad &&\text{ by  $\lambda^\star_{N}, \lambda_{\opt}\in I$} \\
    &\le L \bigl\{\lambda_{\max} [(\lambda_{\max}/\lambda_{\min})^{1/N}-1]\bigr\}^{1/2} && \text{ by $|\lambda_{\opt} - \lambda^\star_N| \leq 
    [(\lambda_{\max}/\lambda_{\min})^{1/N} - 1] \lambda_{\max}$}\\
    &\le \delta/2 &&\text{ if $(\lambda_{\max}/\lambda_{\min})^{1/N} \le 1+ {\delta^2}/(4L\lambda_{\max})$}. 
\end{align*}
Therefore, if we take $N=N_\delta:=\lceil{\log({\lambda_{\max}}/{\lambda_{\min}})/\log\{1+{\delta^2}/(4L\lambda_{\max})\}}\rceil$, we obtain $\PP(|\alpha^2(\hat{\lambda}_{N_\delta}) - \alpha^2(\lambda_{\opt})| < \delta)\to 1$. 

Next, we prove  $\PP(|R(\hat{\lambda}_N) - \alpha^2(\lambda_{\opt})| \le \epsilon)\to 1$. By the triangle inequality, it follows that
\begin{align*}
    |R(\hat{\lambda}_N) -  \alpha^2(\lambda_{\opt})| \le |R(\hat{\lambda}_N) - \min_{\lambda \in I_N}{R}(\lambda) | +
        | \min_{\lambda\in I_N} R(\lambda) - \min_{\lambda\in I_N} \alpha^2(\lambda)| + |\min_{\lambda\in I_N} \alpha^2(\lambda) - \alpha^2(\lambda_{\opt})|.
\end{align*}
The first term is $o_P(1)$ by \Cref{cor:tuning}, while the second term is $o_P(1)$ as $R(\lambda)\to^p \alpha^2(\lambda)$ for all $\lambda>0$ and the cardinality of $I_N$ is finite. For the third term, we have shown that it is less than $\delta/2$ for $N=N_\delta$. This completes the proof of $\PP(|R(\hat{\lambda}_{N_\delta}) - \alpha^2(\lambda_{\opt})| \le \delta)\to 1$. 
\end{proof}

\begin{proof}[Proof of \Cref{th:lambda_tuning}]
Our goal is to construct an array $(N_n)_{n=1}^\infty$ of integers such that
\begin{align}\label{eq:increasing_n}
   \forall\epsilon>0, \quad \lim_{n\to\infty} \PP(|\alpha^2(\hat{\lambda}_{N_n}) - \alpha^2(\lambda_{\opt})| > \epsilon) + \PP(|R(\hat{\lambda}_{N_n}) - \alpha^2(\lambda_{\opt})| > \epsilon) = 0. 
\end{align}
By \Cref{lm:tuning_delta} with $\delta=1/k$, there exists a function $N: \mathbb{N}\to\mathbb{N}$ such that
\begin{equation}\label{eq:k}
    \forall k\in \mathbb{N}, \quad \lim_{n\to\infty}\quad \underbrace{\PP(|\alpha^2(\hat{\lambda}_{{N}(k)}) - \alpha^2(\lambda_{\opt})| \ge 1/k) + \PP(|R(\hat{\lambda}_{{N}(k)}) - \alpha^2(\lambda_{\opt})| \ge 1/k)}_{:=p_{n,k}} = 0. 
\end{equation}
Thus, if we can find some function $\varphi:\mathbb{N}\to\mathbb{N}$ such that $\lim_{n\to\infty}\varphi(n)\to\infty$ and $\lim_{n\to\infty} p_{n, \varphi(n)}=0$, the array $\{N_n\}_{n=1}^\infty:=\{N\circ\varphi(n)\}_{n=1}^\infty$ satisfies \eqref{eq:increasing_n}, thereby completing the proof. 
Below we construct such $n\mapsto \varphi(n)$.

By \eqref{eq:k}, there exists an array of integers $\{n_k\}_{k=1}^\infty$ such that 
$$
\forall k\in \mathbb{N}, \quad \forall n \ge n_k, \quad  |p_{n, k}| \le 1/k. 
$$
Here, we assume without of loss of generality that $\{n_k\}_k$ is strictly increasing; otherwise redefine recursively as $\tilde n_k :=\max(\tilde n_{k-1}, n_k)+1$. For this $\{n_k\}_{k=1}^\infty$, define the step function $\varphi:\mathbb{N}\to\mathbb{N}$ as $\varphi(n) := \sum_{k=1}^\infty k \bm{1}\{n_k \le n< n_{k+1}\}$. 
Then, we have $\lim_{n\to\infty} \varphi(n)=+\infty$, and 
% because $\varphi$ is unbounded. Indeed,
% if $\varphi$ was bounded by some integer $K$ then for $n = n_{K+1}$ we should have $\varphi(n) = K+1$, a contradiction.
$|p_{n,k}| \le 1/k$ holds for $k=\varphi(n)$  since $n\ge n_{k}$ by construction. Therefore, we obtain
$$
\forall n\ge n_1, \quad |p_{n,\varphi(n)}|\le 1/\varphi(n).
$$
Since $\varphi(n)\to+\infty$ as $n\to\infty$, we have $\lim_{n\to\infty} p_{n,\varphi(n)}=0$. This finishes the proof. 
\end{proof}

\section{Additional numerical simulation}\label{sec:additional_simulation}
\subsection{Relaxing \Cref{as:noise}}\label{subsec:relax_noise_assumption}
\editline{
In the same setting as \Cref{sec:numeric}, we change the noise distribution to be the following discrete distribution so that \Cref{as:noise} is not satisfied:
\begin{align*}
    \epsilon_i \iid  3 \cdot \lceil \tdist (\df=2) \rceil 
\end{align*}
where $\lceil x \rceil = \max\{n\in\mathbb{Z}: n\le x\}$. \Cref{fig:risk_consistency_integer} implies  that our result still holds in this setting.
}

\begin{figure}[htpb]
    \centering
    \begin{subfigure}[b]{0.49\textwidth}
        \centering
        \includegraphics[width=\textwidth]{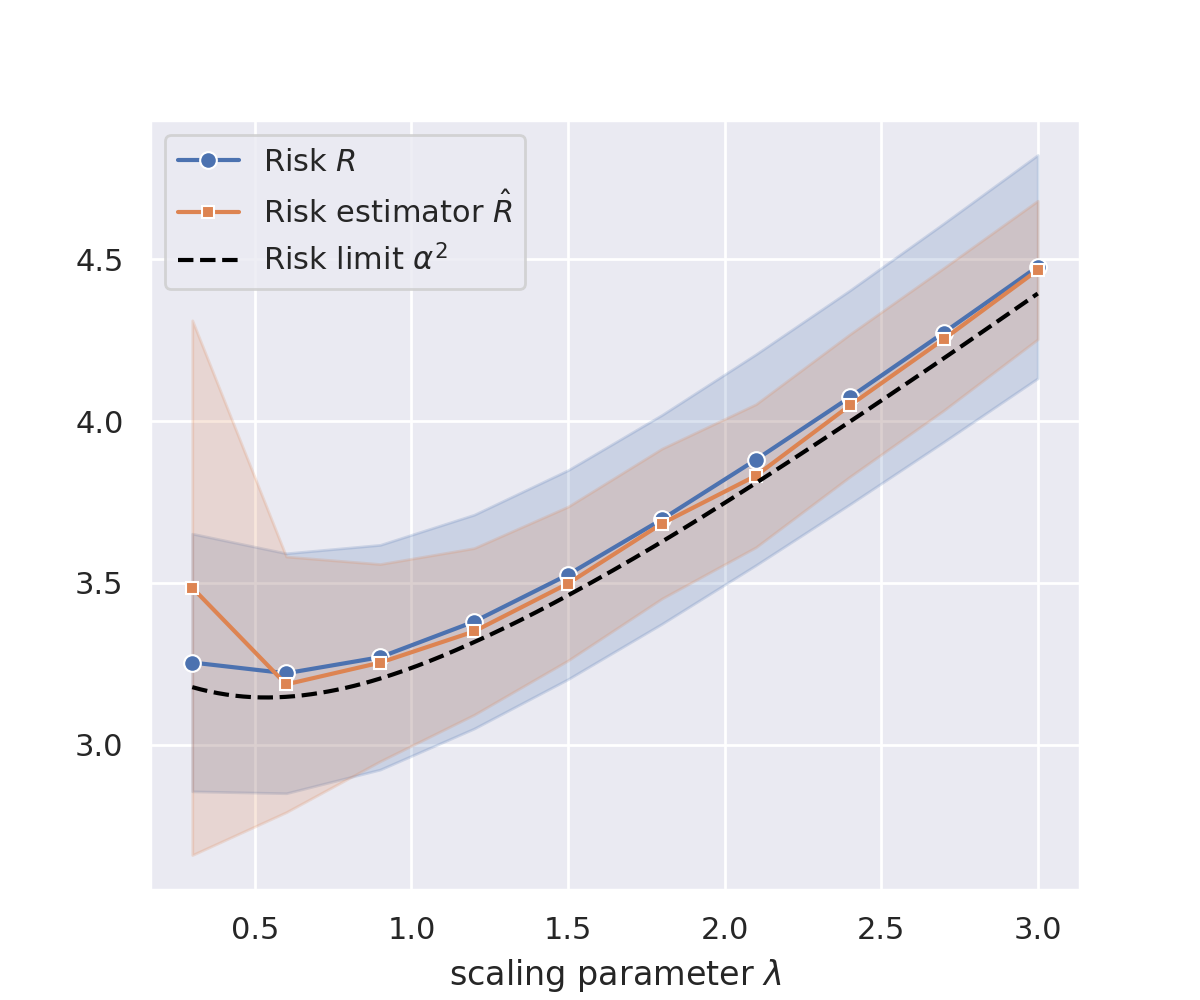}
        \caption{$R(\lambda)$, $\hat{R}(\lambda)$ and $\alpha^2(\lambda)$.}
        \label{subfig:ofs_consistency_sigma0}
    \end{subfigure}
    \begin{subfigure}[b]{0.49\textwidth}
        \centering
        \includegraphics[width=\textwidth]{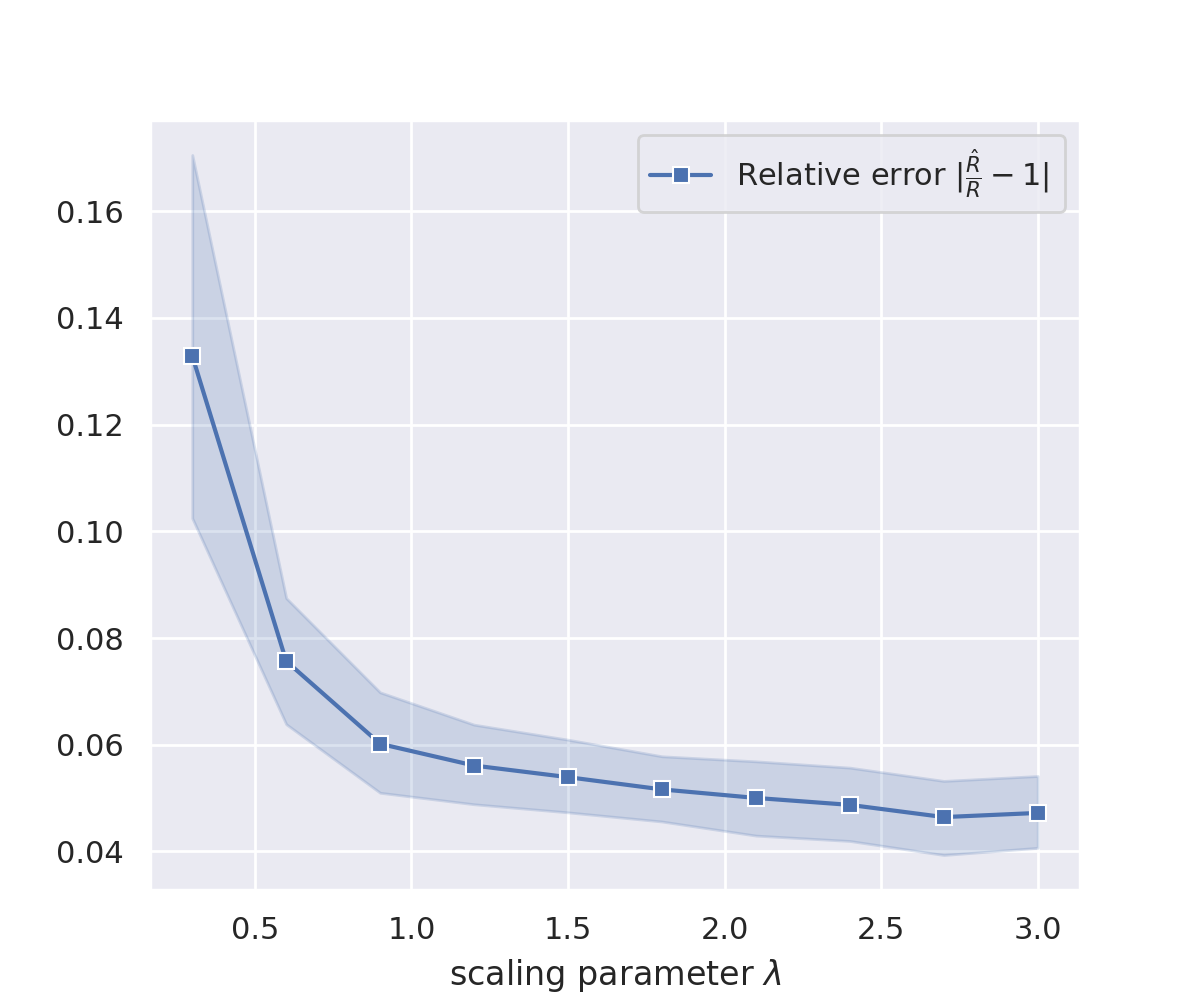}
        \caption{Relative error $|\hat{R}(\lambda)/R(\lambda)-1|$}
        \label{subfig:ofs_relative_error_sigma0}
    \end{subfigure}
    \caption{
        \editline{
        Plot of the out-of-sample error $R(\lambda)$ and estimator $\hat{R}(\lambda)$ over 100 repetitions, with $n=4000$, $p=1200$, the Huber loss for different values of the scale parameters $\lambda$. The noise distribution is $3 \lceil \tdist (\df=2) \rceil $. 
    $\alpha^2(\lambda)$ is the theoretical limit given by the nonlinear system \eqref{eq:intro_nonlinear}.
        }}
    \label{fig:risk_consistency_integer}
\end{figure}

% \begin{figure}
%     \centering
%     \begin{subfigure}[b]{0.49\textwidth}
%         \centering
%         \includegraphics[width=\textwidth]{Figure/adaptive_tuning_appendix.png}
%         \caption{Adaptive tuning vs oracles}
%         \label{subfig:tuning_oracle_sigma0}
%     \end{subfigure}
%     \begin{subfigure}[b]{0.49\textwidth}
%         \centering
%         \includegraphics[width=\textwidth]{Figure/adaptive_tuning_relative_error_appendix.png}
%         \caption{Relative error}
%         \label{subfig:tuning_relative_error_sigma0}
%     \end{subfigure}
%     \caption{
%     \editline{
%     Adaptive tuning as the scale of noise $\sigma$ changes as $F_\epsilon = \sigma \cdot \lceil \tdist (\df=2) \rceil $.
%     $I=[0.5,5]$, $n=4000$, $p=1200$, and $I_N$ is the uniform grid in log-scale of length $101$. 
%     ${R}(\hat{\lambda}_N)$ is the out-of-sample error with $\lambda$ selected by $\hat{R}$, $\min_{\lambda\in I_N} R(\lambda)$ is the optimal out-of-sample error among $I_N$, and $\min_{\lambda\in I}\alpha^2(\lambda)$ is the theoretically optimal risk limit. We repeat $100$ times.
%     }}
%     \label{fig:tuning_sigma0}
% \end{figure}

\editline{
\subsection{Robustness of risk estimation across covariate distributions}\label{subsec:universality}
We vary the distribution of the covariate $\bX$ from Gaussian to Rademacher, uniform and  Student's t-distribution in \Cref{fig:universality}. 
 In each of these settings, we confirm the effectiveness of the risk estimator, suggesting that our results extend to a broader class of covariate distributions.

\begin{figure}[htpb]
    \centering
    \begin{subfigure}[b]{0.32\textwidth}
        \centering
        \includegraphics[width=\textwidth]{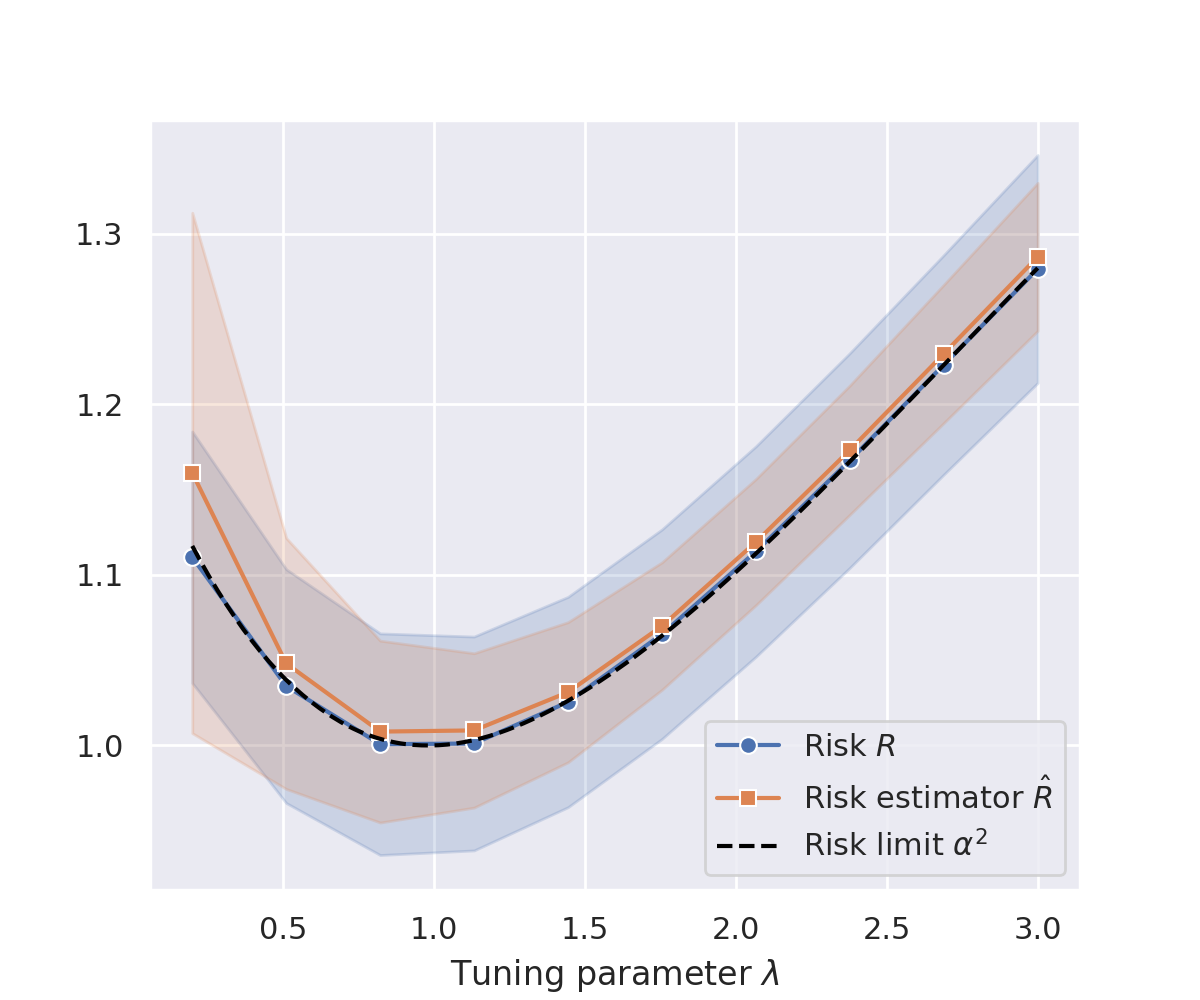}
        \caption{Rademacher}
    \end{subfigure}
    \begin{subfigure}[b]{0.32\textwidth}
        \centering
        \includegraphics[width=\textwidth]{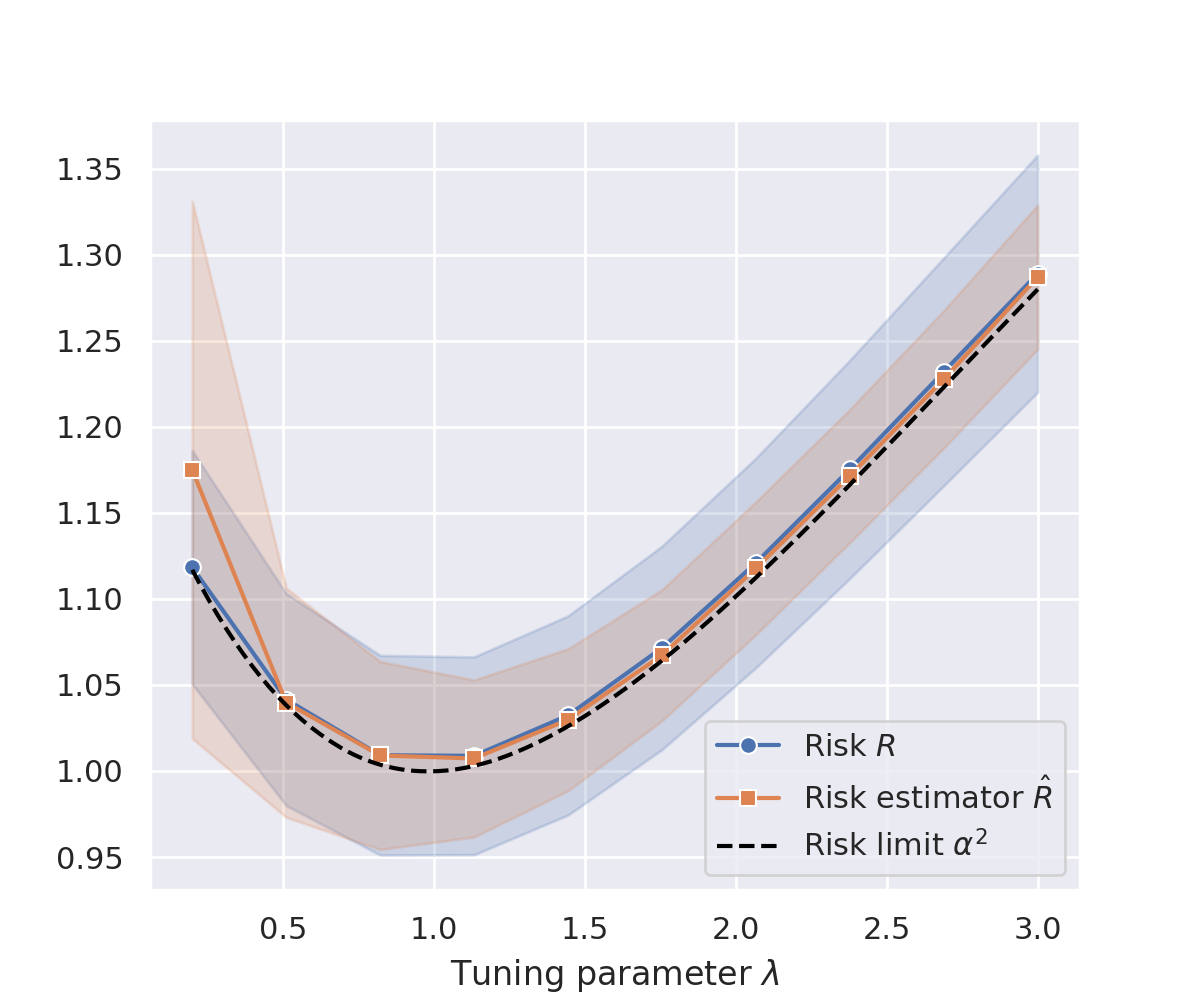}
        \caption{Uniform $[-1, 1]$ }
    \end{subfigure}
    \begin{subfigure}[b]{0.32\textwidth}
        \centering
        \includegraphics[width=\textwidth]{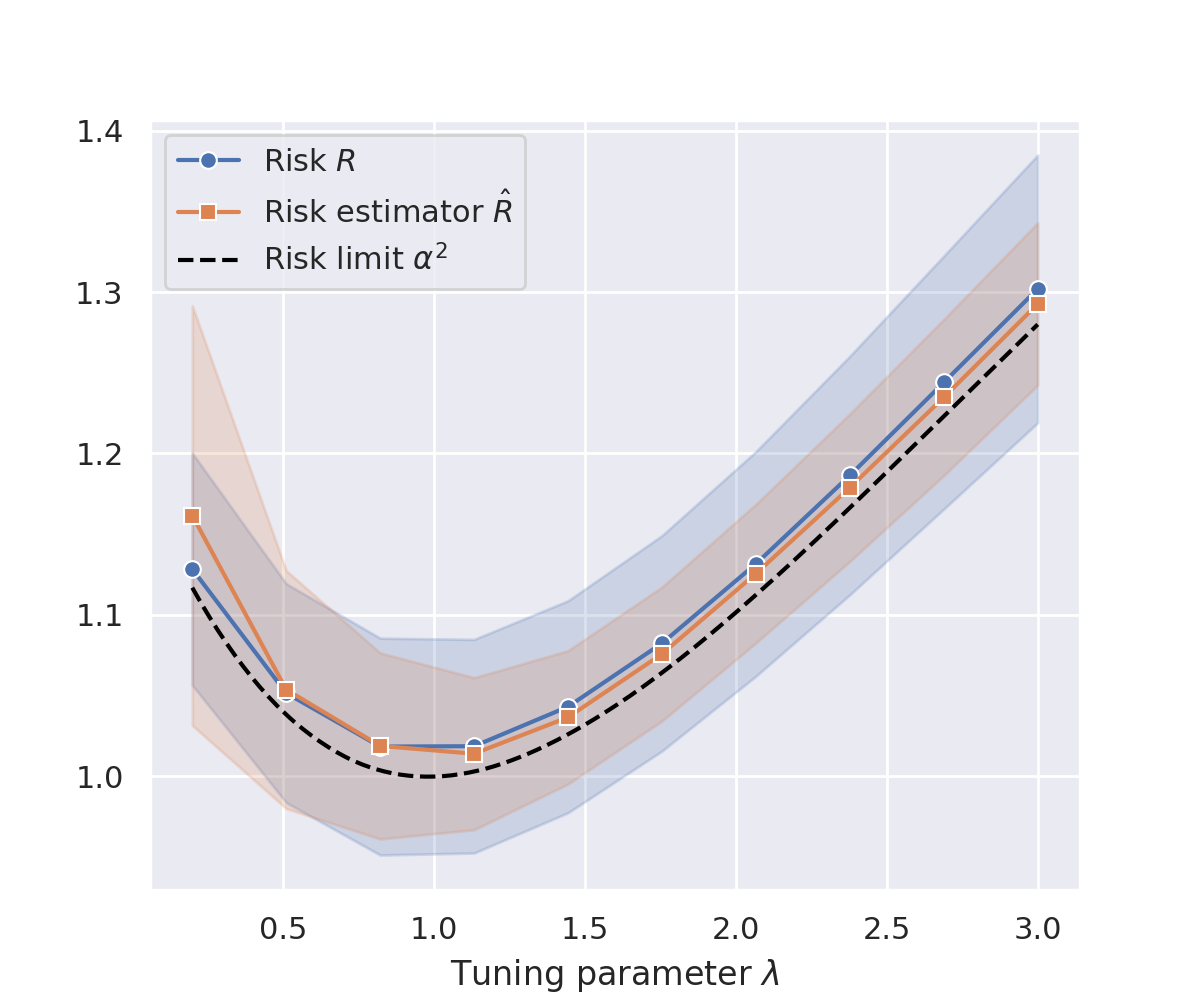}
        \caption{T-dist with $\text{df}=4$}
    \end{subfigure}
    \caption{
    \editline{We change the distribution of the design matrix $\bX$ to Uniform, Student's T, and Rademacher distributions with proper normalization so that the variance is $1$. 
    Simulation parameter: the noise distribution $F_\epsilon = \text{t-dist}(\text{df}=2)$, 
    sample size $n=4000$, feature $p=1200$, scaling parameter $\lambda=1$, $100$ repetition. 
    }
    }
    \label{fig:universality}
\end{figure}

\subsection{Risk behavior as dimensionality ratio $\gamma=p/n$ approaches $1$}\label{subsec:dependence_on_gamma}
We plot the risk and its estimator as $\gamma=p/n$  varies in \Cref{fig:change_gamma}. The results show that the risk estimation becomes increasingly unstable as  $\gamma \to 1^{-}$.

\begin{figure}[htpb]
    \centering
    \begin{subfigure}[b]{0.49\textwidth}
        \centering
        \includegraphics[width=\textwidth]{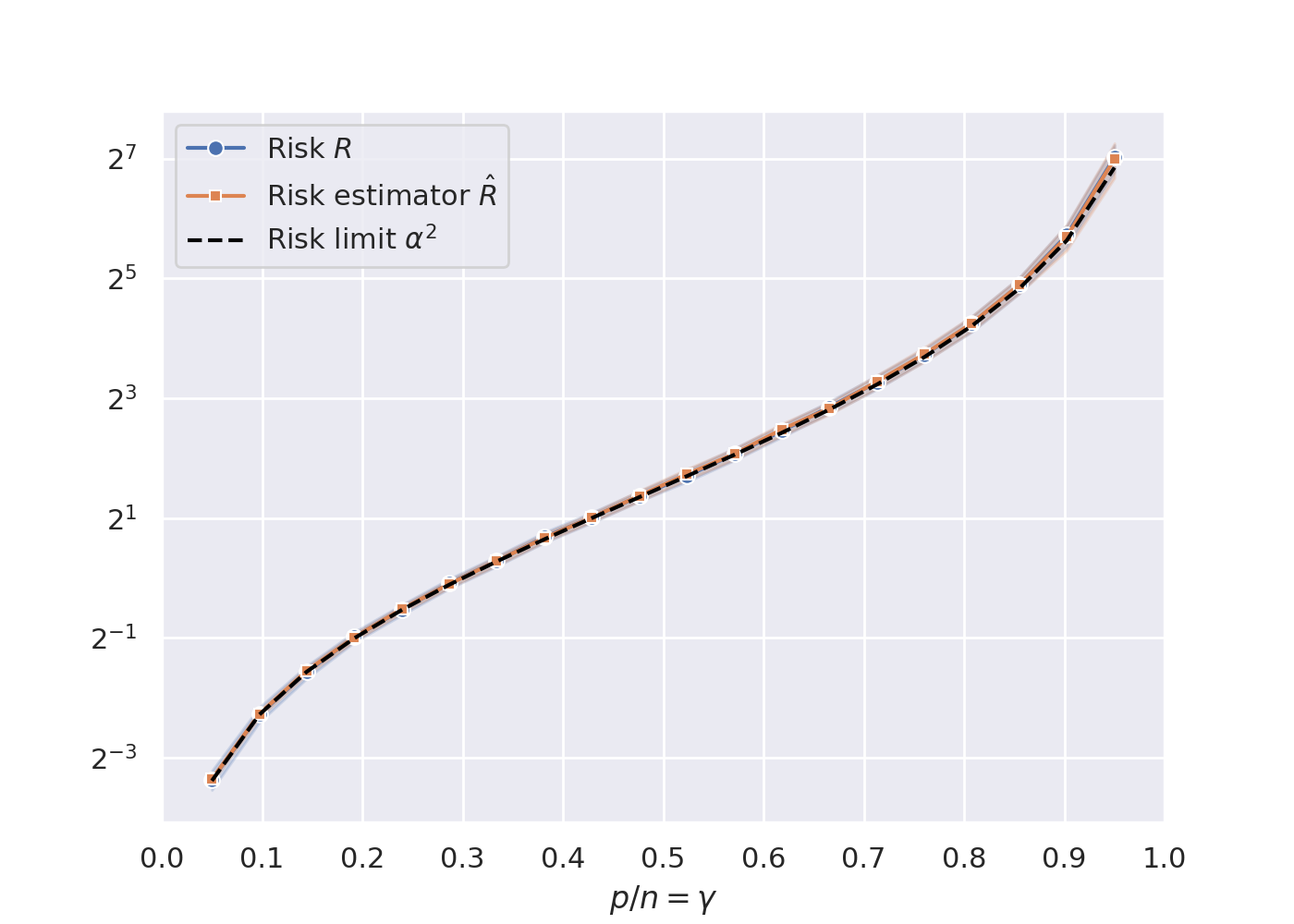}
        % \caption{}
    \end{subfigure}
    \begin{subfigure}[b]{0.49\textwidth}
        \centering
        \includegraphics[width=\textwidth]{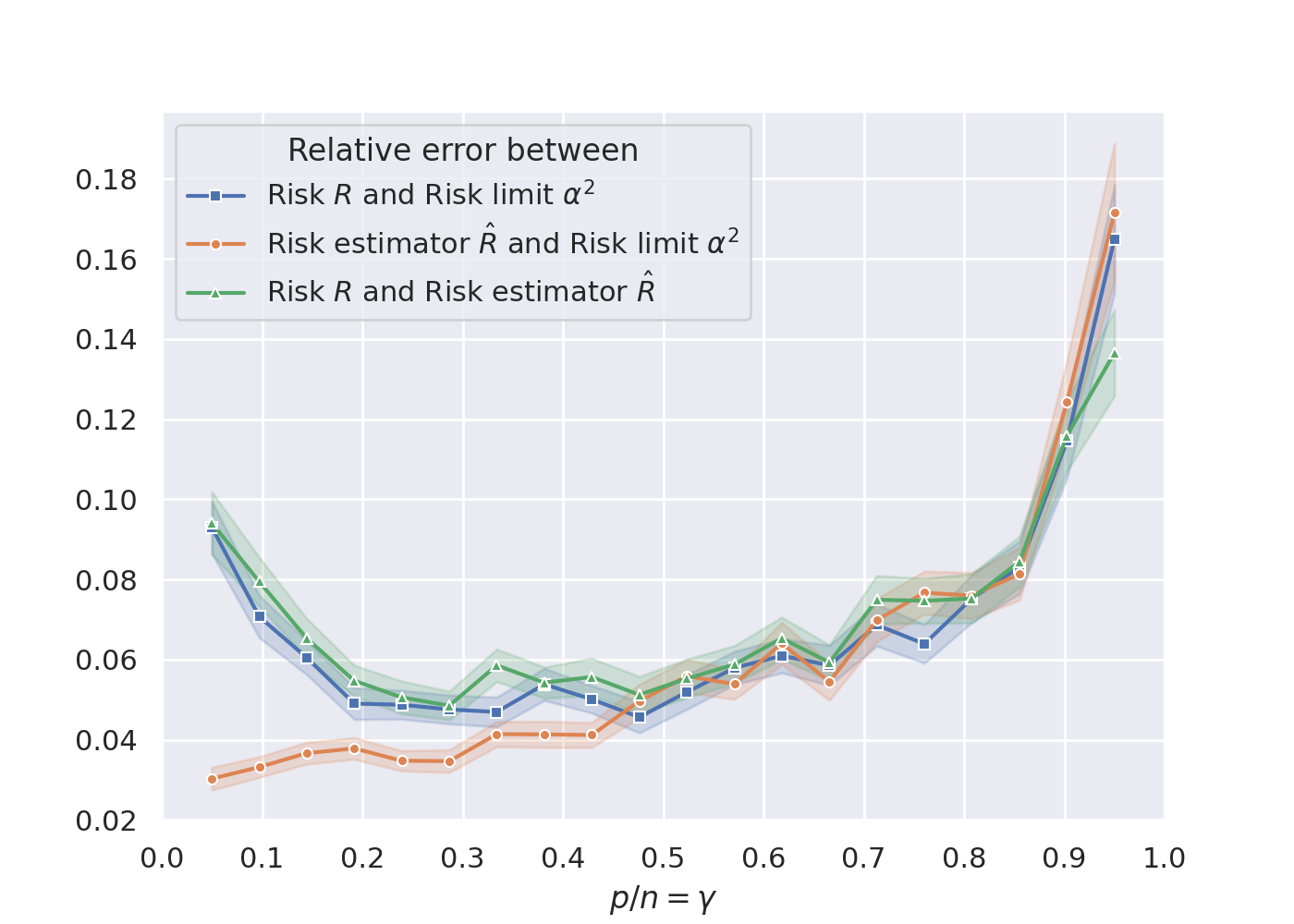}
        % \caption{T-dist with $\text{df}=4$}
    \end{subfigure}
    \caption{
        \editline{
        We plot the risk and its estimator for different ratio $\gamma= p/n$. 
       Simulation parameters are as follows: sample size $n=4000$, noise distribution $F_\epsilon=\text{t-dist}(\text{df}=2)$, scaling parameter $\lambda=1$, with $100$ repetition. The relative error between $A$ and $B$ is $|A/B-1|$. 
        }
    }
    \label{fig:change_gamma}
\end{figure}
}

\editline{
\subsection{Convergence of risk estimation variance with increasing sample size}\label{subsec:increase_n}
We plot the risk and its estimator as $n$  varies for different values of  $\gamma = p/n$ :  $\gamma = 0.25$  in \Cref{fig:change_n_gamma025} and  $\gamma = 0.5$  in \Cref{fig:change_n_gamma05}. In both cases, the variances of the risk and risk estimator decrease as  $n \to +\infty$. Interestingly, these figures suggest that the variances decays in $\sqrt{n}$-rate.

\begin{figure}[htpb]
    \centering
    \begin{subfigure}[b]{0.49\textwidth}
        \centering
        \includegraphics[width=\textwidth]{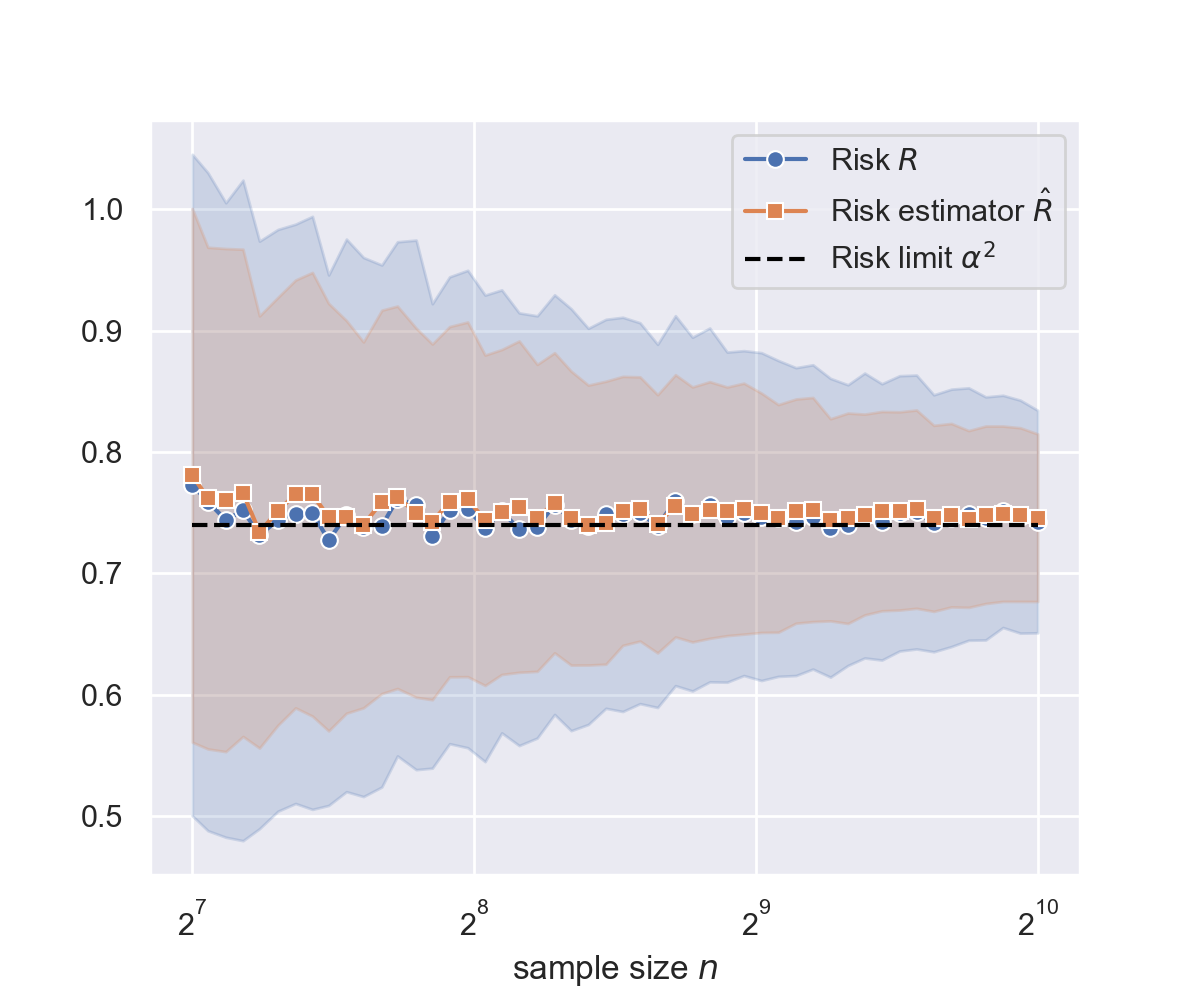}
    \end{subfigure}
    \begin{subfigure}[b]{0.49\textwidth}
        \centering
        \includegraphics[width=\textwidth]{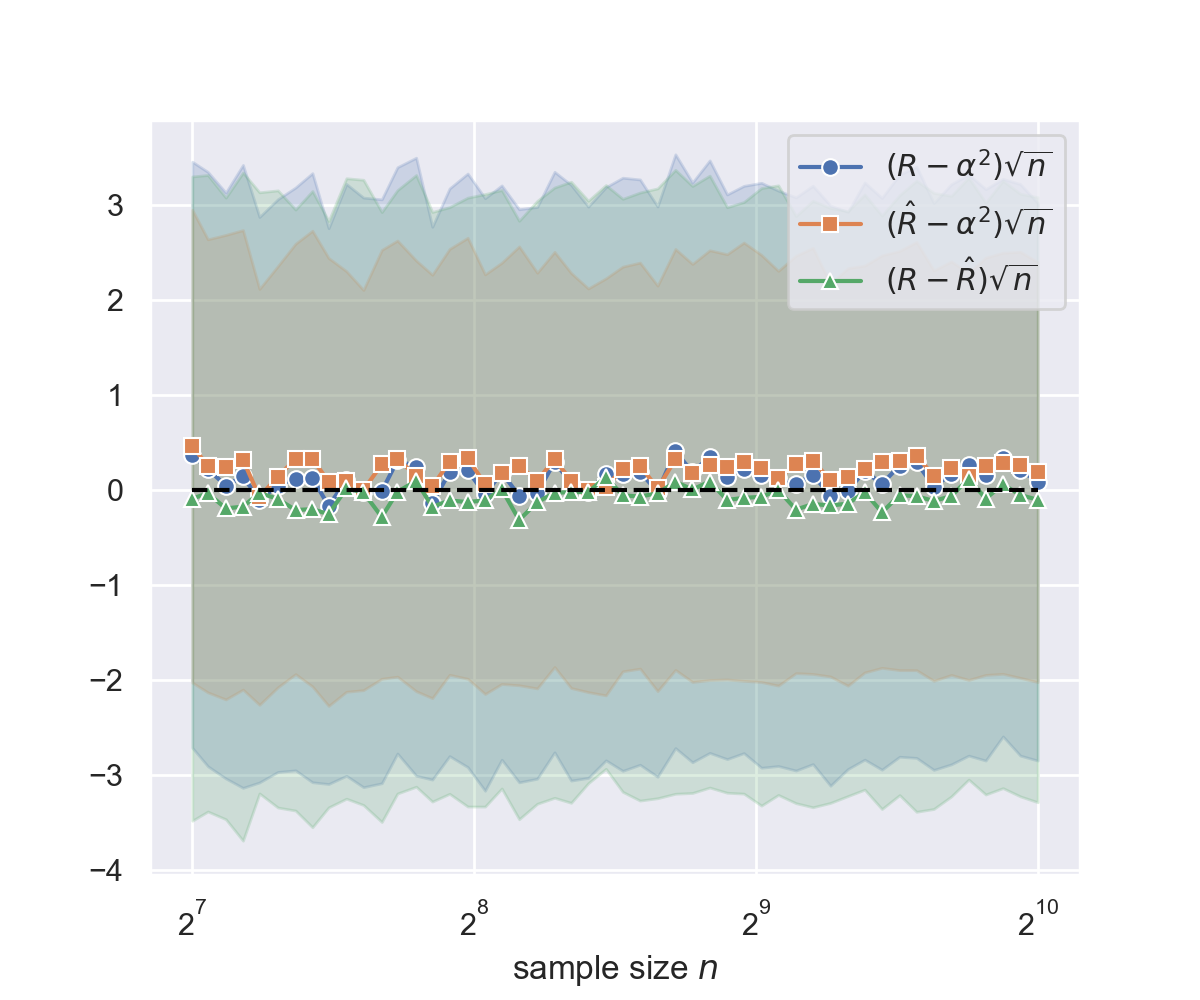}
    \end{subfigure}
    \caption{
        \editline{We plot the risk and its estimator for different $n$. 
    Simulation parameters are as follows: $p/n = \gamma=0.25$, $F_\epsilon=\text{t-dist}(\text{df}=2)$, scaling parameter $\lambda=1$, with $1000$ repetition.}}
    \label{fig:change_n_gamma025}
\end{figure}

\begin{figure}[htpb]
    \centering
    \begin{subfigure}[b]{0.49\textwidth}
        \centering
        \includegraphics[width=\textwidth]{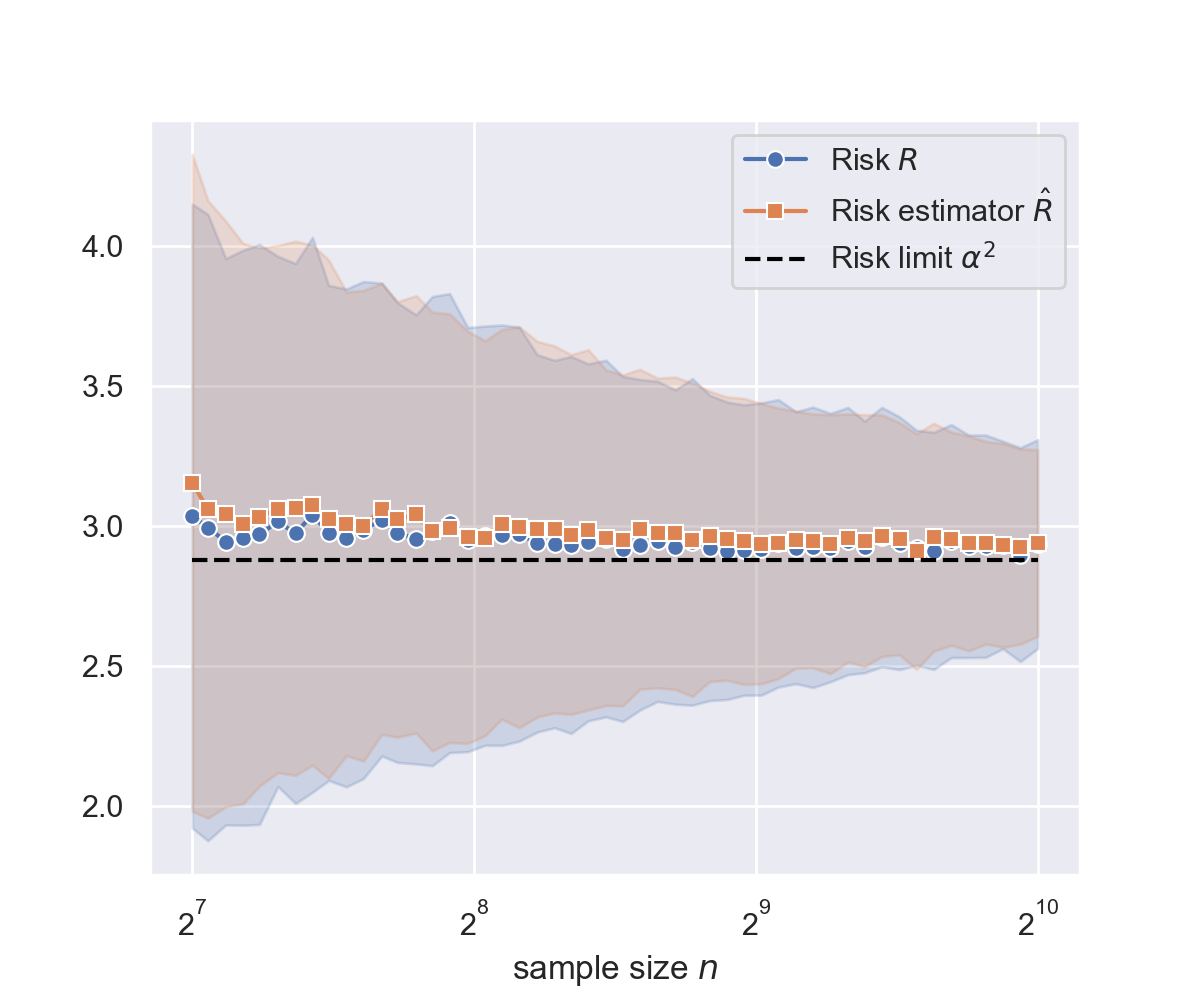}
    \end{subfigure}
    \begin{subfigure}[b]{0.49\textwidth}
        \centering
        \includegraphics[width=\textwidth]{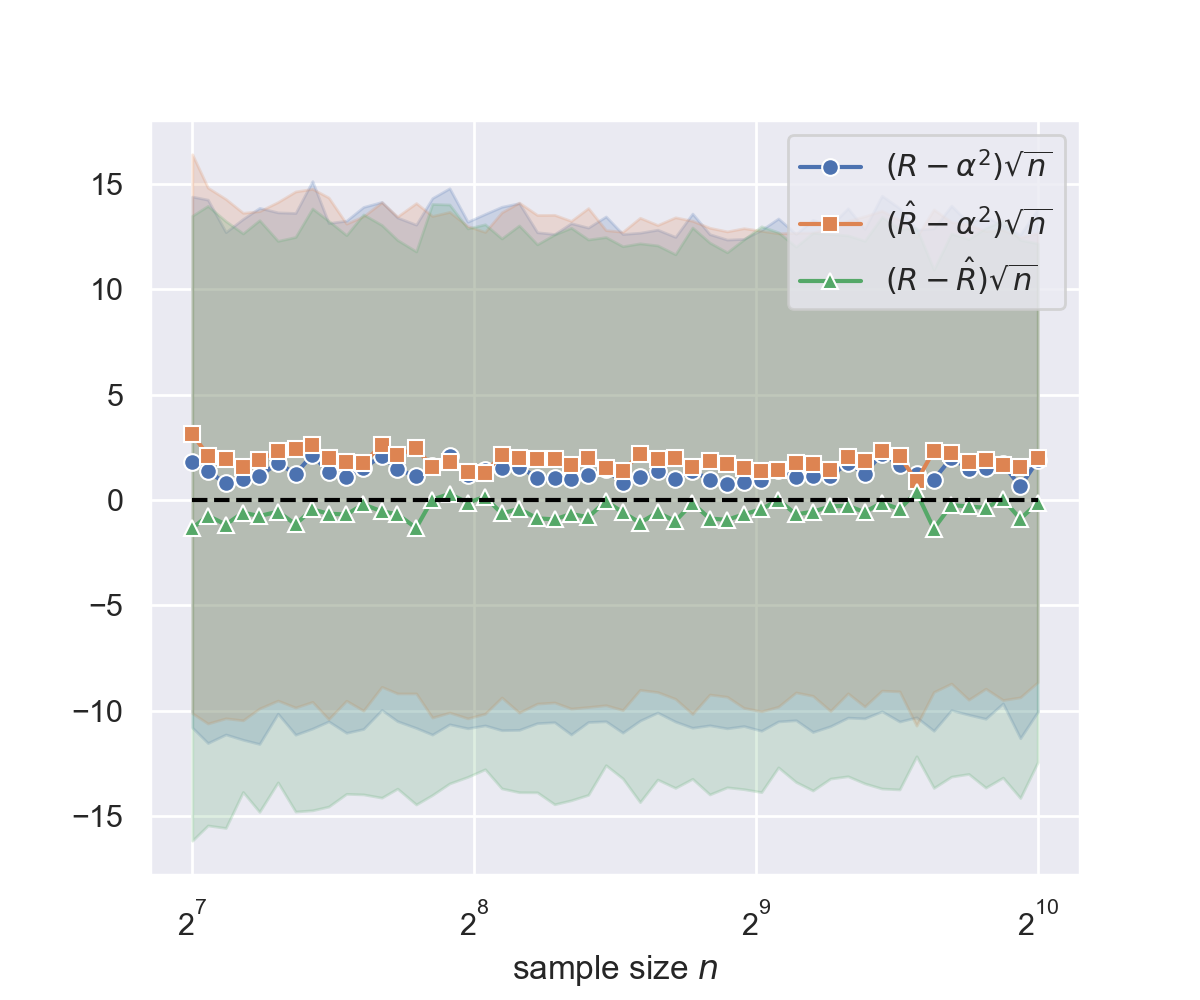}
    \end{subfigure}
    \caption{
    \editline{    
    We plot the risk and its estimator for different $n$. 
    Simulation parameters are as follows: $p/n = \gamma=0.5$, $F_\epsilon= \text{t-dist}(\text{df}=2)$, scaling parameter $\lambda=1$, with $1000$ repetition.}}
    \label{fig:change_n_gamma05}
\end{figure}
}

\end{document}